\documentclass[preprint,3p]{elsarticle}

\usepackage{amssymb,amsmath,amsthm,amsfonts}
\usepackage{comment,url,enumerate}
\usepackage{tikz,bussproofs}
\usepackage{framed}		% for Gentzen conditions
\usepackage{mathtools}	% for nicer overbraces
\usepackage{hyperref}

\usetikzlibrary{arrows,shapes}

\newlength{\auxlength}
\newlength{\auxlengthtwo}
\newlength{\auxlengththree}

\newcommand{\alg}{\mathbf}
\newcommand{\class}{\mathsf}

\newcommand{\elem}{\mathsf}

\newcommand{\set}[2]{\{ #1 \mid #2 \}}
\newcommand{\pair}[2]{\langle #1, #2 \rangle}

\newcommand{\equals}{\approx}
\newcommand{\inequals}{\leq}
\newcommand{\assign}{\mathrel{:=}}
\newcommand{\dual}{\partial}

\newcommand{\bs}{\backslash}
\newcommand{\sqle}{\sqsubset}
\newcommand{\sqge}{\sqsupset}
\newcommand{\sqleq}{\sqsubseteq}
\newcommand{\sqgeq}{\sqsupseteq}
\newcommand{\into}{\hookrightarrow}
\newcommand{\Iso}{\mathbb{I}}

\newcommand{\fracof}[1]{#1^{\div}}
\newcommand{\fracalg}[1]{\alg{#1}^{\div}}
\newcommand{\cdmalg}[1]{\alg{#1}^{\!\Delta}}
\newcommand{\cdmhom}[1]{#1^{\!\Delta}}
\newcommand{\extalg}[1]{\overline{\alg{#1}}}
\newcommand{\fracsigma}{\sigma^{\div}}
\newcommand{\fraciota}{\iota^{\div}}
\newcommand{\fracclass}[1]{\class{#1}^{\div}}
\newcommand{\frachom}[1]{#1^{\div}}
\newcommand{\prefracalg}[1]{\fracalg{#1}_{\mathrm{pre}}}

\newcommand{\algAsquared}{\alg{A}^{\!2}}
\newcommand{\fracalgA}{\alg{A}^{\!\div}}
\newcommand{\fracalgB}{\fracalg{B}}
\newcommand{\fracalgN}{\fracalg{\N}}
\newcommand{\fracclassK}{\fracclass{K}}

\newcommand{\multipair}[2]{\pair{#1}{#2}_{\scriptscriptstyle\bullet}}

\newcommand{\transplus}{\alpha}
\newcommand{\transminus}{\beta}

\newcommand{\comp}[1]{\overline{#1}}
\newcommand{\compl}[1]{#1^{\ell}}
\newcommand{\compr}[1]{#1^{r}}
\newcommand{\complr}[1]{#1^{\ell r}}
\newcommand{\comprl}[1]{#1^{r \ell}}

\newcommand{\ltor}{\ltorleft}
\newcommand{\rtol}{\rtolleft}
\newcommand{\ltorleft}{\ell_{\scriptscriptstyle\circ}}
\newcommand{\ltorright}{r_{\scriptscriptstyle\circ}}
\newcommand{\rtolleft}{\ell_{\scriptscriptstyle\oplus}}
\newcommand{\rtolright}{r_{\scriptscriptstyle\oplus}}

\newcommand{\leftpair}{\multipair}
\newcommand{\rightpair}[2]{\pair{#1}{#2}_{\!\scriptscriptstyle\oplus}}

\newcommand{\leftdot}{\circ}
\newcommand{\rightplus}{\oplus}
\newcommand{\leftunit}{1_{\scriptscriptstyle\circ}}
\newcommand{\rightzero}{0_{\scriptscriptstyle\oplus}}

\newcommand{\invframe}[1]{#1}
\newcommand{\invframeofalg}[1]{F_{\alg{#1}}}
\newcommand{\galois}[1]{#1^{+}}

\newcommand{\toright}{^{\vartriangleright}}
\newcommand{\toleft}{^{\vartriangleleft}}
\newcommand{\torighttoleft}{^{\vartriangleright\vartriangleleft}}
\newcommand{\tolefttoright}{^{\vartriangleleft\vartriangleright}}

\newcommand{\N}{\mathbb{N}}
\newcommand{\Z}{\mathbb{Z}}
\newcommand{\Zplus}{\mathbb{Z}_{+}}
\newcommand{\ZlexZ}{\Z \overrightarrow{\times} \Z}

\newcommand{\elembot}{\boldsymbol{\bot}}
\newcommand{\elemone}{\boldsymbol{1}}

\newcommand{\Lukthree}{\text{\bf\L}_{\alg{3}}}
\newcommand{\Lukthreeext}{\Lukthree^{\!\Delta}}

\newcommand{\Mthree}{\alg{M}_{\alg{3}}}
\newcommand{\Nfive}{\alg{N}_{\alg{5}}}

\newcommand{\Hfive}{\alg{H}_{\alg{5}}}
\newcommand{\fracHfive}{\fracalg{H}_{\alg{5}}}
\newcommand{\Hfivesquared}{\Hfive^{2}}

\newcommand{\topmirroralgA}{\alg{A}^{\!\ast}}
\newcommand{\botmirroralgA}{\alg{A}_{\ast}}

\newcommand{\preleq}{\preccurlyeq}
\newcommand{\precdot}{\leftdot}
\newcommand{\preplus}{\rightplus}
\newcommand{\preone}{1^{\div}}
\newcommand{\prezero}{0^{\div}}

\newcommand{\twoalg}{\alg{2}}
\newcommand{\omegaplusonealg}{\boldsymbol{\omega} + \boldsymbol{1}}

\newcommand{\thistheoremref}{}

\newtheoremstyle{plainnewline}
{}{}
{\itshape}{}
{\bfseries}{.}
{\newline}{\thmname{#1}\thmnumber{ #2}\thistheoremref\thmnote{ (#3)}}

\theoremstyle{plainnewline}
\newtheorem{theorem}{Theorem}[section]
\newtheorem{lemma}[theorem]{Lemma}
\newtheorem{proposition}[theorem]{Proposition}
\newtheorem{definition}[theorem]{Definition}
\newtheorem{corollary}[theorem]{Corollary}

\theoremstyle{plain}
\newtheorem{fact}[theorem]{Fact}
\newtheorem{lemmanoname}[theorem]{Lemma}
\newtheorem{corollarynoname}[theorem]{Corollary}
\newtheorem{problem}{Problem}

\begin{document}

\begin{frontmatter}

\author[1]{Nick Galatos}
\ead[1]{ngalatos@du.edu}
\address[1]{Department of Mathematics, University of Denver, Denver, CO, USA}

\author[2]{Adam \texorpdfstring{P\v{r}enosil}{Prenosil}}
\ead[2]{adam.prenosil@vanderbilt.edu}
\address[2]{Department of Mathematics, Vanderbilt University, Nashville, TN, USA}

\title{Complemented MacNeille completions and algebras of fractions}

\begin{abstract}
  We introduce ($\ell$-)bimonoids as ordered algebras consisting of two compatible monoidal structures on a partially ordered (lattice-ordered) set. Bimonoids form an appropriate framework for the study of a general notion of complementation, which subsumes both Boolean complements in bounded distributive lattices and multiplicative inverses in monoids. The central question of the paper is whether and how bimonoids can be embedded into complemented bimonoids, generalizing the embedding of cancellative commutative monoids into their groups of fractions and of bounded distributive lattices into their free Boolean extensions. We prove that each commutative ($\ell$-)bimonoid indeed embeds into a complete complemented commutative $\ell$-bimonoid in a doubly dense way reminiscent of the Dedekind--MacNeille completion. Moreover, this complemented completion, which is term equivalent to a commutative involutive residuated lattice, sometimes contains a tighter complemented envelope analogous to the group of fractions. In the case of cancellative commutative monoids this algebra of fractions is precisely the familiar group of fractions, while in the case of Brouwerian (Heyting) algebras it is a (bounded) idempotent involutive commutative residuated lattice. This construction of the algebra of fractions in fact yields a categorical equivalence between varieties of integral and of involutive residuated structures which subsumes as special cases the known equivalences between Abelian $\ell$-groups and their negative cones, and between Sugihara monoids and their negative cones.
\end{abstract}


\end{frontmatter}

\section{Introduction}

  In various areas of mathematics, ordered algebraic structures naturally arise in which a partially ordered monoid (or \emph{pomonoid}) $\langle A, \leq, \cdot, 1 \rangle$ is equipped with an order-inverting involution $x \mapsto \comp{x}$. This gives rise to another partially ordered monoid $\langle A, \geq, +, 0 \rangle$ over the same set, where $0 = \comp{1}$ and the two monoidal operations are related by De~Morgan duality:
\begin{align*}
  \comp{x \cdot y} & = \comp{y} + \comp{x} & & \text{and} & \comp{x + y} & = \comp{y} \cdot \comp{x}.
\end{align*}
  Moreover, the order-inverting involution often relates the two monoids in a way that is reminiscent of Boolean algebras. Namely, in Boolean algebras, where we take the two monoids to be the meet and join semilattices, we have
\begin{align*}
  x \leq z \vee \neg y \iff x \wedge y \leq z \iff y \leq \neg x \vee z.
\end{align*}
  Similarly, in partially ordered groups, where the two monoids coincide, we have
\begin{align*}
  x \leq z \cdot y^{-1} \iff x \cdot y \leq z \iff y \leq x^{-1} \cdot z.
\end{align*}
  Finally, in MV-algebras we have
\begin{align*}
  x \leq z \oplus \neg y \iff x \odot y \leq z \iff y \leq \neg x \oplus z.
\end{align*}
  Groups can and will be be identified with partially ordered groups whose partial order is the equality relation.

  The notion of an \emph{involutive residuated pomonoid} was devised to subsume the above examples and more. Involutive residuated pomonoids can thus be thought of as ordered algebras of the form $\langle A, \leq, \cdot, 1, +, 0, \comp{\phantom{0}}\,\rangle$ where each element has a Boolean-like or group-like complement and the two monoids are De~Morgan duals of each other. \emph{Involutive residuated lattices} moreover require that the partial order be a lattice.

  This paper is devoted to the study of the \emph{positive subreducts} of commutative involutive residuated pomonoids. These are defined as the ordered algebras of the form $\langle A, \leq, \cdot, 1, +, 0 \rangle$ which can be embedded into some commutative involutive residuated pomonoid. More generally, we consider the following problem:
\begin{center}
  Given a class $\class{K}$ of involutive residuated pomonoids, \\ describe the positive subreducts of the structures in $\class{K}$.
\end{center}
  We also consider the positive subreducts of commutative involutive residuated lattices, i.e.\ algebras of the form $\langle A, \vee, \wedge, \cdot, 1 ,+, 0 \rangle$ which can be embedded into some commutative involutive residuated lattice. Although some of our definitions and constructions allow for non-commutative bimonoids, throughout the paper we focus on the commutative case.\footnote{Beyond the commutative case this problem is highly non-trivial already for groups. Although monoids which embed into a group of (right) fractions are precisely the so-called left-reversible cancellative monoids, the~description of monoids which embed into \emph{some} group, not necessarily a group of fractions, is substantially more complicated (see~\cite[Ch.\,12]{clifford+preston67}).}

  There are in fact at least two classical results of this form: each cancellative commutative monoid can be embedded into an Abelian group (the smallest such group is the \emph{group of fractions} of the monoid) and each bounded distributive lattice can be embedded into a Boolean algebra (the smallest such algebra is the \emph{free Boolean extension} of the bounded distributive lattice). Conversely, Abelian groups are precisely those cancellative commutative monoids where each element $x$ has a multiplicative inverse $x^{-1}$ and Boolean algebras are precisely those bounded distributive lattices where each element $x$ has a Boolean complement~$\neg x$. These results can be expressed more compactly as follows: bounded distributive lattices are precisely the positive subreducts of Boolean algebras and cancellative commutative monoids are precisely the positive subreducts of Abelian groups.

  One of our goals is to unify and extend these two constructions, bringing out both their common features and their essential differences. Our first main result describes the positive subreducts of commutative involutive residuated pomonoids. We identify these subreducts as \emph{commutative bimonoids}. These are pairs of commutative pomonoids over the same partially ordered set satisfying the compatibility condition
\begin{align*}
  a \cdot (b + c) & \leq (a \cdot b) + c,
\end{align*}
  which we call \emph{hemidistributivity}.\footnote{The term ``hemidistributivity'' was first introduced for this condition by Dunn \& Hardegree~\cite{dunn+hardegree01algebraic-methods}. The term ``bimonoid'' is also used to denote a compatible pair of a monoid and a comonoid over a symmetric monoidal category (see~\cite{porst08}). Our usage of the term is essentially unrelated.} Two elements $a$, $b$ of a commutative bimonoid are called \emph{complements}~if
\begin{align*}
  a \cdot b & \leq 0 & & \text{and} & 1 & \leq a + b.
\end{align*}
  This notion subsumes both Boolean complements in bounded distributive lattices and multiplicative inverses in monoids. Complemented commutative bimonoids are in fact term equivalent to commutative involutive residuated pomonoids. The~problem of describing the positive subreducts of involutive residuated structures therefore reduces to the following problem:
\begin{align*}
  \text{Given a bimonoid, embed it into a complemented bimonoid.}
\end{align*}
  We show that this can always be done in the commutative case.

  Specifically, we prove in Section~\ref{sec: macneille} that each commutative bimonoid $\alg{A}$ embeds into what we call its \emph{complemented Dedekind--MacNeille completion}~$\cdmalg{A}$. This is a complete commutative involutive residuated lattice (unique up to a unique isomorphism which fixes $\alg{A}$) containing $\alg{A}$ as a sub-bimonoid where each element is a join of the form $\bigvee_{i \in I} (a_{i} \cdot \comp{b}_{i})$, or equivalently a meet of the form $\bigwedge_{i \in I} (a_{i} + \comp{b}_{i})$, for $a_{i}, b_{i} \in \alg{A}$. The construction of the complemented Dedekind--MacNeille completion $\cdmalg{A}$ relies heavily on the machinery of involutive residuated frames developed by Galatos \& Jipsen \cite{galatos+jipsen13residuated-frames}.

  If $\alg{A}$ is moreover a commutative $\ell$-bimonoid, i.e.\ a commutative lattice-ordered bimonoid which satisfies the equations $x \cdot (y \vee z) \equals (x \cdot y) \vee (x \cdot z)$ and $x + (y \wedge z) \equals (x + y) \wedge (x + z)$, then the embedding of $\alg{A}$ into $\cdmalg{A}$ preserves finite meets and joins. Commutative $\ell$-bimonoids are thus the positive subreducts of commutative involutive residuated lattices. Moreover, in Section~\ref{sec: subreducts} we show how to axiomatize the positive subreducts of any variety of commutative involutive residuated lattices defined by equations in the signature $\{ \vee, \cdot, 1 \}$.

  The above construction provides a commutative complemented envelope for each commutative bimonoid. This is a satisfactory result, but observe that not all complemented extensions are created equal: a generic element of a free Boolean extension has the form $\bigvee_{i \in I} (a_{i} \wedge \neg b_{i})$, while a generic element of a group of fractions has the simpler form $a \cdot b^{-1}$. Let us introduce a name for complemented extensions of this second, simpler kind: we shall say that a commutative involutive residuated pomonoid $\alg{B}$ which contains a commutative bimonoid $\alg{A}$ as a sub-bimonoid is a \emph{commutative complemented bimonoid of fractions} of $\alg{A}$ if each element $x$ of $\alg{B}$ has the form $x = a \cdot \comp{b}$, or equivalently $x = a + \comp{b}$, for some $a, b \in \alg{A}$. If~a~commutative complemented bimonoid of fractions exists, it is unique up to isomorphism and it embeds into $\cdmalg{A}$. We~denote it $\fracalgA$.

  Although in the present paper we restrict our attention to complemented Dedekind--MacNeille completions and their sub\-algebras, this is not the only possible direction in which one can look for generalizations of free Boolean extensions. In particular, the problem of axiomatizing the $\ell$-bimonoidal subreducts of MV-algebras was recently solved by Abbadini et al.~\cite{abbadini+jipsen+kroupa+vannucci22} using a different kind of complemented envelope, where each element is a finite sum rather than a possibly infinite join of elements of the form $a \cdot \comp{b}$.

  The second problem that we consider is then the question of which bimonoids admit such well-behaved complemented extensions. That is:
\begin{center}
  Which bimonoids have a complemented bimonoid of fractions?
\end{center}
  We only consider this problem for commutative complemented bimonoids of fractions of commutative bimonoids. As we already observed above, cancellative commutative monoids admit a bimonoid of fractions, while distributive lattices typically do not. We show in Section~\ref{sec: bimonoids of fractions} that commutative bimonoids which admit a commutative complemented bimonoid of fractions can be described by a certain first-order sentence. Moreover, for residuated bimonoids (bimonoids whose multi\-plicative pomonoid is residuated) the existence of a complemented bimonoid of fractions may be witnessed by a pair of terms satisfying certain inequalities. With the help of these terms we can construct the complemented bimonoid of fractions explicitly.

  The construction of the complemented bimonoid of fractions covers some known cases as well as some new ones. In particular, we prove that all Brouwerian algebras, hence also all Heyting algebras, admit a complemented bimonoid of fractions, which is an idempotent involutive residuated lattice. This~extends to the non-semilinear case of a result of Galatos \& Raftery~\cite{galatos+raftery12} for so-called relative Stone algebras, whose algebras of fractions are known as Sugihara monoids. More generally, complemented bimonoids of fractions can be constructed for Brouwerian semilattices and for Boolean-pointed Brouwerian algebras. The latter construction extends a result of Fussner~\&~Galatos~\cite{fussner+galatos19} for Boolean-pointed relative Stone algebras. The two reflection constructions of Galatos \& Raftery~\cite{galatos+raftery04} are also covered (if we allow for bimonoids of fractions of bisemigroups), as is of course the classical construction of the Abelian group of fractions of a cancellative commutative monoid.\footnote{Let us recall the definitions of these algebras. Brouwerian algebras differ from Heyting algebras by removing the assumption that a bottom element exists. Relative Stone algebras are Brouwerian algebras which satisfy the equation ${(x \rightarrow y) \vee (y \rightarrow x) \equals 1}$; equivalently, they are subdirect products of Brouwerian chains. Relative Stone algebras with a bottom element are better known as G\"{o}del algebras. Boolean-pointed Brouwerian algebras are Brouwerian algebras equipped with a constant $0$ such that the interval $[0, 1]$ is a Boolean lattice. Brouwerian semilattices are unital meet semilattices equipped with relative pseudocomplementation ($x \rightarrow y$). Finally, Sugihara monoids are distributive idempotent commutative involutive residuated lattices.}
 
  The importance of considering bimonoidal structure, instead of merely residuated lattice structure, when constructing complemented extensions can be illustrated on Heyting algebras. These can be seen as bimonoids in two different ways: if we take $0 = \bot$ and $x + y = x \vee y$, the smallest complemented extension is the free Boolean extension, whereas if we take $0 = \top$ and $x + y = x \wedge y$, we obtain a (non-integral) idempotent involutive residuated lattice. This latter algebra is what we referred to above as the bimonoid of fractions of the Heyting algebra.

  Borrowing an idea used by Montagna \& Tsinakis~\cite{montagna+tsinakis10} in the context of groups of fractions, the bimonoid~$\alg{A}$ may be identified inside of its complemented bimonoid of fractions $\fracalgA$ as the image of a certain interior operator $\fracsigma$, provided that multiplication in $\alg{A}$ is residuated. This sometimes allows us to extend the construction of the complemented bimonoid of fractions to a categorical equivalence. In~one direction, we take a suitable residuated bimonoid $\alg{A}$ to a complemented bimonoid $\fracalgA$ equipped with an interior operator~$\fracsigma$, while in the other direction we take a suitable complemented bimonoid $\alg{B}$ equipped with an interior operator~$\sigma$ to the image of this operator $\alg{B}_{\sigma}$, which is a sub-bimonoid of $\alg{B}$. This yields a new categorical equivalence subsuming several known equivalences: the restriction to the commutative case of the equivalence between lattice-ordered groups ($\ell$-groups) and certain integral residuated lattices due to Bahls~et~al.~\cite{bahls+et+al03}, the restriction to the commutative case of the equivalence between certain $\ell$-groups with a conucleus and certain cancellative residuated lattices due to Montagna \& Tsinakis~\cite{montagna+tsinakis10}, the equivalence between odd Sugihara monoids and relative Stone algebras due to Galatos~\&~Raftery~\cite{galatos+raftery12}, and its extension to Sugihara monoids and Boolean-pointed relative Stone algebras due to Fussner~\&~Galatos~\cite{fussner+galatos19}.

  Let us now outline the structure of the paper. In Section~\ref{sec: bimonoids} we review some basic terminology concerning involutive residuated structures and introduce bimonoids, lattice-ordered bimonoids ($\ell$-bimonoids), and related structures as an attempt to describe the positive subreducts of involutive residuated structures. We define complements in bimonoids and show that involutive residuated pomonoids (involutive residuated lattices) are term equivalent to bimonoids ($\ell$-bimonoids) with complementation. We then provide some examples of bimonoids. In Section~\ref{sec: macneille} we prove the first main result of the paper, namely the existence of commutative complemented Dedekind--MacNeille completions of commutative bimonoids. The existence of commutative complemented bimonoids of fractions is then studied in Section~\ref{sec: bimonoids of fractions}. This in particular yields new categorical equivalences between varieties of residuated structures, as well as uniform proofs of some known equivalences. In Section~\ref{sec: subreducts} we study the preservation of equations by the two constructions (complemented Dedekind--MacNeille completions and complemented bimonoids of fractions) and axiomatize the $\ell$-bimonoidal subreducts of each variety of commutative involutive residuated lattices axiomatized by equations in the signature $\{ \vee, \cdot, 1 \}$.  Finally, in Section~\ref{sec: open problems} we list some problems which the paper leaves open.

\section{Bimonoids}
\label{sec: bimonoids}

  The main algebraic structures studied in the present paper, namely bimonoids and their lattice-ordered and residuated variants, are introduced in this section. We define complements in bimonoids and show that complemented bimonoids are term equivalent to involutive residuated pomonoids. We then introduce some basic terminology and provide examples of how bimonoids can be constructed from partially ordered monoids. Familiarity with the basic notions of universal algebra is assumed (see e.g.\ \cite{burris+sankappanavar81,alv87}).

\subsection{Involutive residuated structures}

  The structures that we study in this paper all have at least a partially ordered semigroup reduct.

\begin{definition}[Partially ordered semigroups and s$\ell$-semigroups]
  A \emph{partially ordered semigroup} or \emph{posemigroup} $\langle A, \leq, \cdot \rangle$ is a semigroup $\langle A, \cdot \rangle$ equipped with a partial order $\langle A, \leq \rangle$ which satisfies the implications
\begin{align*}
  x \inequals y \implies x \cdot z \inequals y \cdot z \qquad \text{and} \qquad x \inequals y \implies z \cdot x \inequals z \cdot y.
\end{align*}
  A \emph{semilattice-ordered semigroup} or \emph{s$\ell$-semigroup} $\langle A, \vee, \cdot \rangle$ is a semigroup $\langle A, \cdot \rangle$ equipped with a (join) semilattice structure $\langle A, \vee \rangle$ which satisfies the equations
\begin{align*}
  x \cdot (y \vee z) \equals (x \cdot y) \vee (x \cdot z) \qquad \text{and} \qquad (x \vee y) \cdot z \equals (x \cdot z) \vee (y \cdot z).
\end{align*}
  A \emph{partially ordered monoid} or \emph{pomonoid} $\langle A, \leq, \cdot, 1 \rangle$ is a posemigroup ${\langle A, \leq, \cdot \rangle}$ with a multiplicative unit~$1$. A~\emph{semilattice-ordered monoid} or \emph{s$\ell$-monoid} $\langle A, \vee, \cdot, 1 \rangle$, also known as an \emph{idempotent semiring}, is an s$\ell$-semigroup $\langle A, \vee, \cdot \rangle$ with a multiplicative unit~$1$.
\end{definition}

  The classes of s$\ell$-semigroups and s$\ell$-monoids are varieties of algebras, while the classes of posemigroups and pomonoids are varieties of ordered algebras in the sense of Pigozzi~\cite{pigozzi03partially-ordered-varieties}. Each s$\ell$-semigroup can of course be seen as a posemigroup. Several important classes of posemigroups can be defined by inequalities.

\begin{definition}[Commutative, integral, and idempotent posemigroups]
  A posemigroup is \emph{commutative} if it satisfies $x \cdot y \equals y \cdot x$, it is \emph{integral} if it satisfies the inequalities $x \cdot y \inequals x$ and $x \cdot y \inequals y$, and it is idempotent if it satisfies $x \cdot x \equals x$.
\end{definition}

  A pomonoid is integral if and only if $x \leq 1$ for each $x$. In addition to a partial order and a semigroup structure, the algebras that we study in the present paper are often also equipped with division-like operations called the left and right residual.

\begin{definition}[Residuated posemigroups and s$\ell$-semigroups]
  A binary operation $x \bs y$ on a posemigroup is called \emph{left division} (or the \emph{right residual} of multiplication) if it satisfies
\begin{align*}
  x \cdot y \inequals z \iff y \inequals x \bs z.
\end{align*}
  A binary operation $x / y$ on a posemigroup is called \emph{right division} (or the \emph{left residual} of multiplication) if it satisfies
\begin{align*}
  x \cdot y \inequals z \iff x \inequals z / y.
\end{align*}
  A \emph{residuated posemigroup} $\langle A, \leq, \cdot, \bs, / \rangle$ is a posemigroup $\langle A, \leq, \cdot \rangle$ equipped with the two binary operations $x \bs y$ and $x / y$ of left and right division. A \emph{residuated s$\ell$-semigroup} $\langle A, \vee, \cdot, \bs, / \rangle$ is an s$\ell$-semigroup $\langle A, \vee, \cdot \rangle$ which is also a residuated posemigroup $\langle A, \leq, \cdot, \bs, / \rangle$ with respect to the semilattice order. A \emph{residuated pomonoid (s$\ell$-monoid)} is a residuated posemigroup (s$\ell$-semigroup) equipped with a multiplicative unit $1$.
\end{definition}

  Residuated posemigroups (pomonoids) form a variety of ordered algebras, while residuated s$\ell$-semigroups (s$\ell$-monoids) form an ordinary variety of algebras. In commutative residuated posemigroups $x \bs y = y / x$. In~that case we simplify the algebraic signature and use the notation $x \rightarrow y$ for the residual $x \bs y = y / x$.

\begin{definition}[Admissible joins] \label{def: admissible join}
  Let $\alg{A}$ be a posemigroup and $X \subseteq \alg{A}$. The join $\bigvee X$, if it exists, is called \emph{admissible} if for each $y \in \alg{A}$
\begin{align*}
  (\bigvee X) \cdot y & = \bigvee \set{x \cdot y}{x \in X}, \\
  y \cdot (\bigvee X) & = \bigvee \set{y \cdot x}{x \in X}.
\end{align*}
\end{definition}

\begin{fact}
  All existing joins are admissible in a residuated posemigroup.
\end{fact}

  All finite non-empty joins are admissible in an s$\ell$-semigroup. The empty join, i.e.\ the bottom element~$\bot$, is admissible if and only if $\bot \cdot x = \bot = x \cdot \bot$ for each $x$.

  We shall mainly be interested in so-called involutive (residuated) posemigroups, where the binary division operations can be decomposed into a binary addition ($x + y$) and a unary complementation ($\compl{x}$ or $\compr{x}$). Examples of such algebras include $\ell$-groups, Boolean algebras, MV-algebras, Sugihara monoids, and relation algebras (see e.g.\ \cite[Section\,2.3]{glimpse07} for the definitions of these structures). Although these algebras only have one unary complementation operation ($x^{-1}$ in for $\ell$-groups, $\neg x$ for Boolean algebras and MV-algebras), in general one has to consider two distinct unary operations on a pointed residuated pomonoid, which we denote $\compl{x}$~and~$\compr{x}$. This notation was chosen to be consistent with the existing notation for pregroups~\cite{lambek99}.  

\begin{definition}[Involutive residuated posemigroups]
  An \emph{involutive residuated posemigroup} $\langle A, \leq, \cdot, \compl{}, \compr{} \rangle$ is a posemigroup $\langle A, \leq, \cdot \rangle$ with two antitone unary operations $\compl{x}$ and $\compr{x}$ satisfying $\complr{x} \equals x \equals \comprl{x}$ such that $\langle A, \leq, \cdot, \bs, / \rangle$ is a residuated posemigroup, where
\begin{align*}
  x \bs y & \assign \compr{(\compl{y} \cdot x)}, &
  y / x & \assign \compl{(x \cdot \compr{y})}.
\end{align*}
  An \emph{involutive residuated pomonoid} $\langle A, \leq, \cdot, 1, \compl{}, \compr{} \rangle$ is an involutive residuated posemigroup $\langle A, \leq, \cdot, \compl{}, \compr{} \rangle$ which is also a pomonoid $\langle A, \leq, \cdot, 1 \rangle$.
\end{definition}

  For the basic arithmetic of involutive residuated posemigroups see~\cite[Section~3.3]{glimpse07}. For pomonoids the antitonicity of the two operations $\compl{x}$ and $\compr{x}$ follows from the other conditions. In~involutive residuated posemigroups we may introduce another semigroup operation as the De~Morgan dual of multiplication:
\begin{align*}
  x + y & \assign \compr{(\compl{y} \cdot \compl{x})} = \compl{(\compr{y} \cdot \compr{x})}.
\end{align*}
  Using this operation we may express the two residuals as
\begin{align*}
  x \bs y & \assign \compr{x} + y, &
  y / x & \assign y + \compl{x}.
\end{align*}
  Up to term equivalence we can therefore view involutive residuated posemigroups as ordered algebras of the form $\langle A, \leq, \cdot, +, \compl{}, \compr{} \rangle$. This is indeed how we shall treat them in the rest of this paper. We generally prefer to use the notation $\compr{x} + y$ and $y + \compl{x}$ for $x \bs y$ and $y / x$ in involutive residuated posemigroups.

  In an involutive residuated pomonoid we may also introduce the constant
\begin{align*}
  0 \assign \compr{1} = \compl{1}.
\end{align*}
  We have $x + 0 = x = 0 + x$. Using this constant we may express the two antitone operations $\compl{x}$ and $\compr{x}$ as
\begin{align*}
  \compl{x} & \assign 0 / x & & \text{and} & \compr{x} & \assign x \bs 0.
\end{align*}
  The equations $\complr{x} \equals x \equals \comprl{x}$ then transform into
\begin{align*}
  (0 / x) \bs 0 \equals x \equals 0 / (x \bs 0).
\end{align*}
  Up to term equivalence we can therefore treat involutive residuated pomonoids as ordered algebras of the form $\langle A, \leq, \cdot, 1, +, 0, \compl{}, \compr{} \rangle$. Equivalently, involutive residuated pomonoids can be defined as pointed residuated pomonoids $\langle A, \leq, \cdot, 1, \bs, /, 0 \rangle$ which satisfy the equations $(0 / x) \bs 0 \equals x \equals 0 / (x \bs 0)$.

\begin{definition}[Involutive residuated lattices]
  An \emph{involutive residuated lattice} $\langle A, \vee, \wedge, \cdot, 1, +, 0, \compl{}, \compr{} \rangle$ is a lattice $\langle A, \vee, \wedge \rangle$ which is also an involutive residuated pomonoid $\langle A, \leq, \cdot, 1, +, 0, \compl{}, \compr{} \rangle$ with respect to the lattice order.
\end{definition}

\subsection{Bimonoids and \texorpdfstring{$\ell$-bimonoids}{l-bimonoids}: basic definitions}
\label{subsec: basic definitions}

  A natural problem is now to describe the \emph{positive subreducts} of involutive residuated structures, i.e.\ subreducts without the antitone operations $\compl{x}$ and $\compr{x}$. For example, what are the subreducts of the form $\langle A, \leq, \cdot, + \rangle$ or $\langle A, \leq, \cdot, 1, +, 0 \rangle$ of involutive residuated pomonoids or the subreducts of the form $\langle A, \vee, \wedge, \cdot, + \rangle$ or $\langle A, \vee, \wedge, \cdot, 1, +, 0 \rangle$ of involutive residuated lattices? To describe them, we introduce bisemigroups and $\ell$-bisemigroups and their unital counterparts, bimonoids and $\ell$-bimonoids.

\begin{definition}[Bisemigroups and bimonoids]
  A \emph{bisemigroup} $\alg{A} = \langle A, \leq, \cdot, + \rangle$ is a pair of posemigroups
\begin{align*}
  \alg{A}_{\circ} & = {\langle A, \leq, \cdot \rangle} & & \text{and} & \alg{A}_{+} & = {\langle A, \geq, + \rangle}
\end{align*}
  over dual orders, called the \emph{multi\-plicative} and the \emph{additive posemigroup} of $\alg{A}$, such that 
\begin{align*}
  x \cdot (y + z) \leq (x \cdot y) + z \qquad \text{and} \qquad (z + y) \cdot x \leq z + (y \cdot x).
\end{align*}
  A \emph{bimonoid} $\alg{A} = \langle A, \leq, \cdot, 1, +, 0 \rangle$ is a bisemigroup equipped with a multi\-plicative unit~$1$ and an additive unit~$0$. The \emph{multiplicative} and \emph{additive pomonoids} of $\alg{A}$ are then
\begin{align*}
  \alg{A}_{\circ} & = {\langle A, \leq, \cdot, 1 \rangle} & & \text{and} & \alg{A}_{+} & = {\langle A, \geq, +, 0 \rangle}.
\end{align*}
\end{definition}

  Although the notation is similar, the reader should not think of multiplication and addition in bisemigroups as analogues of the homonymous operations in rings or even semirings. Rather, one should think of splitting the single multiplication operation of partially ordered groups into a mirror pair of two operations.

\begin{definition}[Admissible joins and meets] \label{def: admissible join and meet}
  Let $\alg{A}$ be a posemigroup and $X \subseteq \alg{A}$. The join $\bigvee X$, if it exists, is called \emph{admissible} if it is admissible in the multiplicative posemigroup $\alg{A}_{\circ}$, i.e.\ if for each $y \in \alg{A}$
\begin{align*}
  (\bigvee X) \cdot y & = \bigvee \set{x \cdot y}{x \in X} & & \text{and} & y \cdot (\bigvee X) & = \bigvee \set{y \cdot x}{x \in X}.
\end{align*}
  The meet $\bigwedge X$, if it exists, is called \emph{admissible} if it is admissible in the additive posemigroup $\alg{A}_{+}$, i.e.\ if for each $y \in \alg{A}$
\begin{align*}
  (\bigwedge X) + y & = \bigwedge \set{x + y}{x \in X} & & \text{and} & y \cdot (\bigwedge X) & = \bigwedge \set{y + x}{x \in X}.
\end{align*}
\end{definition}

  A \emph{homomorphism} of bisemigroups (bimonoids) is an order-preserving homomorphism of multiplicative and additive semigroups (monoids). An \emph{embedding} of bisemigroups (bimonoids) is a homomorphism which is an order embedding. A \emph{complete embedding} moreover preserves all existing joins and meets.

\begin{definition}[Lattice-ordered bisemigroups and bimonoids]
  A \emph{lattice-ordered bisemigroup} or \emph{$\ell$-bisemigroup} $\alg{A} = \langle A, \vee, \wedge, \cdot, + \rangle$ is a pair of s$\ell$-semigroups $\alg{A}_{\circ} = \langle A, \vee, \cdot \rangle$ and $\alg{A}_{+} = \langle A, \wedge, + \rangle$, called respectively the \emph{multiplicative} and the \emph{additive s$\ell$-semigroup} of $\alg{A}$, such that $\langle A, \vee, \wedge \rangle$ is a lattice with the lattice order $\leq$ and $\langle A, \leq, \cdot, + \rangle$ is a bisemigroup.
  A \emph{lattice-ordered bimonoid} or \emph{$\ell$-bimonoid} $\alg{A} = \langle A, \vee, \wedge, \cdot, 1, +, 0 \rangle$ is a pair of s$\ell$-monoids $\alg{A}_{\circ} = \langle A, \vee, \cdot, 1 \rangle$ and $\alg{A}_{+} = \langle A, \wedge, +, 0 \rangle$, called respectively the \emph{multiplicative} and the \emph{additive s$\ell$-monoid} of $\alg{A}$, such that $\langle A, \vee, \wedge \rangle$ is a lattice with the lattice order $\leq$ and $\langle A, \leq, \cdot, 1, +, 0 \rangle$ is a bimonoid.
\end{definition}

  In other words, $\ell$-bisemigroups ($\ell$-bimonoids) form a variety of algebras axiomatized by the lattice axioms, two sets of semigroup (monoid) axioms, the hemidistributivity axioms
\begin{align*}
  x \cdot (y + z) & \leq (x \cdot y) + z, & (x + y) \cdot z & \leq x + (y \cdot z),
\end{align*}
  and the axioms stating that finite non-empty joins and meets are admissible:
\begin{align*}
  x \cdot (y \vee z) & \equals (x \cdot y) \vee (x \cdot z), & x + (y \wedge z) & \equals (x + y) \wedge (x + z), \\
  (x \vee y) \cdot z & \equals (x \cdot z) \vee (y \cdot z), & (x \wedge y) + z & \equals (x + z) \wedge (y + z).
\end{align*}
  Bisemigroups and bimonoids exhibit an order duality which extends the order duality on posets. Namely, the following two bimonoids are said to be \emph{order dual}:
\begin{align*}
  & \langle A, \leq, \cdot, 1, +, 0 \rangle & & \text{and} & \langle A, \geq, +, 0, \cdot, 1 \rangle.
\end{align*}
  Similarly, the following two $\ell$-bimonoids are said to be order dual:
\begin{align*}
  & \langle A, \vee, \wedge, \cdot, 1, +, 0 \rangle & & \text{and} & \langle A, \wedge, \vee, +, 0, \cdot, 1 \rangle.
\end{align*}
  This duality involves exchanging the roles of addition and multiplication as well as inverting the partial order. The order dual of an ($\ell$-)bimonoid $\alg{A}$ is an \mbox{($\ell$-)}bimonoid $\alg{A}^{\dual}$, thus an inequality holds in all ($\ell$-)bimonoids if and only if its naturally defined order dual~does. Like ordinary monoids, bimonoids and $\ell$-bimonoids also exhibit a \emph{left--right symmetry}, which consists in changing the operations $x \cdot y$ and $x + y$ to $y \cdot x$ and $y + x$. The above of course applies \emph{mutatis mutandis} to bisemigroups and $\ell$-bisemigroups too.

  We use both of the notations $x \cdot y$~and~$x y$ for multiplication. The notation $x \cdot y$ will be preferred when we wish to emphasize that multiplication and addition are on an equal footing, while the tighter notation $x y$ will be preferred when we wish to avoid writing too many parentheses. In particular, multiplication in such contexts is assumed to bind more tightly than other operations, e.g.\ we write $xy + z$ for $(x \cdot y) + z$. We~also use the notation
\begin{align*}
  x^n & = \overbrace{x \cdot \ldots \cdot x}^{n \text{ times}} & & \text{and} & n x & = \overbrace{x + \ldots + x}^{n \text{ times}}.
\end{align*}

  The compatibility condition between addition and multiplication, which in the case of bimonoids can also be written as
\begin{align*}
  x \cdot (y + z) \cdot w & \leq (x \cdot y) + (z \cdot w),
\end{align*}
  will be called \emph{hemidistributivity}. This term is due to Dunn \& Hardegree~\cite{dunn+hardegree01algebraic-methods}, who studied this compatibility condition as an algebraic formulation of the multiple-conclusion cut rule. The condition itself, however, is older. To~the best of our knowledge, the hemidistributive law was first explicitly written down in a paper of Meyer \& Routley \cite[p.\,236]{meyer+routley74e-is-a-conservative-extension}.\footnote{We thank James Raftery for bringing this to our attention.} It was also independently considered by Grishin~\cite{grishin83ajdukiewicz-lambek} in his study of symmetric variants of the Lambek calculus. In this context the condition is called \emph{mixed associativity}. A~categorical version of hemidistributivity was studied by Cockett \& Seely~\cite{cockett+seely97weakly-distributive-categories}, who called it \emph{weak distributivity}. Their weakly distributive categories in fact form the categorical counterpart of bimonoids, or conversely bimonoids are obtained from weakly distributive categories by restricting to partial orders.

  The hemidistributive laws can also be thought of as an ordered version of the \emph{inter\-associative laws} between two semigroup operations
\begin{align*}
  x \cdot (y + z) \equals (x \cdot y) + z \qquad \text{and} \qquad (y + z) \cdot x \equals y + (z \cdot x),
\end{align*}
  introduced under this name for arbitrary binary operations by Zupnik~\cite{zupnik71interassociativity} and later studied in the context of semigroups \cite{boyd+gould+nelson97interassociativity,gorbatkov13interassociativity}. Note, however, that inter\-associativity binds the two operations together very closely. If a multiplicative unit exists, addition can be defined in terms of multiplication as $x +_{a} y \assign x \cdot a \cdot y$ for some suitable $a$. This is because by interassociativity
\begin{align*}
  x + y = (x \cdot 1) + (1 \cdot y) = x \cdot (1 + 1) \cdot y = x +_{(1 + 1)} y.
\end{align*}
 Conversely, the operation $x +_{a} y$ is interassociative with multiplication for each~$a$. As we shall see, hemidistributivity allows for a much looser relationship between two monoidal operations than interassociativity.

  It is in particular important to note that the two monoidal structures in an $\ell$-bimonoid do not determine each other, in contrast to the case of lattices, where the join semilattice is uniquely determined by the meet semilattice and vice versa. For example, if $\langle L, \vee, \wedge, 0, 1 \rangle$ is a bounded distributive lattice, then taking $x \cdot y = x \wedge y$ and $x + y = x \vee y$ yields an $\ell$-bimonoid, but $x \cdot y = x \wedge y = x + y$ also yields an $\ell$-bimonoid (with the appropriate multiplicative and additive units).

  The bisemigroups that we shall consider in this paper will often be residuated, meaning that the multiplication has a right and a left residual. Note that addition will not be required to have a residual (or a dual residual) in such structures.

\begin{definition}[Residuated bisemigroups and $\ell$-bisemigroups]
  A \emph{residuated bisemigroup} is an ordered algebra $\langle A, \leq, \cdot, \bs, /, + \rangle$ such that $\langle A, \leq, \cdot, + \rangle$ is a bisemigroup and $\langle A, \leq, \cdot, \bs, / \rangle$ is a residuated posemigroup. A \emph{residuated $\ell$-bisemigroup} is an algebra $\langle A, \vee, \wedge, \cdot, \bs, /, + \rangle$ such that $\langle A, \vee, \wedge, \cdot, + \rangle$ is an $\ell$-bisemigroup and $\langle A, \leq, \cdot, \bs, / \rangle$ is a residuated posemigroup. A \emph{residuated bimonoid} (\emph{residuated $\ell$-bimonoid}) is a bimonoid ($\ell$-bimonoid) which is also a residuated bisemigroup ($\ell$-bisemigroup).
\end{definition}

\subsection{Bimonoids and \texorpdfstring{$\ell$-bimonoids}{l-bimonoids}: complementation}
\label{subsec: complementation}

  Bimonoids form an appropriate setting for the study of a general notion of complementation, which in particular subsumes Boolean complements in distributive lattices and multiplicative inverses in pomonoids. We~show that involutive residuated pomonoids (lattices) are up to term equivalence precisely bimonoids ($\ell$-bimonoids) equipped with complementation.

\begin{definition}[Complements in bisemigroups]
  Let $x$ and $y$ be elements of a bisemigroup $\alg{A}$. Then $y$ is called a \emph{left complement ($\ell$-complement)} of $x$ in $\alg{A}$ if the following inequalities hold for each $w \in \alg{A}$:
\begin{align*}
  w + (y \cdot x) & \leq w, & w & \leq w \cdot (x + y), \\
  (y \cdot x) + w & \leq w, & w & \leq (x + y) \cdot w.
\end{align*}
  It is called a \emph{right complement ($r$-complement)} of $x$ in $\alg{A}$ if the following inequalities hold for each $w \in \alg{A}$:
\begin{align*}
  w + (x \cdot y) & \leq w, & w & \leq w \cdot (y + x), \\
  (x \cdot y) + w & \leq w, & w & \leq (y + x) \cdot w.
\end{align*}
  A bisemigroup is \emph{complemented} if each element has a left complement and a right complement.
\end{definition}

  Surjective homomorphisms of bisemigroups preserve all existing left and right complements. In bimonoids, the definition of a complement can be simplified substantially.

\begin{proposition}[Complements in bimonoids]
  Let $x$ and $y$ be elements of a bimonoid $\alg{A}$. Then $y$ is an $\ell$-complement of $x$ in $\alg{A}$ if and only if
\begin{align*}
  y \cdot x & \leq 0 & & \text{ and } & 1 & \leq x + y.
\end{align*}
  Likewise, $y$ is an $r$-complement of $x$ in $\alg{A}$ if and only if
\begin{align*}
  x \cdot y & \leq 0 & & \text{ and } & 1 & \leq y + x.
\end{align*}
\end{proposition}

  In particular, $0$ and $1$ are both left and right complements of each other in a bimonoid. All homomorphisms of bimonoids preserve existing left and right complements. Crucially, left and right complements are unique whenever they exist. This holds even in a bisemigroup.

\begin{proposition}[Uniqueness of complements]
  Each element of a bisemigroup has at most one $\ell$-complement (at most one $r$-complement).
\end{proposition}

\begin{proof}
  If $y$ and $z$ are $\ell$-complements of $x$, then $y \leq y \cdot (x + z) \leq (y \cdot x) + z \leq z$. The other case is analogous.
\end{proof}

  We use $\compl{x}$ ($\compr{x}$) to denote the unique left (right) complement of $x$ whenever it exists. (We shall verify shortly that this is consistent with our previous usage of this notation.) In~the commutative case clearly $\compl{x} = \compr{x}$ whenever these exist. We use the notation $\comp{x}$ in this case.

\begin{proposition}[Residuation laws for complements] \label{prop: residuation for complements}
  The following residuation laws hold in each bisemigroup whenever the complements exist:
\begin{align*}
  x \cdot y \leq z & \iff y \leq \compr{x} + z, & x \leq y + z & \iff \compl{y} \cdot x \leq z, \\
  x \cdot y \leq z & \iff x \leq z + \compl{y}, & x \leq y + z & \iff x \cdot \compr{z} \leq y.
\end{align*}
\end{proposition}

\begin{proof}
  If $x \cdot y \leq z$, then $y \leq (\compr{x} + x) \cdot y \leq \compr{x} + (x \cdot y) \leq \compr{x} + z$. Conversely, if $y \leq \compr{x} + z$, then $x \cdot y \leq x \cdot (\compr{x} + z) \leq (x \cdot \compr{x}) + z \leq z$. The other claims follow by left--right duality and order duality.
\end{proof}

\begin{proposition}[De~Morgan laws for bisemigroups] \label{prop: de morgan laws}
  Bisemigroups satisfy the De Morgan laws and double negation elimination whenever the complements on the right-hand side of the equation exist:
\begin{align*}
  \compl{(x \cdot y)} & = \compl{y} + \compl{x}, & \compl{(x + y)} & = \compl{y} \cdot \compl{x}, & x & = \comprl{x}, \\
  \compr{(x \cdot y)} & = \compr{y} + \compr{x}, & \compr{(x + y)} & = \compr{y} \cdot \compr{x}, & x & = \complr{x}.
\end{align*}
  Complementation is antitone in bisemigroups whenever the complements exist:
\begin{align*}
  x \leq y & \implies \compl{y} \leq \compl{x}, \\
  x \leq y & \implies \compr{y} \leq \compr{x}.
\end{align*}
  Moreover, $\ell$-bisemigroups satisfy the De Morgan laws for lattice connectives whenever the complements on the right-hand side of the equation exist:
\begin{align*}
  \compl{(x \wedge y)} & = \compl{x} \vee \compl{y}, & \compl{(x \vee y)} & = \compl{x} \wedge \compl{y}, \\
  \compr{(x \wedge y)} & = \compr{x} \vee \compr{y}, & \compr{(x \vee y)} & = \compr{x} \wedge \compr{y}.
\end{align*}
\end{proposition}

\begin{proof}
  The equalities $x = \complr{x}$ and $x = \comprl{x}$ hold because by definition $y$ is an $\ell$-complement of $x$ if and only if $x$ is an $r$-complement of $y$. If $x \leq y$, then
\begin{align*}
  \compl{y} & \leq \compl{y} \cdot (x + \compl{x}) \leq (\compl{y} \cdot x) + \compl{x} \leq (\compl{y} \cdot y) + \compl{x} \leq \compl{x},
\end{align*}
  and likewise $\compr{y} \leq \compr{x}$. The De~Morgan laws for the lattice connectives follow from the antitonicity of $\compl{x}$ and $\compr{x}$ and the equalities $x = \complr{x}$ and $x = \comprl{x}$. The De~Morgan laws for $\compr{(x \cdot y)}$ and $\compr{(x + y)}$ will follow from the De~Morgan laws for $\compl{(x \cdot y)}$ and $\compl{(x + y)}$ by left--right symmetry. Likewise, the De~Morgan law for $\compl{(x + y)}$ follows from the De~Morgan law for $\compl{(x \cdot y)}$. It remains to prove that $\compl{(x \cdot y)}$ exists and equals $\compl{y} + \compl{x}$. But
\begin{align*}
  w + ((\compl{y} + \compl{x}) \cdot x \cdot y) \leq w + ((\compl{y} + (\compl{x} \cdot x)) \cdot y) \leq w + (\compl{y} \cdot y) \leq w, \\
  w \cdot ((x \cdot y) + \compl{y} + \compl{x}) \geq w \cdot ((x \cdot (y + \compl{y})) + \compl{x}) \geq w \cdot (x + \compl{x}) \geq w,
\end{align*}
  and likewise for the other two inequalities which define $\compl{(x \cdot y)}$.
\end{proof}  

  The above propositions in fact show that complemented bisemigroups (bimonoids) are precisely involutive residuated posemigroups (pomonoids) presented in a slightly different way. For the purposes of the following proposition, the expansion of a complemented bisemigroup (bimonoid) by the two unary operations $\compl{x}$ and $\compr{x}$ will be called a bisemigroup (bimonoid) \emph{with complementation}. In other words, the class of all bisemigroups (bimonoids) with complementation forms an ordered variety which is axiomatized by adding the following inequalities to an axiomatization of bisemigroups (bimonoids):
\settowidth{\auxlength}{$w \cdot (\compr{x} + x)$}
\settowidth{\auxlengthtwo}{$w + (x \cdot \compr{x})$}
\begin{align*}
  \hbox to \auxlengthtwo{$w + (\compl{x} \cdot x)$} & \inequals w, & \hbox to \auxlengthtwo{$w + (x \cdot \compr{x})$} & \inequals w, & w & \inequals \hbox to \auxlength{$w \cdot (x + \compl{x})$}, & w & \inequals \hbox to \auxlength{$w \cdot (\compr{x} + x)$}, \\
  \hbox to \auxlengthtwo{$(\compl{x} \cdot x) + w$} & \inequals w, & \hbox to \auxlengthtwo{$(x \cdot \compr{x}) + w$} & \inequals w, & w & \inequals \hbox to\auxlength{$(x + \compl{x}) \cdot w$}, & w & \inequals \hbox to\auxlength{$(\compr{x} + x) \cdot w$}.
\end{align*}

\begin{proposition}[Involutive residuated posemigroups are complemented bisemigroups]
  Bimonoids with complementation are term equivalent to involutive residuated pomonoids. Bisemigroups with complementation are term equivalent to involutive residuated posemigroups which satisfy the inequalities
\begin{align*}
  x & \inequals x (y \bs y), & x & \inequals (y \bs y) x, \\
  x & \inequals x (y / y), & x & \inequals (y / y) x.
\end{align*}
\end{proposition}

\begin{proof}
  The previous proposition shows that bisemigroups with complementation can be seen as involutive residuated posemigroups. Moreover, the above four inequalities (which are satisfied in each residuated pomonoid) are precisely four of the eight inequalities defining complements in bisemigroups. Conversely, given an involutive residuated posemigroup satisfying the above four inequalities, we need to show that $\compl{x}$ and $\compr{x}$ are the left and right complements of $x$ and that the hemidistributive law holds. Let us first prove hemidistributivity:
\begin{align*}
  x \cdot (y + z) \leq (x \cdot y) + z & \iff x \cdot \compr{(\compl{z} \cdot \compl{y})} \leq \compr{(\compl{z}\cdot \compl{(x \cdot y)})} \iff \compl{z}\cdot \compl{(x \cdot y)} \cdot x \leq \compl{z} \cdot \compl{y}.
\end{align*}
  The last inequality holds because $\compl{(x \cdot y)} \cdot x \leq \compl{y}$ by residuation. The four inequalities above now immediately yield four of the eight inequalities defining complements. The other four inequalities follow by applying the operations $\compl{}$ and $\compr{}$. For example, $\compl{w} + (\compl{x} \cdot x) \leq \compl{w}$ implies that $w = \comprl{w} \leq \compr{(\compl{w} + (\compl{x} \cdot x))} = (\compr{x} + x) \cdot w$.
\end{proof}

\subsection{Examples of bisemigroups and bimonoids}
\label{subsec: examples of bimonoids}

  We saw in the previous section that bisemigroups (bimonoids) occur as subreducts of involutive residuated posemigroups (pomonoids). Let us now consider several other ways of constructing bisemigroups and bimonoids in order to provide the reader with a stock of examples.

\begin{definition}[Commutative, integral, and idempotent bisemigroups]
  A bisemigroup is \emph{commutative} if its multiplicative and additive posemigroups are both commutative, it is \emph{multiplicatively (additively) integral} if its multiplicative (additive) posemigroup is integral, it is \emph{bi-integral} if it is both multiplicative and additively integral, and it is \emph{idempotent} if its multiplicative and additive posemigroups are both idempotent.
\end{definition}

  In particular, a bimonoid is multiplicatively (additively) integral if it satisfies $x \inequals 1$ ($0 \inequals x$). Idempotent bi-integral bisemigroups are in fact very familiar objects.

\begin{proposition}[Distributive lattices as bisemigroups]
  Idempotent bi-integral bisemigroups (bimonoids) are precisely (bounded) distributive lattices equipped with the lattice order.
\end{proposition}

\begin{proof}
  Integral idempotent posemigroups (pomonoids) are known to be precisely (unital) meet semilattices, the partial order coinciding with the semilattice order. Integral idem\-potent bisemigroups (bimonoids) are thus (bounded) lattices, the partial order being the lattice order. Moreover, for lattices hemidistributivity is equivalent to distributivity.
\end{proof}

  Secondly, each posemigroup (pomonoid) can be seen as a bisemigroup (bimonoid) if we take
\begin{align*}
  x + y & \assign x \cdot y & & \text{and} & 0 & \assign 1 \text{ (for pomonoids)}.
\end{align*}
  In the same way, each lattice-ordered semigroup (monoid) where multiplication distributes over binary meets and joins can be seen as an $\ell$-bisemigroup ($\ell$-bimonoid). Although these examples may at first sight seem too trivial to be of any interest, we shall see that non-trivial bimonoids may be constructed from these as bimonoids of fractions. In particular, it will be useful to view $\ell$-groups and Brouwerian algebras as residuated $\ell$-bimonoids where the multi\-plicative and additive monoids coincide.

  Thirdly, observe that in the presence of multiplicative integrality the inequalities
\begin{align*}
  x \cdot (y + z) & \leq (x \cdot y) + (x \cdot z) & & \text{and} & (y + z) \cdot x & \leq (y \cdot x) + (z \cdot x)
\end{align*}
imply hemidistributivity. Examples~of multiplicatively integral $\ell$-bisemigroups ($\ell$-bimonoids) thus include all integral s$\ell$-semigroups (with a bottom element $\bot$) if we take
\begin{align*}
  x + y & \assign x \vee y & & \text{and} & 0 & \assign \bot \text{ (if $\bot$ exists)}.
\end{align*}

  A bisemigroup can also be obtained from any posemigroup with a top element $\top$ by taking the $\top$-drastic addition: $x + y \assign \top$ for all $x$ and $y$. Similarly a bisemigroup can be obtained from a posemigroup with a bottom element $\bot$ such that $\bot \cdot x = x = x \cdot \bot$ for all $x$ by taking the $\bot$-drastic addition: $x + y \assign \bot$ for all $x$ and~$y$. Only the first of these constructions extends to pomonoids in a reasonable way: if we start with an integral pomonoid with a bottom element $\bot$ such that
\begin{align*}
  x \cdot y = \bot \iff x = \bot \text{ or } y = \bot,
\end{align*}
  then the following modified $\top$-drastic addition yields a bimonoid with $0 \assign \bot$:
\settowidth{\auxlength}{$x$}
\begin{align*}
  x + y & \assign \begin{cases}\hbox to \auxlength{\hfil$1$\hfil} \text{ if } x > \bot \text{ and } y > \bot, \\ \hbox to \auxlength{$x$} \text{ if } y = \bot, \\ \hbox to \auxlength{\hfil$y$\hfil} \text{ if } x = \bot. \end{cases}
\end{align*}
  Integral pomonoids which satisfy the required condition are in fact easy to come by: take an arbitrary integral pomonoid and append a new bottom element $\bot$ with
\begin{align*}
  \bot \cdot x = \bot = x \cdot \bot.
\end{align*}

  Our next family of examples consists of bimonoids obtained from pointed Brouwerian algebras.

\begin{definition}[Brouwerian algebras and Heyting algebras]
  A \emph{Brouwerian algebra} $\langle A, \vee, \wedge, 1, \rightarrow \rangle$ is a distributive lattice $\langle A, \vee, \wedge \rangle$ with a top element $1$ such that the binary operation $x \rightarrow y$ satisfies the equivalence
\begin{align*}
  x \wedge y \leq z & \iff y \leq x \rightarrow z.
\end{align*}
  A~\emph{pointed Brouwerian algebra} is one equipped with a constant~$0$. A \emph{Heyting algebra} is a pointed Brouwerian algebra where $0$ is the smallest element.
\end{definition}

  Equivalently, Brouwerian (Heyting) algebras are precisely (bounded) idempotent integral residuated lattices. Heyting algebras and (pointed) Brouwerian algebras are varieties, and the variety of Brouwerian algebras may be identified with the subvariety of pointed Brouwerian algebras such that $0 = 1$. Pointed Brouwerian algebras are also called Johansson algebras or $j$-algebras~\cite{odintsov08}.

\begin{proposition}[Pointed Brouwerian algebras as bimonoids] \label{prop: pointed brouwerian algebras}
  The variety of pointed Brouwerian algebras and the variety of multiplicatively integral idem\-potent com\-mutative residuated $\ell$-bimonoids are term equivalent via the correspondence
\begin{align*}
  x \cdot y & \assign x \wedge y & & \text{ and } & x + y & \assign (0 \rightarrow (x \wedge y)) \wedge (x \vee y).
\end{align*}
\end{proposition}

\begin{proof}
  Given such a commutative residuated $\ell$-bimonoid $\langle A, \vee, \wedge, \cdot, 1, \rightarrow, +, 0 \rangle$, idempotence and multi\-plicative integrality imply that $x \cdot y = x \wedge y$, therefore $\langle A, \vee, \wedge, 1, \rightarrow, 0 \rangle$ is a pointed Brouwerian algebra. We now show that $x + y = (0 \rightarrow xy)(x \vee y)$. Clearly $x + y \leq (x \vee y) + (x \vee y) = x \vee y$, and $0 (x + y) = 0^2 (x + y) \leq 0x + 0y \leq 0 x y$, since $0x + 0y \leq 0 x + 0 = 0 x$ and $0 x + 0 y \leq 0 + 0 y \leq 0 y$, therefore $x + y \leq 0 \rightarrow xy$. Conversely, to prove that $(0 \rightarrow x y) (x \vee y) \leq x + y$, it suffices to prove that $(0 \rightarrow x y) x \leq x + y$. But we have $(0 \rightarrow xy) x \leq x (0 \rightarrow y) = (x + 0) (0 \rightarrow y) \leq x + 0 (0 \rightarrow y) \leq x + y$. Thus $x + y = (0 \rightarrow (x \wedge y)) (x \vee y)$.

  Conversely, given a pointed Brouwerian algebra, we have $x + y = y + x$ and $x + 0 = (0 \rightarrow x0) (x \vee 0) = (0 \rightarrow x) (x \vee 0) = (0 \rightarrow x)x \vee  (0 \rightarrow x)0 = x \vee 0 x = x$, since $x=x\cdot x \leq  (0 \rightarrow x)x\leq 1\cdot x=x$. Moreover,
\begin{align*}
  x + (y + z) & = (0 \rightarrow x (0 \rightarrow yz) (y \vee z)) (x \vee (0 \rightarrow yz) (y \vee z)) \\
  & = (0 \rightarrow xyz) (x \vee (0 \rightarrow yz) (y \vee z)) \\
  & = (0 \rightarrow xyz) (x \vee y \vee z)),
\end{align*}
  so $x + (y + z) = z + (x + y) = (x + y) + z$. Finally, we need to verify hemidistributivity: $x (y + z) = x (0 \rightarrow yz) (y \vee z) \leq (0 \rightarrow xyz) (xy \vee z) = xy + z$. The addition operation therefore yields a bimonoid.
\end{proof}

  In particular, a specific variety of pointed Brouwerian algebras will be important later.

\begin{definition}[Boolean-pointed Brouwerian algebras]
  A pointed Brouwerian algebra is called \emph{Boolean-pointed} if the interval $[0, 1]$ is a Boolean lattice, or equivalently if it satisfies the equation $x \vee (x \rightarrow 0) \equals 1$.
\end{definition}

 The following lemma will be used later to simplify the proof of Fact~\ref{f:BpBr}.

\begin{lemma}[Inequational validity in Boolean-pointed Brouwerian algebras] \label{lemma: validity in bpbras}
  Let $t(x_1, \dots, x_m, y_1, \dots, y_n)$ and $u(x_1, \dots, x_m, y_1, \dots, y_n)$ be terms in the signature of Brouwerian algebras. The inequality $t(x_1, \dots, x_m, 0, \dots, 0) \inequals u(x_1, \dots, x_m, 0, \dots, 0)$  holds in all Boolean-pointed Brouwerian algebras if and only if $t(x_1, \dots, x_m, 1, \dots, 1) \inequals u(x_1, \dots, x_m, 1, \dots, 1)$ holds in all Brouwerian algebras and $t(x_1, \dots, x_m, 0, \dots, 0) \vee 0 \inequals u(x_1, \dots, x_m, 0, \dots, 0) \vee 0$ holds in all Boolean-pointed Brouwerian algebras.
\end{lemma}

\begin{proof}
  The left-to-right implication is immediate. Conversely, suppose that $b = t(a_1, \dots, a_m, 0, \dots, 0) \nleq u(a_1, \dots, a_m, 0 \dots, 0) = c$ in some Boolean-pointed Brouwerian algebra $\alg{A}$. Then either $b \wedge 0 \nleq c \wedge 0$ or $b \vee 0 \nleq c \vee 0$. In the former case, let $F$ be the filter generated by $0$, and let $\pi\colon \alg{A} \to \alg{A} / F$ be the appropriate projection map, which is a homomorphism of Brouwerian algebras. Then $b \wedge 0 \nleq c \wedge 0$ implies that $\pi(b) \nleq \pi(c)$. It follows that the equality $t(x_1, \dots, x_m, 1, \dots, 1) \equals u(x_1, \dots, x_m, 1, \dots, 1)$ fails in the Brouwerian algebra $\alg{A} / F$ for $x_i = \pi(a_i)$, since $\pi(0) = \pi(1)$. In the latter case, $c \rightarrow 0 \nleq b \rightarrow 0$.
\end{proof}

  Further examples of bi-integral bisemigroups can be obtained using an ordinal sum construction. Consider a family of bisemigroups $\alg{A}_{i}$ for $i \in I$, where $I$ is a chain ordered by $\sqleq$. The chain is called \emph{non-trivial} if it has at least two elements. We define the \emph{ordinal sum} of this family to be an ordered algebra $\alg{A} = \langle A, \leq, +, \cdot \rangle$ over the universe $A \assign \bigcup_{i \in I} A_{i}$. Let $a \in \alg{A}_{i}$ and $b \in \alg{A}_{j}$. The~order on $\alg{A}$ is defined as follows:
\begin{align*}
  a \leq b & \iff \text{either } i \sqle j \text{ or } i = j \text{ and } a \leq b.
\end{align*}
  The operations of $\alg{A}$ are:
\begin{align*}
  a \cdot b & = \begin{cases}a & \text{ if } i \sqleq j, \\ a \cdot b & \text{ if } i = j, \\ b & \text{ if } i \sqge j, \end{cases} & a + b & = \begin{cases}b & \text{ if } i \sqle j, \\ a + b & \text{ if } i = j, \\ a & \text{ if } i \sqge j. \end{cases}
\end{align*}

\begin{fact}
  Let $I$ be a non-trivial chain ordered by $\sqleq$ and let $\alg{A}_{i}$ be a family of (commutative) bisemigroups. Then its ordinal sum $\alg{A}$ is a bisemigroup if and only if each bisemigroup $\alg{A}_{i}$ for $i \in I$ is bi-integral. In that case $\alg{A}$ is a (commutative) bi-integral bisemigroup.
\end{fact}

\begin{proof}
  Given a family of bi-integral bisemigroups, the two operations are clearly associative. To show that they are isotone, consider $a \in \alg{A}_{i}$, $b \in \alg{A}_{j}$, $c \in \alg{A}_{k}$. If $i = j$ and $a \leq b$, then case analysis ($k \sqle i$, $k = i$, $k \sqle i$) shows that $a \cdot c \leq b \cdot c$ and $c \cdot a \leq c \cdot b$. Suppose therefore that $i \sqle j$. If $k \sqle i$ or $j \sqle k$, the inequalities $a \cdot c \leq b \cdot c$ and $c \cdot a \leq c \cdot b$ follow from $c \leq c$ and $i \sqle j$, respectively. If $k = i$, they reduce to the inequalities $a \cdot c \leq c$ and $c \cdot a \leq c$, which hold by the multiplicative integrality of $\alg{A}_{i} = \alg{A}_{k}$. If $k = j$, they follow from $i \sqle j$. The proof for addition is entirely analogous and involves additive integrality.

  Let us now verify that $a \cdot (b + c) \leq (a \cdot b) + c$ for $a \in \alg{A}_{i}$, $b \in \alg{A}_{j}$, $c \in \alg{A}_{k}$. The proof of the other hemidistributive law is analogous. If $i, j, k$ are distinct, then the inequality $a \cdot (b + c) \leq (a \cdot b) + c$ in $\alg{A}$ follows from the inequality $i \wedge (j \vee k) \sqleq (i \wedge j) \vee k$ in $I$. If $i = j = k$, then the inequality follows from the same inequality in $\alg{A}_{i} = \alg{A}_{j} = \alg{A}_{k}$. If $i \sqle j = k$, then the inequality simplifies to $a \leq c$, which follows from $i \leq k$. If $j = k \sqle i$, then the inequality simplifies to $b + c \leq b + c$. The~cases $i = j \sqle k$ and $k \sqle i = j$ are similar. Finally, if $i = k \sqle j$, then the inequality simplifies to $a \leq a + c$, while if $j \sqle i = k$, then the inequality simplifies to $a \cdot c \leq c$. These inequalities follow respectively from the additive integrality of $\alg{A}_{i}$ and the multiplicative integrality of $\alg{A}_{k}$.

  The ordered algebra $\alg{A}$ is thus a bisemigroup, and it is easily seen to be bi-integral. Conversely, suppose that $\alg{A}_{i}$ is not multiplicatively integral for some $i \in I$, say $a \cdot b \nleq b$ for some $a, b \in \alg{A}_{i}$. (If $\alg{A}_{i}$ is not additively integral instead of multiplicatively integral, or if $a \cdot b \nleq a$ instead of $a \cdot b \nleq b$, the proof is entirely analogous.) Consider some $c \in \alg{A}_{j}$ for $j \neq i$. If $i \sqle j$, then $a \leq c$ but $a \cdot b \nleq b = c \cdot b$. If on the other hand $j \sqle i$, then $a \cdot (c + b) = a \cdot b \nleq b = (a \cdot c) + b$.
\end{proof}

  If the chain $I$ has an upper bound $\top$ and $\alg{A}_{\top}$ has a multiplicative unit $1$, then clearly $1$ is a multiplicative unit for the whole of $\alg{A}$. Likewise, if $I$ has a lower bound $\bot$ and $\alg{A}_{\bot}$ has an additive unit $0$, then $0$ is an additive unit for the whole of $\alg{A}$.

  One special case of this construction is worth mentioning: if $\alg{B}$ is a bi-integral bisemigroup, we may take $I = \{ 0, 1 \}$ with $0 \sqle 1$, $\alg{A}_{0} \assign \alg{B}$, and $\alg{A}_{1} \assign \alg{B}^{\dual}$ (the order dual of $\alg{B}$). If $\alg{B}$ has an additive unit, then $\alg{A}$ has both a multiplicative and an additive unit, i.e.\ it is a bi-integral bimonoid. For example, consider the additive posemigroup $\alg{B} \assign \langle \N, \leq, + \rangle$ of non-negative integers with the usual ordering. Expanding it by the drastic multiplication $x \cdot y \assign 0$ yields a bi-integral bisemigroup with an additive unit. Applying the above construction yields a bi-integral bimonoid, which is order isomorphic to the Chang MV-algebra. However, the bimonoidal structure is quite different when multiplying elements from $\alg{A}_{0}$ and~$\alg{A}_{1}$.

  Finally, any bisemigroup may be extended to a bounded bisemigroup by adding new top and bottom elements $\top$ and $\bot$ such that
\begin{align*}
  \bot \cdot x & = \bot = x \cdot \bot, & \top + x & = \top = x + \top, \\
  x + \bot = \bot + x & = \begin{cases}\bot \text{ if } x < \top, \\ \top \text{ if } x = \top,\end{cases} & x \cdot \top = \top \cdot x & = \begin{cases}\top \text{ if } x > \bot, \\ \bot \text{ if } x = \bot. \end{cases}
\end{align*}
  On the other hand, the task of extending an arbitrary bisemigroup to a bimonoid is not quite so simple. The~reader may for example consider the problem of extending the bisemigroup $\langle \Zplus, \geq, \cdot, + \rangle$ of positive integers with the usual multiplication and addition and the dual of the usual order to a bimonoid.

\section{Complemented MacNeille completions of bimonoids and bisemigroups}
\label{sec: macneille}

  We now establish the first main result of the present paper: each commutative ($\ell$-)bimonoid $\alg{A}$ embeds into a complete complemented commutative $\ell$-bimonoid~$\cdmalg{A}$ in a certain doubly dense way akin to the Dedekind--MacNeille completion. The construction of this complemented completion is accomplished using so-called involutive residuated frames, introduced by Galatos \& Jipsen~\cite{galatos+jipsen13residuated-frames}.

\subsection{Complemented MacNeille completions: definition}
\label{subsec: macneille definition}

  The Dedekind--MacNeille completion of a lattice $\alg{L}$ is defined as an embedding of~$\alg{L}$ into a complete lattice where the image of $\alg{L}$ is both meet dense and join dense in a precise sense. This completion is unique up to a unique isomorphism which fixes (the image of)~$\alg{L}$. We~now define the \emph{complemented} Dedekind--MacNeille (DM) completion of a bimonoid, and show that in the commutative case it is unique in the same sense.

\begin{definition}[Join density and meet density]
  A subset $X$ of a poset $P$ is \emph{meet dense (join dense)} in $P$ if each element of $P$ is a meet (join) of some subset of~$X$. A subset $X$ of a bimonoid $\alg{A}$ is \emph{admissibly meet dense (admissibly join dense)} in $\alg{A}$ if each element of $\alg{A}$ is an admissible meet (admissible join) of some subset of $X$.
\end{definition}

  Equivalently, a set $X$ is meet dense in $P$ if and only if for all $a, b \in P$
\begin{align*}
  a \leq b \iff (b \leq x \implies a \leq x \text{ for all } x \in X).
\end{align*}
  Likewise, $X$ is join dense in $P$ if and only if for all $a, b \in P$
\begin{align*}
  a \leq b \iff (x \leq a \implies x \leq b \text{ for all } x \in X).
\end{align*}
  Recall that $\comp{a}$ denotes the complement of $a$ in a commutative bimonoid (if it exists), i.e.\ $\comp{a} = \compl{a} = \compr{a}$.

\begin{definition}[Commutative $\Delta_{1}$-extensions]
  Let $\iota\colon \alg{A} \into \alg{B}$ be an embedding of a commutative bimonoids. We call the embedding $\iota$, and by extension the $\ell$-bimonoid~$\alg{B}$ itself, a \emph{commutative (admissible) $\Delta_{1}$-extension} of $\alg{A}$ if
\begin{align*}
  \text{elements of the form $\iota(a) \cdot \comp{\iota(b)}$ with $a, b \in \alg{A}$ are (admissibly) join dense in $\alg{B}$,}
\end{align*}
  and
\begin{align*}
  \text{elements of the form $\iota(a) + \comp{\iota(b)}$ with $a, b \in \alg{A}$ are (admissibly) meet dense in $\alg{B}$.}
\end{align*}
\end{definition}

  The above definition does not assume that the complements $\comp{\iota(b)}$ exist for each $b \in \alg{A}$. In other words, a commutative \mbox{$\Delta_{1}$-extension} need not be complemented. In particular, each commutative bimonoid is a commutative admissible $\Delta_{1}$-extension of itself, since each of its elements $a$ has the form $a \cdot \comp{0}$, or equivalently $a + \comp{1}$. In a complemented commutative $\Delta_{1}$-extension each existing meet and join is admissible, therefore complemented commutative $\Delta_{1}$-extensions are always admissible. Moreover, the meet density and join density conditions are equivalent in this case.

\begin{definition}[Complemented Dedekind--MacNeille completions]
  A \emph{commutative complemented Dedekind--MacNeille (DM) completion} of a commutative bimonoid is a complete commutative complemented $\Delta_{1}$-extension.
\end{definition}

  A commutative complemented DM completion of a finite commutative bimonoid of cardinality $n$ has cardinality at most $2^{(n^{2})}$, since there are at most $n^{2}$ elements of the form $\iota(a) \cdot \comp{\iota(b)}$ and each element of the completion is a join of some set of these elements.

\begin{fact} \label{fact: admissible joins preserved}
  Let $\iota\colon \alg{A} \into \extalg{A}$ be a commutative $\Delta_{1}$-extension. Then $\iota$ preserves all admissible joins and meets which exist in $\alg{A}$. In particular, if $\alg{A}$ is an $\ell$-bimonoid, then $\iota$ is an embedding of $\ell$-bimonoids.
\end{fact}

\begin{proof}
  We only prove the claim for joins. Suppose that $a = \bigvee_{i \in I} a_{i}$ is an admissible join in $\alg{A}$. Clearly $\iota (a_{i}) \leq \iota (a)$ for each $i \in I$. If $\iota(a_{i}) \leq \iota(c) + \comp{\iota(d)}$ for each $i \in I$, then $\iota(a_{i} \cdot d) = \iota(a_{i}) \cdot \iota(d) \leq \iota(c)$ for each $i \in I$, therefore $a_{i} \cdot d \leq c$ for each $i \in I$, and $a \cdot d \leq c$ by the admissibility of $a$. Thus $\iota(a) \cdot \iota(d) = \iota(a \cdot d) \leq \iota(c)$ and $\iota(a) \leq \iota(c) + \comp{\iota(d)}$. It now follows from the meet density of elements of the form $\iota(c) + \comp{\iota(d)}$ that $\iota(a) = \bigvee_{i \in I} \iota(a_{i})$.
\end{proof}
\nopagebreak
  We generally disregard the embedding $\iota\colon \alg{A} \into \extalg{A}$ and treat $\alg{A}$ as a sub-bimonoid of $\extalg{A}$. Each commutative $\Delta_{1}$-extension $\extalg{A}$ of a commutative bimonoid $\alg{A}$ turns out to be an \emph{essential extension} in the sense that each homomorphism of bimonoids $h\colon \extalg{A} \to \alg{B}$ is an embedding whenever its restriction to $\alg{A}$ is.

\begin{proposition}[Commutative $\Delta_{1}$-extensions are essential] \label{prop: essential extensions}
  Each commutative $\Delta_{1}$-extension of a bimonoid is essential.
\end{proposition}

\begin{proof}
  Let $\extalg{A}$ be commutative $\Delta_{1}$-extension of $\alg{A}$ and ${h\colon \extalg{A} \to \alg{B}}$ be a homomorphism of bimonoids whose restriction to $\alg{A}$ is an embedding. Then $h$ preserves all existing complements. Suppose that $h$ is not an embedding. Then there are $a_{i}, b_{i}, c_{j}, d_{j} \in \alg{A}$ for $i \in I$ and $j \in J$ such that
\begin{align*}
  & h \biggl( \bigvee_{i \in I} a_{i} \cdot \comp{b_{i}} \biggr) \leq_{\alg{B}} h \biggl( \bigwedge_{j \in J} c_{j} + \comp{d_{j}} \biggr) & & \text{and} & \bigvee_{i \in I} a_{i} \cdot \comp{b_{i}} \nleq_{\extalg{A}} \bigwedge_{j \in J} c_{j} + \comp{d_{j}}.
\end{align*}
  It follows that $h(a \cdot \comp{b}) \leq_{\alg{B}} h(c + \comp{d})$ and $a \cdot \comp{b} \nleq_{\extalg{A}} c + \comp{d}$ for some $a, b, c, d \in \alg{A}$. But~then $h(a \cdot d) \leq_{\alg{B}} h(c + b)$ and $a \cdot d \nleq_{\extalg{A}} c + b$.  Since $a \cdot d \in \alg{A}$ and $c + b \in \alg{A}$, the~restriction of $h$ to $\alg{A}$ is not an embedding.
\end{proof}

  A given commutative bimonoid may have many distinct commutative $\Delta_{1}$-extensions. However, there is up to isomorphism only one such extension which is both \emph{complete} and \emph{complemented}. Moreover, it is universal in the sense that any other commutative admissible $\Delta_{1}$-extension embeds into it. This is a consequence of the following sequence of easy lemmas, where $\extalg{A}$ is a commutative $\Delta_{1}$-extension of a commutative bimonoid~$\alg{A}$ and $a_{i}, b_{i}, c_{j}, d_{j} \in \alg{A}$ for $i \in I$ and $j \in J$.

\begin{lemma}[Joins below meets] \label{lemma: joins below meets}
  Whenever the join and meet exist in~$\extalg{A}$,
\begin{align*}
   \bigvee_{i \in I} a_{i} \cdot \comp{b_{i}} \leq_{\extalg{A}} \bigwedge_{j \in J}  \comp{c_{j}} + d_{j} & \iff c_{j} \cdot a_{i} \leq_{\alg{A}} d_{j} + b_{i} \text{ for all } i \in I \text{ and } j \in J.
\end{align*}
\end{lemma}

\begin{proof}
  This follows from the residuation law for complements in bimonoids (Proposition~\ref{prop: residuation for complements}).
\end{proof}

\begin{lemma}[Joins below joins] \label{lemma: joins below joins}
  Whenever the two joins exist in $\extalg{A}$,
\begin{align*}
   \bigvee_{i \in I} a_{i} \cdot \comp{b_{i}} & \leq_{\extalg{A}} \bigvee_{j \in J} c_{j} \cdot \comp{d_{j}}
\end{align*}
  if and only if for all $x, y \in \alg{A}$
\begin{align*}
  x \cdot c_{j} \leq_{\alg{A}} y + d_{j} \text{ for all } j \in J ~ & \implies ~ x \cdot a_{i} \leq_{\alg{A}} y + b_{i} \text{ for all } i \in I.
\end{align*}
  If $\alg{A}$ is residuated and the meet below exists in $\alg{A}$, then this is equivalent to
\begin{align*}
  \bigwedge_{j \in J} (x + d_{j}) / c_{j} & \leq_{\alg{A}} (x + b_{i}) / a_{i} \text{ for all } i \in I \text{ and } x \in \alg{A}.
\end{align*}
  In particular $a \cdot \comp{b} \leq_{\extalg{A}} c \cdot \comp{d}$ if and only if $(x + d) / c \leq_{\alg{A}} (x + b) / a$ for all $x \in \alg{A}$.
\end{lemma}

\begin{proof}
  The inequality $\bigvee_{i \in I} a_{i} \cdot \comp{b_{i}} \leq_{\extalg{A}} \bigvee_{j \in J} c_{j} \cdot \comp{d_{j}}$ is equivalent to
\begin{align*}
  \bigvee_{j \in J} c_{j} \cdot \comp{d_{j}} \leq_{\extalg{A}} \comp{x} + y & \implies \bigvee_{i \in I} a_{i} \cdot \comp{b_{i}} \leq_{\extalg{A}} \comp{x} + y \text{ for all } x, y \in \alg{A}.
\end{align*}
  This amounts to
\begin{align*}
  & (\forall j \in J) \, ( c_{j} \cdot \comp{d_{j}} \leq_{\extalg{A}} \comp{x} + y ) \implies (\forall i \in I) \, ( a_{i} \cdot \comp{b_{i}} \leq_{\extalg{A}} \comp{x} + y ) \text{ for all } x, y \in \alg{A}.
\end{align*}
  But by the residuation law for complemented bimonoids (Proposition~\ref{prop: residuation for complements}) and the fact that $\alg{A}$ is a sub-bimonoid of $\extalg{A}$ this is equivalent to
\begin{align*}
  & (\forall j \in J) \, ( x \cdot c_{j} \leq_{\alg{A}} y + d_{j} ) \implies (\forall i \in I) \, ( x \cdot a_{i} \leq_{\alg{A}} y + b_{i}) \text{ for all } x, y \in \alg{A}.
\end{align*}
  Finally, if $\alg{A}$ is residuated, then the above condition is equivalent to
\begin{align*}
  & (\forall j \in J) \, ( x \leq_{\alg{A}} (y + d_{j}) / c_{j}) \implies (\forall i \in I) \, ( x \leq_{\alg{A}} (y + b_{i}) / a_{i} ) \text{ for all } x, y \in \alg{A}.
\end{align*}
  If the meet $\bigwedge_{j \in J} (y + d_{j}) / c_{j}$ exists in $\alg{A}$, then this is equivalent to
\begin{align*}
  \bigwedge_{j \in J} (y + d_{j}) / c_{j} & \leq_{\alg{A}} (y + b_{i}) / a_{i} \text{ for each } i \in I. \qedhere
\end{align*}
\end{proof}

\begin{lemma}[Meets below joins] \label{lemma: meets below joins}
  Whenever the join and meet exist in~$\extalg{A}$, the inequality
\begin{align*}
  \bigwedge_{i \in I}  \comp{a_{i}} + b_{i} \leq_{\extalg{A}} \bigvee_{j \in J} c_{j} \cdot \comp{d_{j}}
\end{align*}
  holds if and only if for all $u, v, x, y \in \alg{A}$
\begin{align*}
  (\forall i \in I) \, (\forall j \in J) \, ( u \cdot a_{i} \leq_{\alg{A}} v + b_{i} ~ \& ~ x \cdot c_{j} \leq y + d_{j} ) & \implies u \cdot x \leq v + y.
\end{align*}
\end{lemma}

\begin{proof}
  This equivalence is proved using the same method as the previous lemma, taking advantage of the observation that $x \leq y$ in a poset if and only if
\begin{align*}
  a \leq x ~ \& ~ y \leq b & \implies a \leq b
\end{align*}
  for each join generator $a$ and each meet generator $b$.
\end{proof}

\begin{theorem}[Universality of complemented DM completions] \label{thm: universality of dm completions}
  Let $\iota\colon \alg{A} \to \alg{B}$ be an isomorphism of commutative bimonoids and let $\extalg{A}$ and $\extalg{B}$ be commutative admissible $\Delta_{1}$-extensions of~$\alg{A}$ and $\alg{B}$. If $\extalg{B}$ is complete and complemented, then there is a unique complete embedding of bimonoids $\overline{\iota}\colon \extalg{A} \to \extalg{B}$ which extends the isomorphism $\iota$.
\end{theorem}

\begin{proof}
  We define the map $\overline{\iota}_{+}\colon \extalg{A} \to \extalg{B}$ as
\begin{align*}
  \overline{\iota}_{+} \colon \sideset{}{^{\extalg{A}}}\bigvee_{i \in I} a_{i} \comp{b_{i}} \mapsto \sideset{}{^{\extalg{B}}}\bigvee_{i \in I} \iota(a_{i}) \comp{\iota(b_{i})}.
\end{align*}
  The element $\bigvee_{i \in I}^{\extalg{B}} \iota(a_{i}) \comp{\iota(b_{i})}$ exists because $\extalg{B}$ is complete and complemented, and the map is a well-defined order-embedding by Lemma~\ref{lemma: joins below joins}. It preserves arbitrary joins by definition, and moreover it preserves products (here we use the admissibility of the extension $\extalg{\alg{A}}$):
\begin{align*}
  \overline{\iota}_{+} \biggl(\sideset{}{^{\extalg{A}}}\bigvee_{i \in I} a_{i} \comp{b_{i}} \cdot \sideset{}{^{\extalg{A}}}\bigvee_{i \in I} c_{i} \comp{d_{i}} \biggr) & = \overline{\iota}_{+} \biggl( \sideset{}{^{\extalg{A}}}\bigvee_{i \in I} \sideset{}{^{\extalg{A}}}\bigvee_{j \in J} a_{i} \comp{b_{i}} \cdot c_{j} \comp{d_{j}} \biggr) 
   = \overline{\iota}_{+} \biggl( \sideset{}{^{\extalg{A}}}\bigvee_{i \in I} \sideset{}{^{\extalg{A}}}\bigvee_{j \in J} a_{i} c_{j} \cdot \comp{d_{j} + b_{i}} \biggr) = \\
  & = \sideset{}{^{\extalg{B}}}\bigvee_{i \in I} \sideset{}{^{\extalg{B}}}\bigvee_{j \in J} \iota(a_{i} c_{j}) \cdot \comp{\iota(d_{j} + b_{i})} 
   = \sideset{}{^{\extalg{B}}}\bigvee_{i \in I} \sideset{}{^{\extalg{B}}}\bigvee_{j \in J} \iota(a_{i}) \comp{\iota(b_{i})} \cdot \iota(c_{j}) \comp{\iota(d_{j})} = \\
  & = \sideset{}{^{\extalg{B}}}\bigvee_{i \in I} \iota(a_{i}) \comp{\iota(b_{i})} \cdot \sideset{}{^{\extalg{B}}}\bigvee_{j \in J} \iota(c_{j}) \comp{\iota(d_{j})} 
   = \iota_{+} \biggl( \sideset{}{^{\extalg{A}}}\bigvee_{i \in I} a_{i} \comp{b_{i}} \biggr) \cdot \iota_{+} \biggl( \sideset{}{^{\extalg{A}}}\bigvee_{i \in I} c_{i} \comp{d_{i}} \biggr)
\end{align*}
  We can also define the map $\overline{\iota}_{-}\colon \extalg{A} \to \extalg{B}$ as
\begin{align*}
  \overline{\iota}_{-} \colon \sideset{}{^{\!\!\extalg{A}}}\bigwedge_{i \in I} \comp{a_{i}} + b_{i} \mapsto \sideset{}{^{\!\!\extalg{B}}}\bigwedge_{i \in I} \comp{\iota(a_{i})} + \iota(b_{i}).
\end{align*}
  This map is a well-defined order-embedding by the order dual of Lemma~\ref{lemma: joins below joins}. It again preserves arbitrary meets by definition, and it preserves sums by the order dual of the above argument. Moreover, $\overline{\iota}_{+}(1) = \overline{\iota}_{+}(1 \comp{0}) = \iota(1) \comp{\iota(0)} = \iota(1) \iota(1) = \iota(1)$. Similarly, $\overline{\iota}_{-}(0) = \iota(0)$.

  We know that $\overline{\iota}_{-}(x) = \overline{\iota}_{+}(x)$ for each $x \in \extalg{A}$ by Lemmas~\ref{lemma: joins below meets}~and~\ref{lemma: meets below joins}. Let~us denote this common value by $\overline{\iota}(x)$. Then $\overline{\iota}$ is an embedding of bimonoids and clearly $\overline{\iota}(a) = \iota(a)$ for $a \in \alg{A}$. Conversely, it is clear from its definition that $\overline{\iota}$ is the only complete embedding of $\extalg{A}$ into $\extalg{B}$ which extends $\iota$.
\end{proof}

  The proof of the universality of the commutative complemented DM completion relies substantially on commutativity. Without this assumption, one would have to impose a specific way of simplifying products of the form, say, $a \compr{b} c \compr{d}$ into $x \compr{y}$, and simplifying sums of the form, say, $a + \compl{b} + c + \compl{d}$ into $x + \compl{y}$.

\begin{corollary}[Uniqueness of complemented DM completions] \label{cor: uniqueness of dm completions}
  Let $\iota\colon \alg{A} \to \alg{B}$ be an isomorphism of commutative bimonoids and $\cdmalg{A}$ and $\cdmalg{B}$ be commutative complemented DM completions of $\alg{A}$ and $\alg{B}$. Then there is a unique isomorphism $\cdmhom{\iota}\colon \cdmalg{A} \to \cdmalg{B}$ which extends~$\iota$.
\end{corollary}

  Each complemented DM completion of $\alg{A}$ is in fact an ordinary DM completion of any complemented $\Delta_{1}$-extension of $\alg{A}$. Here~by an \emph{(ordinary) DM completion} we mean an embedding of bimonoids $\iota\colon \alg{A} \into \alg{B}$ such that $\iota[\alg{A}]$ is both join and meet dense in $\alg{B}$. Note that (ordinary) DM completions of involutive residuated structures were already constructed in~\cite{galatos+jipsen13residuated-frames} using the machinery of involutive residuated frames.

\begin{proposition}[Ordinary DM completions of complemented $\Delta_{1}$-extensions]
  Let $\iota_{1}\colon \alg{A} \into \extalg{A}$ be a commutative complemented $\Delta_{1}$-extension of a commutative bimonoid $\alg{A}$ and $\iota\colon \alg{A} \into \cdmalg{A}$ be a commutative complemented DM completion of a $\alg{A}$. Then there is a unique ordinary DM completion $\iota_{2}\colon \extalg{A} \into \cdmalg{A}$ such that $\iota = \iota_{2} \circ \iota_{1}$.
\end{proposition}

\begin{proof}
  By the universality of complemented DM completions, there is a unique complete embedding of bimonoids $\iota_{2}\colon \extalg{A} \into \cdmalg{A}$ such that $\iota = \iota_{2} \circ \iota_{1}$. Each element $x \in \cdmalg{A}$ therefore has the form $x = \bigvee_{i \in I} \iota_{2}(\iota_{1}(a_{i})) \comp{\iota_{2}(\iota_{1}(b_{i}))} = \bigvee_{i \in I} \iota_{2}(\iota_{1}(a_{i}) \comp{\iota_{1}(b_{i})})$ for some $a_{i}, b_{i} \in \alg{A}$, i.e.\ $\bigvee_{i \in I} \iota_{2}(y_{i})$ for some $y_{i} \in \extalg{A}$. Similarly, each $x \in \cdmalg{A}$ has the form $\bigwedge_{i \in I} \iota_{2}(y_{i})$ for some $y_{i} \in \extalg{A}$. Thus $\iota_{2}$ is an ordinary DM completion. To prove uniqueness, observe that any ordinary DM completion $\iota_{2}\colon \extalg{A} \into \cdmalg{A}$ such that $\iota = \iota_{2} \circ \iota_{1}$ agrees with this one on where it sends elements of the forms $\iota_{1}(a) \comp{\iota_{1}(b)}$ for $a, b \in \alg{A}$. Since each element of $\extalg{A}$ is a join of such elements, any two ordinary DM completions $\iota_{2}\colon \extalg{A} \into \cdmalg{A}$ such that $\iota = \iota_{2} \circ \iota_{1}$ must coincide.
\end{proof}

\begin{corollary}[Complemented DM completions of complemented bimonoids]
  If $\alg{A}$ is a commutative complemented bimonoid, then the commutative complemented DM completion of $\alg{A}$ is the ordinary DM completion of $\alg{A}$.
\end{corollary}

  In particular, the commutative complemented DM completion of a bounded distributive lattice $\alg{A}$ is the DM completion of the free Boolean extension of $\alg{A}$ (the smallest Boolean algebra into which $\alg{A}$ embeds). Let~us remark that this is known to be precisely the injective hull of $\alg{A}$ in the category of bounded distributive lattices (see~\cite{banaschewski+bruns68}). In~particular, it is a maximal essential extension (in the sense of Proposition~\ref{prop: essential extensions}).

  The question now arises whether this categorical characterization of the complemented DM completion as the injective hull generalizes to some broader class of bimonoids. Without pursuing the topic further, let us merely observe that there are commutative $\ell$-bimonoids $\alg{A}$ whose commutative complemented DM completion is not a maximal essential extension of~$\alg{A}$ (and therefore it is not the injective hull of $\alg{A}$). Equivalently, there are complete commutative complemented $\ell$-bimonoids with proper essential extensions.

  For example, one may take the essential embedding of the additive $\ell$-group of integers $\Z$ into the lexico\-graphic product $\ZlexZ$: the group $\Z \times \Z$ equipped with the lexico\-graphic order, where $\pair{a}{b} \leq \pair{c}{d}$ if and only if either $a \leq c$ or $a = c$ and $b \leq d$. The map $\iota\colon \Z \to \ZlexZ$ such that $\iota(n) \assign \pair{n}{0}$ is an essential embedding of $\Z$ into $\ZlexZ$. If we now extend $\Z$ and $\ZlexZ$ by a new top and bottom element $\top$ and $\bot$ and extend the embedding $\iota$ so that $\iota(\top) \assign \top$ and $\iota(\bot) \assign \bot$, we obtain a proper essential extension of a commutative complete complemented $\ell$-bimonoid. The complemented DM completion of $\Z$, i.e.\ the extension of $\Z$ by the top and bottom elements, is therefore not an injective hull of $\Z$, because the composite embedding of $\Z$ into the extension of $\ZlexZ$ by bounds is also essential.

\subsection{Involutive residuated frames}
\label{subsec: macneille frames}

  Having established the uniqueness of commutative complemented DM completions, the rest of this section is devoted to proving their existence using the \emph{involutive (residuated) frames} of Galatos \& Jipsen~\cite{galatos+jipsen13residuated-frames}. The~definition given below in fact differs from the original definition of an involutive (Gentzen) frame in several respects. However, adapting the results of Galatos \& Jipsen~\cite{galatos+jipsen13residuated-frames} to suit our current needs is a straightforward task, and the results of this subsection are not substantially novel.\footnote{Let us briefly summarize the differences between the two frameworks. Firstly, we take $\alg{A}$ to be a bimonoid rather than a partial involutive residuated lattice (or some more general involutive residuated structure). Secondly, we do not assume that the positive or the negative sides of the frame are generated by $\alg{A}$ as a monoid. Thirdly, we allow the positive and the negative sides of the frame to be distinct sets. Finally, we assume that the monoidal operations on both sides of the frame are single-valued and associative, rather than satisfying the weaker conditions $((a \leftdot b) \leftdot c)\toright = (a \leftdot (b \leftdot c))\toright$ and $((a \rightplus b) \rightplus c)\toleft = (a \rightplus (b \rightplus c))\toleft$. Nothing in our proofs depends on this last assumption, it merely makes our definitions easier to state.}

  Before giving the formal definition of an involutive frame, let us first explain the purpose of introducing these structures. Our goal will be to construct a complete complemented $\ell$-bimonoid given a monoid of join generators $\alg{L} = \langle L, \leftdot, \leftunit \rangle$, a monoid of meet generators $\alg{R} = \langle R, \rightplus, \rightzero \rangle$, and a link between these two monoids. This link consists of a binary relation $x \sqleq y$, which tells us how to compare a join generator $x$ and a meet generator $y$, and mutually inverse monoidal anti-isomorphisms, which tell us how complementation acts on these monoids. Some natural compatibility conditions are of course postulated. The~above data constitute the frame from which we then construct a complete complemented bimonoid.

  Moreover, we wish to embed a given bimonoid $\alg{A}$ into the complete complemented $\ell$-bimonoid constructed from the frame. This embedding can be constructed from a pair of maps $\lambda\colon A \to L$ and $\rho\colon A \to R$ (not necessarily homomorphisms) embedding $\alg{A}$ into $\alg{L}$~and~$\alg{R}$. This yields what we call a involutive $\alg{A}$-frame, shown in Figure~\ref{fig: a-frame}. If the embeddings $\lambda$ and $\rho$ reflect the order in a natural sense, we call the involutive $\alg{A}$-frame faithful. The~key result is that an ($\ell$-)bimonoid $\alg{A}$ embeds into the involutive residuated lattice obtained from every faithful involutive $\alg{A}$-frame, and moreover there is a natural way to build a faithful commutative involutive $\alg{A}$-frame given a commutative bimonoid $\alg{A}$.

  All of the above builds on the Galois connection between a pair of sets $L$ and $R$ induced by a binary relation ${\sqleq} \subseteq L \times R$. Such a triple $\langle L, R, {\sqleq} \rangle$ is often called a polarity. For $X \subseteq L$, $x \in L$ and $Y \subseteq R$, $y \in R$ we use the notation
\begin{align*}
  X \sqleq y & \iff x \sqleq y \text{ for all } x \in X, \\
  x \sqleq Y & \iff x \sqleq y \text{ for all } y \in Y,
\end{align*}
\begin{align*}
  X \toright & \assign \set{y \in R}{X \sqleq y}, & x\toright & \assign \{ x \}\toright, \\
  Y \toleft & \assign \set{x \in L}{x \sqleq Y}, & y\toleft & \assign \{ y \}\toleft.
\end{align*}
  A set $X \subseteq L$ ($Y \subseteq R$) is called \emph{Galois closed} if $X = X\torighttoleft$ (if $Y = Y\tolefttoright$). The Galois closed subsets of $L$~and~$R$ ordered by inclusion form lattices which are anti-isomorphic via the maps $X \mapsto X\toright$ and $Y \mapsto Y\toleft$. Observe that $x \in X\torighttoleft$ if and only if $X \sqleq y$ implies $x \sqleq y$ for all $y \in R$, and $y \in Y\tolefttoright$ if and only if $x \sqleq Y$ implies $x \sqleq y$ for all $x \in L$. Consequently,
\begin{align*}
  X\torighttoleft \sqleq y & \iff X \sqleq y & & \text{and} & x \sqleq Y\tolefttoright & \iff x \sqleq Y.
\end{align*}

\begin{figure}
\caption{An involutive $\alg{A}$-frame}
\label{fig: a-frame}
\begin{center}
\begin{tikzpicture}
  \node[circle,draw,thick,outer sep=0.15cm] (A) at (-4.5, 1.5) {$\alg{A}$};
  \node[ellipse,draw,thick,outer sep=0.15cm] (L) at (0, 0) {$\langle L, \circ, \leftunit \rangle$};
  \node[ellipse,draw,thick,outer sep=0.15cm] (R) at (0, 3) {$\langle R, \rightplus, \rightzero \rangle$};
  \node (sqleq) at (0, 1.5) {$x \sqleq y$};
  \draw[dashed] (L) -- (sqleq) -- (R);
  \draw [->,thick] (L.north west) to [out=135,in=-135] (R.south west);
  \draw [->,thick] (R.south east) to [out=-45,in=45] (L.north east);
  \draw[->,thick] (A.south east) to [out=-45,in=180] (L.west);
  \draw[->,thick] (A.north east) to [out=45,in=-180] (R.west);
  \node at (-2.45, 1.5) {$\ltorleft(x), \ltorright(x)$};
  \node at (2.45, 1.5) {$\rtolleft(y), \rtolright(y)$};
  \node at (-3.25, 0) {$\lambda$};
  \node at (-3.25, 3) {$\rho$};
\end{tikzpicture}
\end{center}
\end{figure}

  This Galois connection is then expanded by a monoidal structure $\langle L, \leftdot, \leftunit \rangle$ on $L$ and a monoidal structure $\langle R, \rightplus, \rightzero \rangle$ on $R$, as well as maps $\ltorleft, \ltorright\colon L \to R$ and $\rtolleft, \rtolright\colon R \to L$. We use the notation
\settowidth{\auxlength}{$X_{1} \leftdot X_{2}$}
\settowidth{\auxlengthtwo}{$y_{1} \rightplus y_{2}$}
\begin{align*}
  \hbox to \auxlength{$X_{1} \leftdot X_{2}$} & \assign \set{\hbox to \auxlengthtwo{$x_{1} \leftdot x_{2}$}}{x_{1} \in X_{1} \text{ and } x_{2} \in X_{2}}, \\
  \hbox to \auxlength{$Y_{1} \rightplus Y_{2}$} & \assign \set{\hbox to \auxlengthtwo{$y_{1} \rightplus y_{2}$}}{y_{1} \in Y_{1} \text{ and } y_{2} \in Y_{2}},
\end{align*}
  with $x_{1} \leftdot X_{2} \assign \{ x_{1} \} \leftdot X_{2}$ and likewise for $X_{1} \leftdot x_{2}$, $x_{1} \rightplus X_{2}$, $X_{1} \rightplus x_{2}$.

\begin{definition}[Involutive frames]
  A \emph{(commutative) involutive frame} is a two-sorted structure consisting of two (commutative) monoids
\settowidth{\auxlength}{$\alg{R}$}
\begin{align*}
  \hbox to \auxlength{\hfil$\alg{L}$\hfil} & = \langle L, \leftdot, \leftunit \rangle, \\
  \hbox to \auxlength{$\alg{R}$} & = \langle R, \rightplus, \rightzero \rangle,
\end{align*}
  two pairs of maps
\begin{align*}
  \ltorleft, \ltorright&\colon L \rightarrow R, \\
  \rtolleft, \rtolright&\colon R \rightarrow L,
\end{align*}
  and a relation ${\sqleq} \subseteq L \times R$ which satisfies the following form of residuation called \emph{nuclearity}:
\settowidth{\auxlength}{$x \sqleq z \rightplus \ltorleft(y)$}
\settowidth{\auxlengthtwo}{$x \sqleq y \rightplus z$}
\settowidth{\auxlengththree}{$y \sqleq \ltorright(x) \rightplus z$}
\begin{align*}
  \hbox to \auxlength{$x \sqleq z \rightplus \ltorleft(y)$} \iff \hbox to \auxlengthtwo{$x \leftdot y \sqleq z$} \iff \hbox to \auxlengththree{$y \sqleq \ltorright(x) \rightplus z$}, \\
  \hbox to \auxlength{$x \leftdot \rtolright(z) \sqleq y$} \iff \hbox to \auxlengthtwo{$x \sqleq y \rightplus z$} \iff \hbox to \auxlengththree{$\rtolleft(y) \leftdot x \sqleq y$}.
\end{align*}
\end{definition}

\begin{lemmanoname} \label{lemma: nuclearity}
  In each involutive frame we have
\begin{align*}
  X_{1}\torighttoleft \leftdot X_{2}\torighttoleft & \subseteq (X_{1} \circ X_{2})\torighttoleft & & \text{and} & Y_{1}\tolefttoright \rightplus Y_{2}\tolefttoright & \subseteq (Y_{1} \rightplus Y_{2})\tolefttoright
\end{align*}
  for all ${X_{1}, X_{2} \subseteq L}$ and $Y_{1}, Y_{2} \subseteq R$. Consequently,
\begin{align*}
  (X_{1}\torighttoleft \leftdot X_{2})\toright & = (X_{1} \leftdot X_{2})\toright = (X_{1} \leftdot X_{2}\torighttoleft)\toright & & \text{and} & (Y_{1}\tolefttoright \rightplus Y_{2})\toleft & = (Y_{1} \rightplus Y_{2})\toleft = (Y_{1} \rightplus Y_{2}\tolefttoright)\toleft.
\end{align*}
\end{lemmanoname}

\begin{proof}
  We only prove the claims for multiplication. The claims for addition then follow if we take the polarity $\langle R, L, \sqgeq \rangle$ instead of $\langle L, R \sqleq \rangle$. The inclusion $X_{1}\torighttoleft \leftdot X_{2}\torighttoleft \subseteq (X_{1} \circ X_{2})\torighttoleft$ states that if $x_{1} \in X_{1}\torighttoleft$, $x_{2} \in X_{2}\torighttoleft$, and $X_{1} \leftdot X_{2} \sqleq y$, then $x_{1} \leftdot x_{2} \sqleq y$. But $X_{1} \leftdot X_{2} \sqleq y$ implies $X_{2} \sqleq \ltorright[X_{1}] \rightplus y$, hence $x_{2} \in X_{2}\torighttoleft \sqleq \ltorright[X_{1}] \rightplus y$ and $X_{1} \leftdot x_{2} \sqleq y$. Similarly, this implies $X_{1} \sqleq y \rightplus \ltorleft[X_{2}]$, hence $x_{1} \in X_{1}\torighttoleft \sqleq y \rightplus \ltorleft[X_{2}]$ and $x_{1} \leftdot x_{2} \sqleq y$.

  The other equalities now follow: $X_{1} \subseteq X_{1}\torighttoleft$, hence $(X_{1}\torighttoleft \leftdot X_{2})\toright \subseteq (X_{1} \leftdot X_{2})\toright$. Conversely, $X_{1}\torighttoleft \leftdot X_{2} \subseteq X_{1}\torighttoleft \leftdot X_{2}\torighttoleft \subseteq (X_{1} \leftdot X_{2})\torighttoleft$, therefore $(X_{1}\torighttoleft \leftdot X_{2})\toright \supseteq ((X_{1} \leftdot X_{2})\torighttoleft)\toright = (X_{1} \leftdot X_{2})\toright$.
\end{proof}

\begin{lemma}[Hemidistributivity for involutive frames] \label{lemma: frame hemidistributivity}
  In an involutive frame for each $X \subseteq L$ and $Y \subseteq R$ we have:
\begin{align*}
  u \leftdot X \sqleq x ~ \& ~ v \sqleq X\toright \rightplus y & \implies u \leftdot v \sqleq x \rightplus y, & u \leftdot Y\toleft \sqleq x ~ \& ~ v \sqleq Y \rightplus y \implies u \leftdot v \sqleq x \rightplus y, \\
  u \sqleq x \rightplus X\toright ~ \& ~ X \leftdot v \sqleq y & \implies u \leftdot v \sqleq x \rightplus y, & u \sqleq x \rightplus Y ~ \& ~ Y\toleft \leftdot v \sqleq y \implies u \leftdot v \sqleq x \rightplus y.
\end{align*}
\end{lemma}

\begin{proof}
  If $u \leftdot X \sqleq x$ and $v \sqleq X\toright \rightplus y$, then $X \sqleq \ltorright(u) \rightplus x$ and $v \leftdot \rtolright(y) \sqleq X\toright$, so $v \leftdot \rtolright(y) \sqleq \ltorright(u) \rightplus x$ and $u \leftdot v \sqleq x \rightplus y$. The other implications are analogous.
\end{proof}

  These implications can be seen as a form of the multiple-premise and multiple-conclusion Cut rule if we interpret the condition $u \leftdot v \sqleq x \rightplus y$ as expressing the provability of the sequent $u, v \vdash x, y$.

   To embed an ($\ell$-)bimonoid $\alg{A}$ into the algebra constructed from an involutive frame, we need to postulate some additional structure: a pair of maps $\lambda\colon A \to L$ and $\rho\colon A \to R$ which satisfy the conditions in Figure \ref{fig: gentzen conditions}. These will be called the \emph{($\ell$-)bimonoidal Gentzen conditions}, depending on whether we include the conditions on meets and joins, due to their similarity to the logical rules in Gentzen calculi. They~are to be interpreted as universally quantified implications: e.g.\ for all $a, b \in \alg{A}$ and $x, y \in R$ if $\lambda(a) \sqleq x$ and $\lambda(b) \sqleq y$, then $\lambda(a + b) \sqleq x \rightplus y$. The maps $\lambda, \rho$ are not required to be homomorphisms and the maps $\ltorleft, \ltorright, \rtolleft, \rtolright$ are not required to be anti-homomorphisms, although in the frames constructed in the current paper they will be.

\begin{definition}[Involutive $\alg{A}$-frames]
  Let $\alg{A}$ be an ($\ell$-)bimonoid. An \emph{involutive $\alg{A}$-frame} is then an involutive frame equipped with two maps $\lambda\colon A \to L$ and $\rho\colon A \to R$ which satisfy the ($\ell$-)bimonoidal Gentzen conditions as well as Identity and Cut:
\begin{prooftree}
  \def\fCenter{\sqleq}
  \AxiomC{}
  \noLine
  \UnaryInf$\lambda (a) \fCenter \rho (a)$
  \Axiom$x \fCenter \rho (a)$
  \Axiom$\lambda (a) \fCenter y$
  \BinaryInf$x \fCenter y$
  \noLine
  \insertBetweenHyps{\hskip 40pt}
  \BinaryInfC{}
\end{prooftree}
  An ($\ell$-)involutive $\alg{A}$-frame is called \emph{faithful} if $\lambda (a) \sqleq \rho (b)$ implies $a \leq b$ for all $a, b \in \alg{A}$.  
\end{definition}

\begin{figure}
\caption{Gentzen conditions}
\label{fig: gentzen conditions}
\begin{framed}
\begin{prooftree}
  \def\fCenter{\sqleq}
  \AxiomC{}
  \noLine
  \UnaryInf$\leftunit \fCenter \rho (1)$
  \Axiom$\leftunit \fCenter y$
  \UnaryInf$\lambda (1) \fCenter y$
  \Axiom$x \fCenter \rightzero$
  \UnaryInf$x \fCenter \rho(0)$
  \AxiomC{}
  \noLine
  \UnaryInf$\lambda (0) \fCenter \rightzero$
  \noLine
  \QuaternaryInfC{}
\end{prooftree}
\begin{prooftree}
  \def\fCenter{\sqleq}
  \Axiom$\lambda (a) \fCenter x$
  \Axiom$\lambda (b) \fCenter y$
  \BinaryInf$\lambda (a + b) \fCenter x \rightplus y$
  \Axiom$x \fCenter \rho (a)$
  \Axiom$y \fCenter \rho (b)$
  \BinaryInf$x \leftdot y \fCenter \rho (a \cdot b)$
  \noLine
  \BinaryInfC{}
\end{prooftree}
\begin{prooftree}
  \def\fCenter{\sqleq}
  \Axiom$\lambda (a) \leftdot \lambda (b) \fCenter x$
  \UnaryInf$\lambda (a \cdot b) \fCenter x$
  \Axiom$x \fCenter \rho (a) \rightplus \rho (b)$
  \UnaryInf$x \fCenter \rho (a + b)$
  \noLine
  \BinaryInfC{}
\end{prooftree}
\begin{prooftree}
  \def\fCenter{\sqleq}
  \Axiom$\lambda (a) \fCenter x$
  \Axiom$\lambda (b) \fCenter x$
  \BinaryInf$\lambda (a \vee b) \fCenter x$
  \Axiom$x \fCenter \rho (a)$
  \Axiom$x \fCenter \rho (b)$
  \BinaryInf$x \fCenter \rho (a \wedge b)$
  \noLine
  \BinaryInfC{}
\end{prooftree}
\begin{prooftree}
  \def\fCenter{\sqleq}
  \Axiom$\lambda (a) \fCenter x$
  \UnaryInf$\lambda (a \wedge b) \fCenter x$
  \Axiom$\lambda (b) \fCenter x$
  \UnaryInf$\lambda (a \wedge b) \fCenter x$
  \Axiom$x \fCenter \rho (a)$
  \UnaryInf$x \fCenter \rho (a \vee b)$
  \Axiom$x \fCenter \rho (b)$
  \UnaryInf$x \fCenter \rho (a \vee b)$
  \noLine
  \QuaternaryInfC{}
\end{prooftree}
\end{framed}
\end{figure}

\begin{definition}[The Galois algebra of an involutive frame]
  The \emph{Galois algebra} $\galois{\invframe{F}}$ of an involutive frame $\invframe{F}$ consists of the Galois closed subsets of $L$ equipped with the following operations:
\begin{align*}
  1 & \assign \{ \leftunit \} \torighttoleft, & X_{1} \cdot X_{2} & \assign (X_{1} \leftdot X_{2})\torighttoleft, & X_{1} \wedge X_{2} & \assign X_{1} \cap X_{2}, & \compl{X} & \assign (\ltorleft[X])\toleft, \\
  0 & \assign \{ \rightzero \} \toleft, & X_{1} + X_{2} & \assign (X_{1}\toright \rightplus X_{2}\toright)\toleft, & X_{1} \vee X_{2} & \assign (X_{1} \cup X_{2})\torighttoleft, & \compr{X} & \assign (\ltorright[X])\toleft.
\end{align*}
\end{definition}

  We now show that the Galois algebra of an involutive frame is a complete complemented $\ell$-bimonoid, and moreover $\alg{A}$ embeds into the Galois algebra of a faithful involutive $\alg{A}$-frame. A map $\lambda$ satisfying the conditions of the following lemma was called a quasi-homomorphism in~\cite{bellardinelli+jipsen+ono04}.

\begin{lemma}[Quasi-homomorphism Lemma] \label{lemma: quasi-homomorphisms}
  In each involutive $\alg{A}$-frame
\settowidth{\auxlength}{$\{ \leftunit \}\torighttoleft$}
\begin{align*}
  \lambda (1) \in \hbox to \auxlength{\hfil$\{ \leftunit \} \torighttoleft$\hfil} \sqleq \rho (1), \\
  \lambda (0) \in \hbox to \auxlength{\hfil$\{ \rightzero \} \toleft$\hfil} \sqleq \rho (0).
\end{align*}
  If $X$ and $Y$ are Galois closed subsets of $L$ such that
\settowidth{\auxlength}{$X$}
\settowidth{\auxlengthtwo}{$\lambda (a)$}
\begin{align*}
  \hbox to \auxlengthtwo{\hfil$\lambda (a)$\hfil} \in \hbox to \auxlength{\hfil$X$\hfil} & \sqleq \rho (a),\\
  \hbox to \auxlengthtwo{\hfil$\lambda (b)$\hfil} \in \hbox to \auxlength{\hfil$Y$\hfil} & \sqleq \rho (b),
\end{align*}
  then we have
\settowidth{\auxlength}{$\lambda(a + b)$}
\settowidth{\auxlengthtwo}{$(X \cup Y)\torighttoleft$}
\settowidth{\auxlengththree}{$(X\toright \rightplus Y\toright)\toleft$}
\newlength{\auxlengthfour}
\settowidth{\auxlengthfour}{$\rho(a + b)$}
\begin{align*}
  \hbox to \auxlength{\hfill$\lambda (a \cdot b)$} \in \hbox to \auxlengththree{\hbox to \auxlengthtwo{$(X \leftdot Y)\torighttoleft$}\hfill} \sqleq \hbox to\auxlengthfour{$\rho(a \cdot b)$,\hfill}\\
  \hbox to \auxlength{$\lambda (a + b)$} \in \hbox to \auxlengththree{$(X\toright \rightplus Y\toright)\toleft$} \sqleq \hbox to \auxlengthfour{$\rho(a + b)$,} \\
  \hbox to \auxlength{$\lambda (a \vee b)$} \in \hbox to \auxlengththree{\hbox to \auxlengthtwo{$(X \cup Y)\torighttoleft$}\hfill} \sqleq \hbox to \auxlengthfour{$\rho(a \vee b)$,} \\
  \hbox to \auxlength{$\lambda (a \wedge b)$} \in \hbox to \auxlengththree{$X \cap Y$\hfill} \sqleq \hbox to \auxlengthfour{$\rho(a \wedge b)$.}
\end{align*}
\end{lemma}

\begin{proof}
  We divide the Gentzen conditions in Figure~\ref{fig: gentzen conditions} in left and right rules, depending on whether the conclusion contains the map $\lambda$ or the map $\rho$. For example, by the left rule for addition we mean the Gentzen condition whose conclusion is $\lambda(a+b) \sqleq x \rightplus y$. Observe that in order to prove that $Z\torighttoleft \sqleq y$ it suffices to prove that $Z \sqleq y$: if $x \in Z\torighttoleft$ and $Z \sqleq y$, then $x \sqleq Z\toright$ and $y \in Z\toright$, hence $x \sqleq y$.

  The claim that $\lambda(1) \in \{ \leftunit \}\torighttoleft$ is precisely the left rule for~$1$. The right rule states that $\leftunit \sqleq \rho(1)$, therefore, as observed above, $\{ \leftunit \}\torighttoleft \sqleq \rho(1)$. The~proof that $\lambda(0) \in \{ \rightzero \} \toleft \sqleq \rho(0)$ is analogous.

  If $X \leftdot Y \sqleq z$, then $\lambda (a) \leftdot \lambda (b) \sqleq z$, hence $\lambda (a \cdot b) \sqleq z$ by the left rule for multiplication. Thus $\lambda (a \cdot b) \in (X \cdot Y)\toright{}\toleft$. For every $x \in X$ and $y \in Y$ we have $x \leftdot y \sqleq \rho(a \cdot b)$ by the right rule for multiplication, thus $X \leftdot Y \sqleq \rho(a \cdot b)$. The~proof that $\lambda (a + b) \in (X\toright \rightplus Y\toright)\toleft \sqleq \rho (a + b)$ is analogous.

  If $X \cup Y \sqleq z$, then $\lambda (a) \sqleq z$ and $\lambda (b) \sqleq z$, hence by the left rule for joins $\lambda (a \vee b) \sqleq z$. Thus $\lambda(a \vee b) \in (X \cup Y) \torighttoleft$. If $X \sqleq \rho(a)$ and $Y \sqleq \rho(b)$, then by the right rule for joins $X \sqleq \rho(a \vee b)$ and $Y \sqleq \rho(a \vee b)$, hence $X \cup Y \sqleq \rho(a \vee b)$ and $(X \cup Y)\toright{}\toleft \sqleq \rho(a \vee b)$. The proof of $\lambda (a \wedge b) \in X \cap Y \sqleq \rho (a \wedge b)$ is analogous.
\end{proof}

\begin{theorem}[Galois algebras are involutive residuated lattices] \label{thm: galois algebras}
  Let $\invframe{F}$ be a (commutative) involutive frame and $\alg{A}$ be an ($\ell$-)bimonoid. Then the Galois algebra $\galois{\invframe{F}}$ is a complete (commutative) involutive residuated lattice. If $\invframe{F}$ is moreover an involutive $\alg{A}$-frame, then the map $a \mapsto \rho (a)\toleft = \lambda(a)\torighttoleft$ is a homomorphism from $\alg{A}$ to $\galois{\invframe{F}}$. It is an embedding if $\invframe{F}$ is faithful.
\end{theorem}

\begin{proof}
  In the following $X, Y, Z \subseteq L$ will be Galois closed sets. We first prove that multiplication in $\galois{\invframe{F}}$ is associative. The proof that addition is associative is analogous. By~Lemma~\ref{lemma: nuclearity} we have
\begin{align*}
  (X \cdot Y) \cdot Z & = ((X \leftdot Y)\torighttoleft \leftdot Z)\torighttoleft \\
  & = ((X \leftdot Y) \leftdot Z)\torighttoleft{}\torighttoleft \\
  & = (X \leftdot (Y \leftdot Z))\torighttoleft \\
  & = (X \leftdot (Y \leftdot Z)\torighttoleft)\torighttoleft \\
  & = X \cdot (Y \cdot Z).
\end{align*}

  Next, we prove that $1$ is a unit element with respect to multiplication. The proof that $0$ is a unit element with respect to addition is analogous. Again by Lemma~\ref{lemma: nuclearity}
\begin{align*}
  X \cdot 1 = (X \leftdot \{ \leftunit \} \torighttoleft) \torighttoleft = (X \leftdot \leftunit) \torighttoleft = X \torighttoleft = X, \\
  1 \cdot X = (\{ \leftunit \} \torighttoleft \leftdot X) \torighttoleft = (\leftunit \leftdot X) \torighttoleft = X \torighttoleft = X,
\end{align*}

  To prove the inclusion $\compl{X} \cdot X \subseteq 0$, observe that
\begin{align*}
  \compl{X} \cdot X \subseteq 0 & \iff ((\ltorleft [X])\toleft \leftdot X) \torighttoleft \subseteq \{ \rightzero \} \toleft \\
  & \iff (\ltorleft [X])\toleft \leftdot X \subseteq \{ \rightzero \} \toleft \\
  & \iff (\ltorleft [X])\toleft \leftdot X \sqleq \rightzero \\
  & \iff \left( x_{1} \sqleq \ltorleft[X] \text{ and } x_{2} \in X \right) \text{ implies } x_{1} \leftdot x_{2} \sqleq \rightzero.
\end{align*}
  But the last implication holds by nuclearity, since
\begin{align*}
  x_{1} \sqleq \ltorleft[X] \iff x_{1} \sqleq \rightzero \rightplus \ltorleft[X] \iff x_{1} \leftdot X \sqleq \rightzero.
\end{align*}
  To prove the inclusion $1 \subseteq X + \compl{X}$, by Lemma~\ref{lemma: nuclearity} we have
\begin{align*}
  1 \subseteq X + \compl{X} & \iff \{ \leftunit \} \torighttoleft \subseteq (X\toright \rightplus (\ltorleft [X])\tolefttoright)\toleft \\
  & \iff \{ \leftunit \} \subseteq (X\toright \rightplus (\ltorleft [X])\tolefttoright)\toleft \\
  & \iff \{ \leftunit \} \subseteq (X\toright \rightplus \ltorleft [X])\toleft \\
  & \iff\leftunit \sqleq X\toright \rightplus \ltorleft [X] \\
  & \iff \leftunit \leftdot X \sqleq X\toright \\
  & \iff X \sqleq X\toright.
\end{align*}
  The proofs of the inclusions $X \cdot \compr{X} \subseteq 0$ and $1 \subseteq \compr{X} + X$ are entirely analogous.

  It remains to prove hemidistributivity. We have
\begin{align*}
  X \cdot (Y + Z) \subseteq (X \cdot Y) + Z & \iff X \leftdot (Y \toright \rightplus Z \toright) \toleft \subseteq ((X \leftdot Y) \toright \rightplus Z \toright) \toleft \\
  & \iff X \leftdot (Y \toright \rightplus Z \toright) \toleft \sqleq (X \leftdot Y) \toright \rightplus Z \toright.
\end{align*}
  Since $(X \leftdot Y)\toright = \bigcap_{x \in X} (x \leftdot Y)\toright$ and $(Y\toright \rightplus Z\toright)\toleft = \bigcap_{z \in Z\toright} (Y\toright \rightplus z)\toleft$, it suffices to prove that $x \circ (Y\toright \rightplus z)\toleft \sqleq (x \leftdot Y)\toright \rightplus z$ for all $x \in L$ and $z \in R$. But this follows if we take $X \assign Y$, $y \assign z$, and $u \assign x$ in the first condition of Lemma~\ref{lemma: frame hemidistributivity}.

  The above proves that the Galois algebra $\galois{\invframe{F}}$ is a complemented $\ell$-bimonoid. It is complete because the join of each family $X_{i} \in \galois{\invframe{F}}$ for $i \in I$ is $\left( \bigcup_{i \in I} X_{i} \right)\torighttoleft$. Moreover, if the monoids $\alg{L}$ and $\alg{R}$ are commutative, then so is $\galois{\invframe{F}}$. Now suppose that $\invframe{F}$ is an $\alg{A}$-frame. Then $\rho(a)\toleft = \lambda(a)\torighttoleft$ by Identity and Cut. To prove that the map ${a \mapsto \rho(a)\toleft}$ is a homomorphism from $\alg{A}$ to $\galois{\invframe{F}}$, by Lemma~\ref{lemma: quasi-homomorphisms} it suffices to show that in an involutive $\alg{A}$-frame the only Galois closed set $X$ such that $\lambda (a) \in X \sqleq \rho (a)$ is~$\rho (a)\toleft$.

  Suppose that $\lambda (a) \in X \sqleq \rho (a)$ for $X \subseteq L$ Galois closed. Then $X \subseteq \rho(a) \toleft$. Proving that $\rho (a) \toleft \subseteq X$ amounts to proving that $x \sqleq \rho (a)$ implies $x \in X = X \torighttoleft$, i.e.\ to proving the implication $X \sqleq y \implies x \sqleq y$. But $X \sqleq y$ implies $\lambda (a) \sqleq y$ and applying Cut to $x \sqleq \rho (a)$ and $\lambda (a) \sqleq y$ yields that $x \sqleq y$.

  Finally, let $\invframe{F}$ be a faithful involutive $\alg{A}$-frame. Then $\rho (a)\toleft \subseteq \rho (b)\toleft$ implies that $\lambda (a) \in \rho (b)\toleft$, i.e.\ $\lambda (a) \sqleq \rho (b)$, hence $a \leq b$ by faithfulness.
\end{proof}

\newcommand{\leftcong}{\theta_{L}}
\newcommand{\rightcong}{\theta_{R}}

  Observe that in each involutive frame $F$ we may define the following equivalence relations on $\alg{L}$ and $\alg{R}$:
\begin{align*}
  \pair{x}{x'} \in \leftcong & \text{ if and only if } (x \sqleq y \iff x' \sqleq y) \text{ for each } y \in R, \\
  \pair{y}{y'}  \in \rightcong & \text{ if and only if } (x \sqleq y \iff x \sqleq y') \text{ for each } x \in L.
\end{align*}
  These are in fact congruences on $F$. More precisely, if $\pair{x_{1}}{x'_{1}} \in \leftcong$, $\pair{x_{2}}{x'_{2}} \in \leftcong$, and $\pair{y}{y'} \in \rightcong$, then $\pair{x_{1} \leftdot x_{2}}{x'_{1} \leftdot x'_{2}} \in \leftcong$ and $\pair{\ltorleft(x_{1})}{\ltorleft(x'_{1})}, \pair{\ltorright(x_{1})}{\ltorright(x'_{1})} \in \rightcong$ (and likewise with $L$ and $R$ exchanged). Moreover, $x \sqleq y$ if and only if $x' \sqleq y'$. It follows that the Galois algebra of the involutive frame $F$ is isomorphic to the Galois algebra of the involutive frame $F / \theta$ consisting of $\alg{L} / \leftcong$ and $\alg{R} / \rightcong$ connected by ${\sqleq}$ and $\ltorright, \ltorleft, \rtolright, \rtolleft$ via the map $X \mapsto \set{[x]_{\leftcong}}{x \in X}$.

\subsection{Complemented MacNeille completions: existence}
\label{subsec: macneille existence}

  We now use the tools introduced in the previous subsection to construct the commutative complemented DM completion of an arbitrary commutative bimonoid $\alg{A}$ as the Galois algebra of an involutive $\alg{A}$-frame.

  We define the structure $\invframeofalg{A}$, which we claim to be a faithful commutative involutive $\alg{A}$-frame, as follows. The monoids $\alg{L}$~and~$\alg{R}$ both have the same universe $A^{2}$, but elements of $\alg{L}$ are denoted $\leftpair{a}{b}$, while elements of $\alg{R}$ are denoted $\rightpair{a}{b}$. These are intended to correspond to $a \cdot \comp{b}$ and $a + \comp{b}$. For such pairs we define the monoidal operations
\settowidth{\auxlength}{$\rightpair{a}{b} \rightplus \rightpair{c}{d}$}
\settowidth{\auxlengthtwo}{$\rightpair{a + c}{b \cdot d}$}
\begin{align*}
  \hbox to \auxlength{$\leftpair{a}{b} \leftdot \leftpair{c}{d}$} & \assign \hbox to \auxlengthtwo{$\leftpair{a \cdot c}{b + d}$}, \\
  \hbox to \auxlength{$\rightpair{a}{b} \rightplus \rightpair{c}{d}$} & \assign \hbox to \auxlengthtwo{$\rightpair{a + c}{b \cdot d}$},
\end{align*}
  with units
\begin{align*}
  \leftunit & \assign \leftpair{1}{0}, \\
  \rightzero & \assign \rightpair{0}{1},
\end{align*}
  and the maps
\begin{align*}
  \ltorleft (\leftpair{a}{b}) & = \rightpair{b}{a} = \ltorright (\leftpair{a}{b}), \\
  \rtolleft (\rightpair{a}{b}) & = \leftpair{b}{a} = \rtolright (\rightpair{a}{b}).
\end{align*}
  The relation connecting the two monoids is defined as
\begin{align*}
  \leftpair{a}{b} \sqleq \rightpair{c}{d} & \iff a \cdot d \leq_{\alg{A}} b + c.
\end{align*}
  If we interpret $\leftpair{a}{b}$ and $\rightpair{c}{d}$ in the intended way, this equivalence is precisely what residuation yields. Finally, $\alg{A}$ embeds into $\alg{L}$ and $\alg{R}$ via the maps
\begin{align*}
  \lambda (a) & \assign \leftpair{a}{0}, \\
  \rho (a) & \assign \rightpair{a}{1}.
\end{align*}
  We now verify that this yields a faithful commutative involutive $\alg{A}$-frame.

\begin{lemma}[Involutive frames from bimonoids] \label{lemma: frames from bimonoids}
  $\invframeofalg{A}$ is a faithful commutative involutive $\alg{A}$-frame.
\end{lemma}

\begin{proof}
  Checking that $\invframeofalg{A}$ is a commutative monoidal pair is straightforward. Nuclearity states that $\leftpair{a}{b} \leftdot \leftpair{c}{d} = \leftpair{a \cdot c}{b + d} \sqleq \rightpair{e}{f}$ is equivalent to $\leftpair{a}{b} \sqleq \rightpair{d + e}{c \cdot f} = \rightpair{d}{c} \rightplus \rightpair{e}{f}$. But this holds, as both of these conditions amount to $a \cdot c \cdot f \leq b + d + e$. Moreover, $\lambda (a) = \leftpair{a}{0} \sqleq \rightpair{b}{1} = \rho (b)$ implies that $a = a \cdot 1 \leq b + 0 = b$, therefore $\invframeofalg{A}$ is faithful. It remains to verify that the Gentzen conditions of Figure \ref{fig: gentzen conditions} hold in $\invframeofalg{A}$. We only deal with the left rules, since the right rules follow the same pattern as the left rules for the dual connective.

  The left rules for $0$ and $1$ clearly hold. To prove the left rule for multiplication, observe that $\lambda (a) \leftdot \lambda (b) = \leftpair{a}{0} \leftdot \leftpair{b}{0} = \leftpair{a \cdot b}{0} = \lambda (a \cdot b)$. To prove the left rule for join, observe that $\lambda (a) \sqleq \rightpair{c}{d}$ and $\lambda (b) \sqleq \rightpair{c}{d}$ imply that $a \cdot d \leq c$ and $b \cdot d \leq c$. By the distributivity of products over joins we have $(a \vee b) \cdot d \leq c$ and $\lambda (a \vee b) \sqleq \rightpair{c}{d}$. The left rule for meet follows from the mono\-tonicity of multiplication, i.e.\ $a \cdot d \leq c$ implies $(a \wedge b) \cdot d \leq c$, as does $b \cdot d \leq c$. Finally, we prove the left rule for addition. If $\lambda (a) \sqleq \rightpair{c}{d}$, then $a \cdot d \leq c$. If $\lambda (b) \sqleq \rightpair{e}{f}$, then $b \cdot f \leq e$. It follows by hemidistributivity that $(a + b) d f \leq a d + b f \leq c + e$, i.e. $\lambda(a + b) = \leftpair{a + b}{0} \sqleq \rightpair{c + e}{d \cdot f} = \rightpair{c}{d} \rightplus \rightpair{e}{f}$.
\end{proof}

\begin{theorem}[Complemented commutative DM completions exist] \label{thm: dm completions exist}
  Each commutative ($\ell$-)bimonoid $\alg{A}$ has a commutative complemented DM completion, namely the Galois algebra $\galois{\invframeofalg{A}}$ of the involutive $\alg{A}$-frame $\invframeofalg{A}$.
\end{theorem}

\begin{proof}
  The Galois algebra of $\invframeofalg{A}$ is by Theorem~\ref{thm: galois algebras} and Lemma~\ref{lemma: frames from bimonoids} a complete involutive commutative residuated lattice into which~$\alg{A}$ embeds via the map $a \mapsto \rho (a)\toleft = \lambda (a)\torighttoleft$. It~remains to prove that elements of the form
\begin{align*}
  \lambda(a)\torighttoleft \cdot \comp{\lambda(b)\torighttoleft} = \lambda(a)\torighttoleft \cdot \ltor[\lambda(b)\torighttoleft]\toleft
\end{align*}
  are join dense in~$\invframeofalg{A}$. Because the elements $\leftpair{a}{b}\torighttoleft$ are join dense in $\invframeofalg{A}$ and
\begin{align*}
  \leftpair{a}{b}\torighttoleft = (\leftpair{a}{0} \leftdot \leftpair{1}{b})\torighttoleft = \left( \lambda (a) \leftdot \rtol (\rho (b)) \right) \torighttoleft = \lambda (a)\torighttoleft \cdot \rtol(\rho(b))\torighttoleft,
\end{align*}
  it suffices to prove that $\ltor [\lambda(b)\torighttoleft] = \rtol(\rho(b))\toright$. But this holds because
\begin{align*}
  y \in \rtol(\rho(b))\toright & \iff \rtol(\rho(b)) \sqleq y \\
  & \iff \leftunit \sqleq y \rightplus \rho(b) \\
  & \iff \rtol(y) \sqleq \rho(b) \\
  & \iff \rtol(y) \in \rho(b)\toleft \\
  & \iff \rtol(y) \in \lambda(b)\torighttoleft \\
  & \iff y \in \ltor[\lambda(b)\torighttoleft],
\end{align*}
  using the fact that $\rtol(y) \in X \iff y \in \ltor[X]$, by the definitions of the maps $\rtol$ and $\ltor$.
\end{proof}

  The existence of the commutative complemented DM completion immediately yields a bimonoidal version of Funayama's theorem for distributive lattices~\cite{funayama59}. This theorem, in its stronger form~\cite{bezhanishvili+gabelaia+jibladze13}, states that a distributive lattice has an embedding into a complete Boolean algebra which preserves all existing joins if and only if the distributive lattice satisfies the join-infinite distributive law for all existing joins.\footnote{We thank Guram Bezhanishvili for bringing this theorem to our attention.}

\begin{theorem}[Funayama's theorem for bimonoids] \label{thm: funayama}
  Let $\iota\colon \alg{A} \into \cdmalg{A}$ be a commutative complemented DM completion of a commutative bimonoid $\alg{A}$. Then the following are equivalent:
\begin{enumerate}[(i)]
\item each join in $\alg{A}$ is admissible,
\item the embedding $\iota$ preserves all existing joins,
\item some embedding of $\alg{A}$ into a commutative complete complemented bimonoid preserves all existing joins.
\end{enumerate}
\end{theorem}

\begin{proof}
  If an embedding of $\alg{A}$ into a commutative complete complemented bimonoid $\alg{B}$ preserves the join $\bigvee X$, then $\bigvee X$ is admissible in $\alg{B}$ (since each join is admissible in an involutive residuated lattice), and therefore also in $\alg{A}$. Conversely, $\iota$ preserves each admissible join which exists in $\alg{A}$ by Fact~\ref{fact: admissible joins preserved}, therefore if each join in $\alg{A}$ is admissible, then each join in $\alg{A}$ is preserved by $\iota$.
\end{proof}

\begin{figure}
\caption{The algebras $\Lukthree$ and $\Lukthreeext$}
\label{fig: lukthree}
\begin{center}
\begin{tikzpicture}[scale=1,dot/.style={circle,fill,inner sep=2.5pt,outer sep=2.5pt}]
  \node (0) at (0, 0) [dot] {};
  \node (1/2) at (0, 1.5) [dot] {};
  \node (1) at (0, 3) [dot] {};
  \draw[-,thick] (0) -- (1/2) -- (1);
  \node (n1) at (0.5, 3) {$\elemone$};
  \node (n1/2) at (0.5, 1.5) {$\elem{a}$};
  \node (n0) at (0.5, 0) {$\elem{b}$};
  \node at (0, -1) {$\Lukthree$};
\end{tikzpicture}
\qquad
\begin{tikzpicture}[scale=1,dot/.style={circle,fill,inner sep=2.5pt,outer sep=2.5pt}]
  \node (b) at (0, 0) [dot] {};
  \node (a) at (0, 1) [dot] {};
  \node (atimesnega) at (0, 2) [dot] {};
  \node (1) at (-1, 3) [dot] {};
  \node (sideways) at (1, 3) [dot] {};
  \node (aplusnega) at (0, 4) [dot] {};
  \node (nega) at (0, 5) [dot] {};
  \node (negb) at (0, 6) [dot] {};
  \node[align=left] (nb) at (0.5, 0) {$\elem{b}$};
  \node[align=left] (na) at (0.5, 1) {$\elem{a}$};
  \node (natimesnega) at (0.5, 2) {$\elem{a} \comp{\elem{a}}$};
  \node (n1) at (-1.5, 3) {$\elemone$};
  \node[align=left] (nsideways) at (1.5, 3) {$\elem{a} \comp{\elem{b}}$};
  \node (naplusnega) at (0.85, 4) {$\elemone \vee \elem{a} \comp{\elem{b}}$};
  \node (nnega) at (0.5, 5) {$\comp{\elem{a}}$};
  \node (nnegb) at (0.5, 6) {$\comp{\elem{b}}$};
  \draw[-,thick] (b) -- (a) -- (atimesnega) -- (1) -- (aplusnega) -- (nega) -- (negb);
  \draw[-,thick] (atimesnega) -- (sideways) -- (aplusnega);
  \node at (0, -1) {$\Lukthreeext$ (joins of $x \comp{y}$)};
\end{tikzpicture}
\qquad
\begin{tikzpicture}[scale=1,dot/.style={circle,fill,inner sep=2.5pt,outer sep=2.5pt}]
  \node (b) at (0, 0) [dot] {};
  \node (a) at (0, 1) [dot] {};
  \node (atimesnega) at (0, 2) [dot] {};
  \node (1) at (-1, 3) [dot] {};
  \node (sideways) at (1, 3) [dot] {};
  \node (aplusnega) at (0, 4) [dot] {};
  \node (nega) at (0, 5) [dot] {};
  \node (negb) at (0, 6) [dot] {};
  \node[align=left] (nb) at (0.5, 0) {$\elem{b}$};
  \node[align=left] (na) at (0.5, 1) {$\elem{a}$};
  \node (natimesnega) at (1.25, 2) {$\elemone \wedge (\elem{b} + \comp{\elem{a}})$};
  \node (n1) at (-1.5, 3) {$\elemone$};
  \node[align=left] (nsideways) at (1.725, 3) {$\elem{b} + \comp{\elem{a}}$};
  \node (naplusnega) at (0.8, 4) {$\elem{a} + \comp{\elem{a}}$};
  \node (nnega) at (0.5, 5) {$\comp{\elem{a}}$};
  \node (nnegb) at (0.5, 6) {$\comp{\elem{b}}$};
  \draw[-,thick] (b) -- (a) -- (atimesnega) -- (1) -- (aplusnega) -- (nega) -- (negb);
  \draw[-,thick] (atimesnega) -- (sideways) -- (aplusnega);
  \node at (0, -1) {$\Lukthreeext$ (meets of $x + \comp{y}$)};
\end{tikzpicture}
\end{center}
\end{figure}

  Let us now illustrate how one can find the complemented DM completion of a small bimonoid. Consider the multiplicative reduct of the three-element MV-chain $\elemone > \elem{a} > \elem{b}$ viewed as a bimonoid. That is, $\elem{a} \cdot \elem{a} = \elem{b}$, $x \cdot \elem{b} = \elem{b} = \elem{b} \cdot x$, $x \cdot \elemone = x = \elemone \cdot x$, and $x + y \assign x \cdot y$. Let us call this bimonoid $\Lukthree$. Although multiplication and addition coincide in $\Lukthree$, they come apart in $\Lukthreeext$. In other words, even though the construction of expanding a pomonoid by $x + y \assign x \cdot y$ and $0 \assign 1$ is trivial on its own, it provides a way of constructing non-trivial bimonoids when combined with the complemented DM completion.

  To describe the algebra $\Lukthreeext$, we may first observe that $\elem{b} = \elem{b} \comp{\elemone} = \elem{b} \comp{\elem{a}} = \elem{b} \comp{\elem{b}}$. Of course, $\elem{a} = \elem{a} \comp{\elemone}$ and $\elemone = \elemone \comp{\elemone}$, since $\elemone$ (not $\elem{b}$!) is the additive unit of $\Lukthree$. Each element of $\Lukthreeext$ is now a join of a subset of $L \assign \{ \elem{b}, \elem{a}, \elemone, \elem{a} \comp{\elem{a}}, \elem{a} \comp{\elem{b}}, \comp{\elem{a}}, \comp{\elem{b}} \}$ as well as a meet of a subset of $R \assign \{ \elem{b}, \elem{a}, \elemone, \elem{b} + \comp{\elem{a}}, \elem{a} + \comp{\elem{a}}, \comp{\elem{a}}, \comp{\elem{b}} \}$. To find which joins of subsets of $L$ are distinct, we list subsets of $L$ of the form $\set{x \in L}{x \leq y}$ for $y \in R$:
\begin{align*}
  \{ \elem{b} \}, \{ \elem{b}, \elem{a} \}, \{ \elem{b}, \elem{a}, \elem{a} \comp{\elem{a}}, \elemone \}, \{ \elem{b}, \elem{a}, \elem{a} \comp{\elem{a}}, \elem{a} \comp{\elem{b}} \}, \{ \elem{b}, \elem{a}, \elem{a} \comp{\elem{a}}, \elemone, \elem{a} \comp{\elem{b}} \}, \{ \elem{b}, \elem{a}, \elem{a} \comp{\elem{a}}, \elemone, \elem{a} \comp{\elem{b}}, \comp{\elem{a}} \}, \{ \elem{b}, \elem{a}, \elem{a} \comp{\elem{a}}, \elemone, \elem{a} \comp{\elem{b}}, \comp{\elem{a}}, \comp{\elem{b}} \}
\end{align*}
  Taking intersection of these yields one more set: $\{ \elem{b}, \elem{a}, \elem{a} \comp{\elem{a}} \}$. These $8$ sets correspond to the distinct joins which exist in $\Lukthreeext$. It follows that the algebra $\Lukthreeext$ has precisely the structure shown in the middle part of Figure~\ref{fig: lukthree}. In particular, the complemented DM completion of a linear bimonoid need not be linear. The calculation is also facilitated by the fact that the elements of $L$ are ordered by $ \elem{b}< \elem{a}<  \elem{a} \comp{\elem{a}}< \elemone, \elem{a} \comp{\elem{b}}< \comp{\elem{a}}< \comp{\elem{b}}$ and that Galois closed sets have to be downsets of $L$.

  This~poset has two order-inverting involutions, depending on the behavior of $\elemone$ and $\elem{a} \comp{\elem{b}}$. However, we know that $\comp{\elemone} = \elemone$ (in each bimonoid $\comp{1} = 0$), hence the map $x \mapsto \comp{x}$ on $\Lukthreeext$ is the unique involution with two fixpoints. This allows us to describe $\Lukthree$ in terms of its meet generators, as shown in the right part of Figure~\ref{fig: lukthree}. Multiplication and addition is now fully determined by the formulas $a \comp{b} \cdot c \comp{d} = (a \cdot c) \comp{(b + d)}$ and $a + \comp{b} + c + \comp{d} = (a + c) + \comp{b \cdot d}$ and the fact that multiplication distributes over joins and addition over meets. For example, $\elem{a} \comp{\elem{b}} \cdot \elem{a} \comp{\elem{b}} = \elem{a}\elem{a} \cdot \comp{\elem{b} + \elem{b}} = \elem{b} \comp{\elem{b}} = \elem{b}$ and $\comp{\elem{a}} \cdot (\elemone \vee \elem{a} \comp{\elem{b}}) = \elemone \comp{\elem{a}} \vee \elem{a} \comp{\elem{a} + \elem{b}} = \comp{\elem{a}} \vee \elem{a} \comp{\elem{b}} = \comp{\elem{a}}$. Note that there is no need to verify that this algebra is indeed a commutative complemented $\ell$-bimonoid: we merely transformed the abstract definition of $\Lukthreeext$ as the commutative complemented DM completion of $\Lukthree$ (we already know that $\Lukthreeext$ exists) into a more tangible form.

\subsection{Complemented MacNeille completions: bisemigroups}
\label{subsec: macneille for bisemigroups}

  The construction of the commutative complemented DM completion of a commutative bimonoid may be extended to bisemigroups. This does not involve any substantial conceptual difficulty: we merely admit join generators of the forms $1$, $a$, $\comp{b}$ in addition to $a \comp{b}$, as well as meet generators of the forms $0$, $c$, $\comp{d}$ in addition to~$c + \comp{d}$. Given data which specifies under what conditions $a \leq 0$, $1 \leq b$, and $1 \leq 0$, the appropriate analogue of the involutive frame $\invframeofalg{A}$ is defined in the ``obvious'' way. The~proof that this construction works is a routine modification of the proof for bimonoids, however, it involves a lot of tedious case analysis. We~therefore merely sketch some of its parts.

  Let $\alg{A}$ be a commutative bisemigroup (an $\ell$-bisemigroup) in this subsection. Let $F$ be an upset of $\alg{A}$ (a lattice filter of $\alg{A}$), $I$ be a downset of $\alg{A}$ (a lattice ideal of $\alg{A}$), and let $\alpha \in \{ +, - \}$. (We admit the empty set as a lattice filter and a lattice ideal here.) We impose the following compatibility conditions on $F$, $I$, and $\alpha$:
\begin{itemize}
\item if $f \in F$, then $a \leq a \cdot f$ for each $a \in \alg{A}$,
\item if $i \in I$, then $a + i \leq a$ for each $a \in \alg{A}$,
\item if $F \cap I$ is non-empty, then $\alpha = +$,
\end{itemize}
  Such $F$, $I$, and $\alpha$ always exist: we may always take $F = I = \emptyset$. Relative to this data, the \emph{unital} commutative complemented DM completion of $\alg{A}$ is the unique complete commutative complemented bimonoid $\cdmalg{A}$ where
\begin{itemize}
\item elements of the forms $1$, $a$, $\comp{b}$, $a \comp{b}$ for $a, b \in \alg{A}$ are join dense,
\item elements of the forms $0$, $a$, $\comp{b}$, $a + \comp{b}$ for $a, b \in \alg{A}$ are meet dense,
\item $1 \leq a$ for $a \in \alg{A}$ if and only if $a \in F$,
\item $a \leq 0$ for $a \in \alg{A}$ if and only if $a \in I$, and
\item $1 \leq 0$ if and only if $\alpha = +$.
\end{itemize}
  The proof of uniqueness (indeed, universality) of $\cdmalg{A}$ carries over almost verbatim from the bimonoidal case, with some tedious case analysis thrown in. This completion again preserves all admissible meets and joins.

  The complemented DM completion $\cdmalg{A}$ is again obtained as the Galois algebra of a certain involutive $\alg{A}$-frame $\invframeofalg{A}$. Of~course, if $\alg{A}$ is a bisemigroup, we have to disregard the Gentzen conditions for $1$ and $0$ in the definition of an involutive $\alg{A}$-frame, since these are not part of the signature of $\alg{A}$.\footnote{The intermediate case where $\alg{A}$ has a multiplicative unit but not an additive unit or vice versa can be handled similarly.}

  The definition of the frame $\invframeofalg{A}$ needs to be modified as follows. The~set of join generators $L$ will consist of elements of four types, representing respectively $1$, $a$, $\comp{b}$, and $a \comp{b}$ for $a, b \in \alg{A}$. For the sake of simplicity, we shall simply write these as $1$, $a$, $\comp{b}$, and $a \comp{b}$, with the understanding that $a \comp{b}$ is to be interpreted as a formal pair consisting of $a$ and $b$, $\comp{b}$ as a formal pair consisting of $b$ and the sign $-$, and $a$ as a formal pair consisting of $a$ and the sign $+$. The set of meet generators $R$ will consist of elements of the form $0$, $a$, $\comp{b}$, or $a + \comp{b}$.

  We equip $L$ and $R$ with a monoidal structure in the obvious way: for example, $a \leftdot 1 = a$, $a \leftdot c = ac$, $a \leftdot \comp{d} = a \comp{d}$, and $a \leftdot c \comp{d} = ac \comp{d}$ etc. The maps $\ltorleft = \ltorright\colon L \to R$ and $\rtolleft = \rtolright\colon R \to L$ are defined as expected:
\begin{align*}
  \ltorleft(1) & = 0, & \ltorleft(a) & = \comp{a}, & \ltorleft(\comp{b}) & = b, & \ltorleft(a \comp{b}) & = b + \comp{a}, \\
  \rtolleft(0) & = 1, & \rtolleft(c) & = \comp{c}, & \rtolleft(\comp{d}) & = d, & \rtolleft(c + \comp{d}) & = d \comp{c}.
\end{align*}
  The relation $\sqleq$ between $L$ and $R$ is then defined as follows:
\begin{align*}
  1 \sqleq 0 & \iff \alpha = +, & a \sqleq 0 & \iff a \in I, & \comp{b} \sqleq 0 & \iff b \in F, & a \comp{b} \sqleq 0 & \iff a \leq b, \\
  1 \sqleq c & \iff c \in F, & a \sqleq c & \iff a \leq c, & \comp{b} \sqleq c & \iff b + c \in F, & a \comp{b} \sqleq c & \iff a \leq b + c, \\
  1 \sqleq \comp{d} & \iff d \in I, & a \sqleq \comp{d} & \iff a d \in I, & \comp{b} \sqleq \comp{d} & \iff d \leq b, & a \comp{b} \sqleq \comp{d} & \iff a d \leq b, \\
  1 \sqleq c + \comp{d} & \iff d \leq c, & a \sqleq c + \comp{d} & \iff a d \leq c, & \comp{b} \sqleq c + \comp{d} & \iff d \leq b + c, & a \comp{b} \sqleq c + \comp{d} & \iff a d \leq b + c.
\end{align*}

  The above definitions yield a commutative involutive frame $\invframeofalg{A}$. To turn it into an involutive $\alg{A}$-frame, we equip it with the maps $\lambda\colon A \to L$ and $\rho\colon A \to R$ such that $\lambda(a) = a$ and $\rho(a) = a$.

\begin{lemma}[Involutive frames from bisemigroups] \label{lemma: frames from bisemigroups}
  $\invframeofalg{A}$ is a faithful commutative involutive $\alg{A}$-frame.
\end{lemma}

\begin{proof}
  The algebras $\langle L, \leftdot, 1 \rangle$ and $\langle R, \rightplus, 0 \rangle$ are clearly commutative monoids, and $\lambda(a) \sqleq \rho(b)$ if and only if $a \leq b$ by definition, in particular $\lambda(a) \sqleq \rho(a)$. It remains to check nuclearity, Cut, and the Gentzen conditions (excluding the conditions for the two units). This is a tedious but completely routine case analysis. We only rehearse the proof that Cut holds in $\invframeofalg{A}$.

  Suppose that $x \sqleq g$ and $g \sqleq y$ for some $g \in \alg{A}$. (i) Suppose that $x = 1$. If $y = 0$, then $g \in F \cap I$, so $\alpha = +$ and $x \sqleq y$. If $y = c \in \alg{A}$, then $g \in F$ and $g \leq c$, so $c \in F$ and $x \sqleq y$. If $y = \comp{d}$, then $a \in F$ and $a d \in I$, so $d \in I$ and $x \sqleq y$. If $y = c + \comp{d}$, then $g \in F$ and $d g \leq c$, so $d \leq c$ and $1 \sqleq x + y$.

  (ii) Suppose that $x = a$ for some $a \in \alg{A}$. If $y = 0$, then $a \leq g \in I$, so $a \in I$ and $x \sqleq y$. If $y = c$, then $a \leq g \leq c$, so $x \sqleq y$. If $y = \comp{d}$, then $a \leq g$ and $g d \in I$, so $a d \in I$ and $x \sqleq y$. If $y = c + \comp{d}$, then $a \leq g$ and $g d \leq c$, so $a d \leq c$ and $x \sqleq y$.

  (iii) Suppose that $x = \comp{b}$ for some $b \in \alg{A}$. If $y = 0$, then $b + g \in F$ and $g \in I$, so $b \in F$ and $x \sqleq y$. If $y = c$, then $b + g \in F$ and $g \leq c$, so $b + c \in F$ and $x \sqleq y$. If $y = \comp{d}$, then $g + b \in F$ and $d \cdot g \in I$, so $d \leq d \cdot (g + b) \leq (d \cdot g) + b \leq b$ and $x \sqleq y$. If $y = c + \comp{d}$, then $g + b \in F$ and $d g \leq c$, so $d \leq d \cdot (g + b) \leq (d \cdot g) + b \leq c + b$ and $x \sqleq y$.

  (iv) Suppose that $x = a \comp{b}$ for some $a, b \in \alg{A}$. If $y = 0$, then $a \leq b + g$ and $g \in I$, so $a \leq b$ and $x \sqleq y$. If $y = c$, then $a \leq b + g$ and $g \leq c$, so $a \leq b + c$ and $x \sqleq y$. If $y = \comp{d}$, then $a \leq b + g$ and $d \cdot g \in I$, so $a \cdot d \leq d \cdot (g + b) \leq (d \cdot g) + b \leq b$ and $x \sqleq y$. If $y = c + \comp{d}$, then $a \leq b + g$ and $d \cdot g \leq c$, so $a \cdot d \leq (b + g) \cdot d \leq b + (d \cdot g) \leq b + c$.
\end{proof}

\begin{theorem}[Unital complemented commutative DM completions exist]
  Each commutative ($\ell$-)bisemigroup $\alg{A}$ has a unital complemented commutative DM completion, relative to a choice of $F$, $I$, $\alpha$, namely the Galois algebra $\galois{\invframeofalg{A}}$ of the involutive $\alg{A}$-frame $\invframeofalg{A}$.
\end{theorem}

\begin{corollary}[Embedding commutative bisemigroups into bimonoids]
  Each commutative ($\ell$-)bisemigroup embeds into a commutative ($\ell$-)bimonoid.
\end{corollary}

\section{Bimonoids of fractions}
\label{sec: bimonoids of fractions}

  We saw in the previous section that each commutative bimonoid $\alg{A}$ has a commutative complemented Dedekind--MacNeille (DM) completion $\cdmalg{A}$, where each element is a join of elements of the form $a \comp{b}$ for $a, b \in \alg{A}$. This completion is unique up to isomorphism, it contains each commutative admissible $\Delta_{1}$-extension of~$\alg{A}$, and it preserves all admissible meets and joins (Theorems~\ref{thm: universality of dm completions} and \ref{thm: dm completions exist} and Fact~\ref{fact: admissible joins preserved}).

  This answers the question of whether commutative bimonoids have complemented extensions. However, some bimonoids enjoy better-behaved complemented extensions than others: each element in the group of fractions of a cancellative commutative monoid has the form $a \comp{b}$, rather than merely being a join of such elements. We call such complemented extensions \emph{complemented bimonoids of fractions}. In this section, we determine which commutative bimonoids have a commutative complemented bimonoid of fractions. We~also show how to construct complemented bimonoids of fractions if they exist. Moreover, we prove that this construction sometimes yields a categorical equivalence between a class of residuated commutative bimonoids and a class of commutative complemented bimonoids with an interior operator.

\subsection{Definition and existence of bimonoids of fractions}
\label{subsec: existence of bimonoids of fractions}

  We define a bimonoid of fractions as a special kind of complemented admissible $\Delta_{1}$-extension.

\begin{definition}[Bimonoids of fractions]
  A commutative bimonoid $\alg{B}$ is called a \emph{commutative bimonoid of fractions} of a commutative bimonoid $\alg{A}$ if there is an embedding of bimonoids $\iota\colon \alg{A} \into \alg{B}$ such that each element of $\alg{B}$ has the form $\iota(a) \cdot \comp{\iota(b)}$ as well as $\iota(c) + \comp{\iota(d)}$, for some $a, b, c, d \in \alg{A}$.
\end{definition}

  In the above definition, we do not assume that $\comp{\iota(b)}$ exists for each $b \in \alg{A}$. As with other $\Delta_{1}$-extensions, we generally disregard the embedding $\iota$ and treat $\alg{A}$ as a sub-bimonoid of~$\alg{B}$. The basic facts about commutative admissible $\Delta_{1}$-extensions and commutative complemented DM completions proved in Subsection~\ref{subsec: macneille definition} also apply to bimonoids of fractions and complemented bimonoids of fractions, either as a direct corollary of results about $\Delta_{1}$-extensions, or by an analogous proof.

  In particular, the embedding $\iota$ preserves all admissible joins and meets. Consequently, if $\alg{A}$ and $\alg{B}$ are $\ell$-bimonoids, then $\iota$ is an embedding of $\ell$-bimonoids. Moreover, commutative complemented bimonoids of fractions are universal among all commutative bimonoids of fractions, and thus unique up to isomorphism.

\begin{theorem}[Universality of complemented bimonoids of fractions] \label{thm: universality of bimonoids of fractions}
    Let $\iota\colon \alg{A} \to \alg{B}$ be an isomorphism of commutative bimonoids and let $\extalg{A}$ and $\extalg{B}$ be commutative bimonoids of fractions of~$\alg{A}$ and $\alg{B}$. If $\extalg{B}$ is complemented, then there is a unique embedding of bimonoids $\overline{\iota}\colon \extalg{A} \to \extalg{B}$ which extends the isomorphism $\iota$.
\end{theorem}

\begin{proof}
  The proof is entirely analogous to the proof of the universality of complemented DM completions (Theorem~\ref{thm: universality of dm completions}). We define the maps $\overline{\iota}_{\pm}\colon \extalg{A} \to \extalg{B}$ as
\begin{align*}
  & \overline{\iota}_{+} \colon a \comp{b} \mapsto \iota(a) \comp{\iota(b)}, & & \overline{\iota}_{-} \colon a + \comp{b} \mapsto \iota(a) + \comp{\iota(b)}.
\end{align*}
  These maps are well-defined order-embeddings by Lemma~\ref{lemma: joins below joins} (and its order dual). The map $\overline{\iota}_{+}$ preserves products (and the multiplicative identity) by definition and the map $\overline{\iota}_{-}$ preserves sums (and the additive identity) by definition. But $\overline{\iota}_{+}(x) = \overline{\iota}_{-}(x)$ for each $x \in \extalg{A}$ by Lemmas~\ref{lemma: joins below meets}~and~\ref{lemma: meets below joins}. Let~us denote this common value by $\overline{\iota}(x)$. Then $\overline{\iota}$ is an embedding of bimonoids. Conversely, it is clear from its definition that $\overline{\iota}$ is the only embedding of $\extalg{A}$ into $\extalg{B}$ which extends $\iota$.
\end{proof}

\begin{corollary}[Uniqueness of complemented bimonoids of fractions] \label{cor: uniqueness of twist products}
  Let $\iota\colon \alg{A} \to \alg{B}$ be an isomorphism of commutative bimonoids and $\extalg{A}$ and $\extalg{B}$ be commutative complemented bimonoids of fractions of $\alg{A}$ and $\alg{B}$. Then there is a unique isomorphism $\overline{\iota}\colon \extalg{A} \to \extalg{B}$ which extends~$\iota$.
\end{corollary}

  In particular, if $\alg{A}$ has a commutative complemented bimonoid of fractions, then it is unique up to isomorphism and we denote it $\fracalgA$. Throughout the following, $\alg{A}$ denotes a commutative bimonoid, and $\cdmalg{A}$ denotes a commutative complemented DM completion of $\alg{A}$ (unique up to isomorphism).

\begin{proposition}[Existence of bimonoids of fractions] \label{prop: existence of bimonoid of fractions}
  The following are equivalent for each commutative bimonoid $\alg{A}$:
\begin{enumerate}[(i)]
\item $\alg{A}$ has a commutative complemented bimonoid of fractions,
\item the elements of the form $a \cdot \comp{b}$ form a complemented sub-bimonoid of~$\cdmalg{A}$,
\item the elements of the form $a + \comp{b}$ form a complemented sub-bimonoid of~$\cdmalg{A}$,
\item for each $a, b \in \alg{A}$ there are $x, y \in \alg{A}$ such that $a \cdot \comp{b} = x + \comp{y}$ in~$\cdmalg{A}$.
\item for each $a, b \in \alg{A}$ there are $x, y \in \alg{A}$ such that $a + \comp{b} = x \cdot \comp{y}$ in~$\cdmalg{A}$.
\end{enumerate}
\end{proposition}

\begin{proof}
  For the implication from (i) to (ii) see the previous proof. The converse implication holds by the definition of a bimonoid of fractions. The equivalence between (ii) and (iii), as well as between (iv) and (v), holds because $\comp{a \cdot \comp{b}} = b + \comp{a}$ and $\comp{a + \comp{b}} = b \cdot \comp{a}$. The last equivalence holds because elements of the form $a + \comp{b}$ are closed under addition and contain the additive unit, while elements of the form $a \cdot \comp{b}$ are closed under multiplication and contain the multiplicative unit.
\end{proof}

  The canonical example of a complemented bimonoid of fractions is the group of fractions of a cancellative commutative monoid (the partial order being the equality relation). On the other hand, many bimonoids do not have complemented bimonoids of fractions. If $\alg{A}$ is a finite distributive lattice, then $\cdmalg{A}$ is the free Boolean extension of $\alg{A}$. In~general the elements of the form $a \wedge \comp{b}$ do not form a subalgebra of~$\cdmalg{A}$. For~example, if $\alg{A}$ is the four-element chain $\elem{p} < \elem{q} < \elem{r} < \elem{s}$, then $\elem{q} \vee \comp{\elem{r}}$ does not have the form $a \wedge \comp{b}$ in~$\cdmalg{A}$.

  In order to use the previous proposition to establish the existence of complemented bimonoids of fractions for certain classes of bimonoids, we need to be able to describe the condition $a \cdot \comp{b} = x + \comp{y}$ directly in terms of~$\alg{A}$ rather than $\cdmalg{A}$. The~following fact shows that this condition can be expressed by a universal sentence in the language of bimonoids. The condition that a commutative bimonoid has a commutative complemented bimonoid of fractions can therefore be expressed by an elementary $\Pi_{3}$-sentence in the language of bimonoids. This may seem like a rather complex description, but note that Heyting algebras are described by a sentence of the same logical complexity in the language of distributive lattices:
\begin{align*}
  (\forall a b \in \alg{A}) \, (\exists x \in \alg{A}) \, (\forall c \in \alg{A}) \, (a \wedge c \leq b \implies c \leq x).
\end{align*}

\begin{fact} \label{fact: transformation conditions}
  Let $\alg{A}$ be a commutative bimonoid and $a,b,x,y \in \alg{A}$. Then $a \cdot \comp{b} = x + \comp{y}$ in $\cdmalg{A}$ if and only if $a \cdot y \leq b + x$ and
\begin{align*}
  (\forall p q u v \in \alg{A}) \, ( u \cdot y \leq v + x ~ \& ~ a \cdot q \leq b + p \implies u \cdot q \leq v + p ).
\end{align*}
  If $\alg{A}$ is residuated, this is equivalent to $y \leq a \rightarrow (b + x)$ and
\begin{align*}
  (\forall p q \in \alg{A}) \, (a \rightarrow (b + p)) \cdot (y \rightarrow (x + q)) \leq p + q.
\end{align*}
  If some $y$ satisfies these conditions, then in particular $y \assign a \rightarrow (b + x)$ does. 
\end{fact}

\begin{proof}
  The inequality $a \comp{b} \leq x + \comp{y}$ in $\cdmalg{A}$ is equivalent to $a y \leq b + x$ and
\begin{align*}
  x + \comp{y} \leq a \comp{b}
  & \iff (\forall uv \in \alg{A}) \left( u \comp{v} \leq x + \comp{y} \implies u \comp{v} \leq a \comp{b} \right) \\
  & \iff (\forall uv \in \alg{A}) \left( u y \leq v + x \implies (\forall p q \in \alg{A}) \left( a q \leq b + p \implies u q \leq v + p \right) \right) \\
  & \iff (\forall pquv \in \alg{A}) \left( u y \leq v + x ~ \& ~ a q \leq b + p \implies u q \leq v + p \right) \\
  & \iff (\forall pquv \in \alg{A}) \left( u y \leq v + x ~ \& ~ q \leq a \rightarrow (b + p) \implies q \leq u \rightarrow (v + p) \right) \\
  & \iff (\forall puv \in \alg{A}) \left( u y \leq v + x \implies a \rightarrow (b + p) \leq u \rightarrow (v + p) \right) \\
  & \iff (\forall puv \in \alg{A}) \left( u \leq y \rightarrow (v + x) \implies u \leq (a \rightarrow (b + p)) \rightarrow (v + p) \right) \\
  & \iff (\forall pv \in \alg{A}) \left( y \rightarrow (v + x) \leq (a \rightarrow (b + p)) \rightarrow (v + p) \right) \\
  & \iff (\forall pv \in \alg{A}) \, (a \rightarrow (b + p)) \cdot (y \rightarrow (v + x)) \leq v + p \\
  & \iff (\forall pq \in \alg{A}) \, (a \rightarrow (b + p)) \cdot (y \rightarrow (x + q)) \leq p + q,
\end{align*}
  provided that the appropriate residuals exist.
\end{proof}

  Let us now use this criterion to prove the existence of commutative complemented bimonoids of fractions for various classes of com\-mutative bimonoids. For unital meet semilattices, where ${a + b = a \wedge b = a \cdot b}$, the~above condition for $a \cdot \comp{b}$ to be equal to $x + \comp{y}$ can be simplified substantially.

\begin{fact} \label{fact: semilattice transformation conditions}
  Let $\alg{A}$ be a unital meet semilattice. Then the following are equivalent for $a, b, y \in \alg{A}$:
\begin{enumerate}[(i)]
\item $a \cdot \comp{b} = x + \comp{y}$ in $\cdmalg{A}$ for some $x \in \alg{A}$,
\item $a \cdot \comp{b} = a + \comp{y}$ in $\cdmalg{A}$,
\item $a \cdot \comp{ab} = a + \comp{y}$ in $\cdmalg{A}$,
\item $a y = a b$ and moreover
\begin{align*}
  (\forall p q \in \alg{A}) \left( p y \leq a b ~ \& ~ a q \leq a b \implies p q \leq a b \right).
\end{align*}
\end{enumerate}
\end{fact}

\begin{proof}
  We know that $a \cdot \comp{b} = x + \comp{y}$ if and only if $a y \leq b x$ and
\begin{align*}
  (\forall pquv \in \alg{A}) \left( u y \leq v x ~ \& ~ a q \leq b p \implies u q \leq v p \right).
\end{align*}
  But then $ay \leq x$, therefore the implication also holds for $x \assign ay$, as does the inequality $ay \leq bx$. Moreover, the implication holds for $x \assign ay$ if and only if it holds for $x \assign a$, and the inequality $ay \leq bx$ also holds for $x \assign a$. Therefore, if some $x$ satisfies these conditions, then $x \assign a$ does too. For $x = a$, the two conditions become $a y \leq b$ and
\begin{align*}
  (\forall pquv \in \alg{A}) \left( u y \leq v a ~ \& ~ a q \leq b p \implies u q \leq v p \right).
\end{align*}
  These are satisfied by $b$ if and only if they are satisfied by $b \assign ab$. Moreover, the implication is equivalent to the conjunction of
\begin{align*}
  (\forall pquv \in \alg{A}) \left( u y \leq v ~ \& ~ u y \leq a ~ \& ~ a q \leq b p \implies u q \leq v \right), \\
  (\forall pquv \in \alg{A}) \left( u y \leq v ~ \& ~ u y \leq a ~ \& ~ a q \leq b p \implies u q \leq p \right).
\end{align*}
  These are equivalent respectively to
\begin{align*}
  (\forall quv \in \alg{A}) \left( u y \leq v ~ \& ~ u y \leq a ~ \& ~ a q \leq b \implies u q \leq v \right), \\
  (\forall pqu \in \alg{A}) \left( u y \leq a ~ \& ~ a q \leq b ~ \& ~ a q \leq p \implies u q \leq p \right).
\end{align*}
  The implication $(\forall v \in \alg{A}) \left( u y \leq v \implies uq \leq v \right)$ is equivalent to $u q \leq u y$, i.e.\ to $u q \leq y$, and the implication $(\forall p) \left( a q \leq p \implies u q \leq p \right)$ is equivalent to $ u q \leq a q$, i.e.\ $u q \leq a$. The two implications are thus equivalent to
\begin{align*}
  (\forall qu \in \alg{A}) \left( u y \leq a ~ \& ~ a q \leq b \implies u q \leq y \right), \\
  (\forall qu \in \alg{A}) \left( u y \leq a ~ \& ~ a q \leq b \implies u q \leq a \right),
\end{align*}
  or in other words, changing the variable $u$ to $p$, 
\begin{align*}
  & (\forall pq \in \alg{A}) \left( p y \leq a ~ \& ~ a q \leq b \implies p q \leq ay \right).
\end{align*}
  Taking $p \assign a$ and $q \assign b$ yields that $a b \leq ay$, and we already know that $a y \leq a b$. Moreover, $p y \leq a$ is equivalent to $p y \leq a b$, since $a y \leq b$.
\end{proof}

  This immediately yields the existence of bimonoids of fractions of Brouwerian semilattices.

\begin{definition}[Brouwerian semilattices]
  A \emph{Brouwerian semilattice} is an integral idempotent residuated pomonoid, or equivalently a unital meet semi\-lattice $\langle A, \wedge, 1 \rangle$ equipped with a binary operation $x \rightarrow y$ such that
\begin{align*}
  x \wedge y \leq z & \iff y \leq x \rightarrow z.
\end{align*}
\end{definition}

\begin{fact} \label{fact: fractions for brsls}
  Each Brouwerian semilattice has a commutative complemented bimonoid of fractions where
\begin{align*}
  a \cdot \comp{b} = a + \comp{a \rightarrow b}.
\end{align*}
\end{fact}

\begin{proof}
  Taking $y \assign a \rightarrow b$ in the previous fact yields that indeed $a y = a b$. Moreover, if $py \leq ab$ and $aq \leq ab$, then $p \leq y \rightarrow ab$ and $q \leq a \rightarrow ab \leq a \rightarrow b$. But then $pq \leq ((a \rightarrow b) \rightarrow ab) (a \rightarrow b)) \leq b$.
\end{proof}

  By contrast, the semilattices $\Mthree$ and $\Nfive$ do not have commutative complemented bimonoids of fractions. We do not know of any non-distributive semilattice which has such a bimonoid of fractions.

  Recall the discussion of Boolean-pointed Brouwerian algebras from Subsection~\ref{subsec: examples of bimonoids}: they are Brouwerian algebras equipped with a constant $0$ such that the interval $[0, 1]$ is Boolean.

\begin{fact} \label{fact: fractions for boolean-pointed bras}
  Each Boolean-pointed Brouwerian algebra has a commutative complemented bimonoid of fractions where
% vskip
\vskip-20pt
\begin{align*}
  a \cdot \comp{b} & = 0 a + \comp{a \rightarrow b}.
\end{align*}
\end{fact}

\begin{proof}
  By Fact~\ref{fact: transformation conditions} it suffices to show that $a (a \rightarrow b) \leq b + 0 a$ and for all $p,q$
\begin{align*}
  a (a \rightarrow (b + p)) ((a \rightarrow b) \rightarrow (0 a + q)) & \leq p + q.
\end{align*}
  Clearly $a (a \rightarrow b) = a b = (b + 0) a \leq b + 0a$. To prove the other inequality, by Lemma~\ref{lemma: validity in bpbras} it suffices to prove that $(a \rightarrow bp) ((a \rightarrow b) \rightarrow a q) \leq pq$ holds in all Brouwerian lattices and that
\begin{align*}
  (a \rightarrow (b + p)) ((a \rightarrow b) \rightarrow (0 a + q)) & \leq (p + q) \vee 0
\end{align*}
  holds in all Boolean-pointed Brouwerian lattices. The former inequality is routine to prove, and it implies that $(a \rightarrow bp) ((a \rightarrow b) \rightarrow a q) \leq 0 \rightarrow p q$, therefore it only remains to prove that
\begin{align*}
  (a \rightarrow (b + p)) ((a \rightarrow b) \rightarrow (0 a + q)) \leq p \vee q \vee 0.
\end{align*}
  Since $p \vee q \vee 0 = ((p \vee q) \rightarrow 0) \rightarrow 0$, this is equivalent to
\begin{align*}
  (a \rightarrow (b + p)) ((a \rightarrow b) \rightarrow (0 a + q)) (p \rightarrow 0) (q \rightarrow 0) \leq 0.
\end{align*}
  But $(a \rightarrow (b+p)) (p \rightarrow 0) \leq a \rightarrow (b + p (p \rightarrow 0)) \leq a \rightarrow (b + 0) = a \rightarrow b$. Likewise, $(a \rightarrow b) ((a \rightarrow b) \rightarrow (0 a + q)) (q \rightarrow 0) \leq (0a + q) (q \rightarrow 0) \leq 0 a + 0 = 0 a \leq 0$.
\end{proof}

  Finally, let us directly verify that our criterion covers groups of fractions.

\begin{definition}[Order-cancellative monoids]
  A pomonoid is called \emph{order-cancellative} if it satisfies the implications
\begin{align*}
  a \cdot x \leq b \cdot x & \implies a \leq b, \\
  x \cdot a \leq x \cdot b & \implies a \leq b.
\end{align*}
\end{definition}

\begin{fact} \label{fact: order-cancellative}
  A residuated pomonoid is order-cancellative if and only if it satisfies the equations $x y \bs x z \equals y \bs z$ and $x z / y z \equals x / y$. An s$\ell$-monoid is order-cancellative if and only if it is \emph{cancellative}, i.e.\ if and only if $x y \equals x z$ implies $y \equals z$ and moreover $x z \equals y z$ implies $x \equals y$.
\end{fact}

\begin{fact} \label{fact: fractions for cancellative pomonoids}
  Each order-cancellative commutative monoid has a commutative complemented bimonoid of fractions where 
% vskip
\vskip-20pt
\begin{align*}
  a \cdot \comp{b} = a + \comp{b}.
\end{align*}
\end{fact}

\begin{proof}
  The condition $a \comp{b} = x + \comp{y}$ is equivalent to $a y \leq b x$ and
\begin{align*}
  (\forall p q u v \in \alg{A}) \left( u y \leq v x ~ \& ~ a q \leq b p \implies u q \leq v p \right).
\end{align*}
  Take $x \assign a$ and $y \assign b$. Then $a y \leq b x$, and if $u b \leq v a$ and $a q \leq b p$, then $u b q \leq v a q \leq v b p$, hence by order-cancellativity $u q \leq v p$.
\end{proof}

  As our final example of complemented bimonoids of fractions, let us consider the two constructions of Galatos \& Raftery~\cite{galatos+raftery04}, which embed certain residuated pomonoids into cyclic involutive residuated pomonoids. Here \emph{cyclicity} means that $\compl{x} = \compr{x}$ for each $x$. This common value will be denoted $\comp{x}$.

  The construction comes in two flavors. If $\alg{A} = \langle A, \leq, \cdot, 1, \bs, / \rangle$ is a residuated pomonoid with an upper bound~$\top$, then we construct an involutive residuated pomonoid $\topmirroralgA$ consisting of two disjoint copies of $A$. The two copies will be denoted $A$ and $A'$, and the elements of $A'$ will be denoted $a'$ for $a \in \alg{A}$. The order of $\topmirroralgA$ extends the order of~$\alg{A}$ as follows for $a, b \in A$:
\begin{align*}
  a' & \leq b, & a & \nleq b', & a' \leq b' & \iff b \leq a.
\end{align*}
  Multiplication in $\topmirroralgA$ extends multiplication in $\alg{A}$ by:
\begin{align*}
  a' \cdot b & \assign (a \bs b)', & a \cdot b' & \assign (b / a)', & a' \cdot b' & \assign \top'.
\end{align*}
  On the other hand, if $\alg{A}$ is a residuated pomonoid with a lower bound~$\bot$ (in which case it also has an upper bound $\top \assign \bot \bs \bot$), then we construct the involutive residuated pomonoid $\botmirroralgA$ over $A \cup A'$ in a similar way. The order on $\botmirroralgA$ extends the order on $\alg{A}$:
\begin{align*}
  a' & \nleq b, & a & \leq b', & a' \leq b' & \iff b \leq a.
\end{align*}
  Multiplication in $\botmirroralgA$ extends multiplication in $\alg{A}$:
\begin{align*}
  a' \cdot b & \assign (a \bs b)', & a \cdot b' & \assign (b / a)', & a' \cdot b' & \assign \bot'.
\end{align*}
  That is, $\topmirroralgA$ adds a mirror copy of $\alg{A}$ below $\alg{A}$, and $\botmirroralgA$ adds a mirror copy of $\alg{A}$ above $\alg{A}$. Both of these constructions yield a bounded cyclic involutive residuated pomonoid where $\comp{a} = a'$ and $\comp{a'} = a$. Also, if $\alg{A}$ is lattice-ordered, then so are $\topmirroralgA$ and $\botmirroralgA$.

  We can now interpret these involutive residuated pomonoids as complemented bimonoids of fractions of~$\alg{A}$, if we suitably expand $\alg{A}$ to a bisemigroup. In the first case, we expand $\alg{A}$ to a bisemigroup with a multiplicative unit by taking what we called the $\top$-drastic addition in Subsection~\ref{subsec: examples of bimonoids}: $x + y \assign \top$. In the second case, we expand $\alg{A}$ to a bisemigroup with a multiplicative unit by taking the $\bot$-drastic addition: $x + y \assign \bot$. (This yields a bisemigroup because $x \cdot \bot = \bot = \bot \cdot x$ in each bounded residuated pomonoid.) But observe that this is precisely the addition of $\topmirroralgA$ and $\botmirroralgA$ restricted to $\alg{A}$, thanks to the definitions $a' \cdot b' \assign \top'$ and $a' \cdot b' \assign \bot'$. The bisemigroup $\alg{A}$ is therefore a sub-bisemigroup of $\topmirroralgA$ and $\botmirroralgA$, and moreover this embedding preserves the multiplicative unit (in the $\top$-drastic case, it also preserves residuals).

  If we extend our definition of a bimonoid of fractions to allow for cyclic bimonoids of fractions of bisemigroups, then the cyclic complemented bimonoids $\topmirroralgA$ and $\botmirroralgA$ are right (as well as left) cyclic bimonoids of fractions of $\alg{A}$: each element has one of the forms $1$, $a$, $\comp{b}$, or $a \comp{b}$ ($\comp{b} a$), as well as one of the forms $0$, $a$, $\comp{b}$, or $a + \comp{b}$ ($\comp{b} + a$) for some $a, b \in \alg{A}$. Indeed, in this case each element simply has the form $a$ or $\comp{b}$.

\subsection{Constructing complemented bimonoids of fractions}
\label{subsec: constructing bimonoids of fractions}

  In the following, let $\alg{A}$ be a commutative bimonoid with a commutative complemented bimonoid of fractions. We~show how to construct this bimonoid of fractions, first as a quotient of a preordered algebra on~$A^2$, then (in some cases) directly as an ordered algebra on a subset of $A^2$. The~input for this construction consists of a pair of functions $\transplus$, $\transminus$ which specify how to solve the equation $a \cdot \comp{b} = x + \comp{y}$ for $x$ and $y$ in~$\cdmalg{A}$. Recall that the condition that $a \cdot \comp{b} = x + \comp{y}$ can be stated directly in terms of $\alg{A}$ as a universal sentence in the language of commutative bimonoids (Fact~\ref{fact: transformation conditions}). Moreover, if $\alg{A}$ is residuated, it can be stated in terms of~$\alg{A}$ as a set of universally quantified inequalities in the language of commutative residuated bimonoids.

\begin{definition}[Transformation functions]
  The functions $\transplus, \transminus\colon A^2 \to A^2$ are called \emph{transformation functions} for $\alg{A}$ if
\begin{align*}
  a \cdot \comp{b} = \transplus(a, b) + \comp{\transminus(a, b)} \text{ for all } a, b \in \alg{A}.
\end{align*}
\end{definition}

  We assume below that $\transplus$ and $\transminus$ are transformation functions for $\alg{A}$. These are not uniquely determined by~$\alg{A}$, therefore our constructions of complemented bimonoids of fractions will be relative to some choice of $\transplus$~and~$\transminus$. However, the resulting complemented bimonoids of fractions will of course be isomorphic.

\begin{theorem}[Constructing complemented bimonoids of fractions] \label{thm: bimonoids of fractions}
  Let $\alg{A}$ be a commutative bimonoid with transformation functions $\transplus, \transminus$. We define $\prefracalg{A}$ to be the algebra on $A^2$ with the operations
\settowidth{\auxlength}{$\,\preplus\,$}
\begin{align*}
  \comp{\multipair{a}{b}} & \assign \multipair{\transminus(a, b)}{\transplus(a, b)}, \\
  {\multipair{a}{b}} \hbox to \auxlength{\hfil$\precdot$\hfil} \multipair{c}{d} & \assign \multipair{a \cdot c}{b + d}, \\
  \multipair{a}{b} \hbox to \auxlength{\hfil$\preplus$\hfil} \multipair{c}{d} & \assign \multipair{\transminus(\transminus(a, b) \cdot \transminus(c, d), \transplus(a, b) + \transplus(c, d))}{\transplus(\transminus(a, b) \cdot \transminus(c, d), \transplus(a, b) + \transplus(c, d))},
\end{align*}
  and the constants $\preone \assign \multipair{1}{0}$ and $\prezero \assign \multipair{0}{0}$. Let $\preleq$ be the preorder
\begin{align*}
  \multipair{a}{b} \preleq \multipair{c}{d} \iff (\forall x, y \in \alg{A}) (x \cdot c \leq y + d \implies x \cdot a \leq y + b).
\end{align*}
  The equivalence relation $\theta$ induced by this preorder is a congruence on $\prefracalg{A}$. Then the ordered algebra $\fracalgA \assign \pair{\prefracalg{A} / \theta}{\preleq}$ of equivalence classes of this congruence ordered by $\preleq$ is a commutative complemented bimonoid of fractions of $\alg{A}$ relative to the embedding $\fraciota\colon a \mapsto [\multipair{a}{0}]_{\theta}$. Moreover, $\fraciota(a) \precdot \comp{\fraciota(b)} = [\multipair{a}{b}]_{\theta}$.
\end{theorem}

\begin{proof}
  Because $\alg{A}$ has a pair of transformation functions, the sub-bimonoid $\alg{B}$ of $\cdmalg{A}$ consisting of elements of the form $a \comp{b}$ is a commutative complemented bimonoid of fractions of $\alg{A}$ (Proposition~\ref{prop: existence of bimonoid of fractions}). Now consider the surjective map $\varepsilon\colon \multipair{a}{b} \mapsto a \comp{b} \in \alg{B}$. By Lemma~\ref{lemma: joins below joins}, $\multipair{a}{b} \preleq \multipair{c}{d}$ if and only if $\varepsilon \multipair{a}{b} \leq \varepsilon \multipair{c}{d}$. Moreover, $\varepsilon$ is a homomorphism of bimonoids: $\varepsilon(\preone) = 1 \comp{0} = 1$ and $\varepsilon(\prezero) = 0 \comp{0} = 0$ for the two units, $\varepsilon (\multipair{a}{b} \circ \multipair{c}{d}) = \varepsilon \multipair{a \cdot c}{b + d} = ac \cdot \comp{b + d} = a\comp{b} \cdot c \comp{d} = \varepsilon \multipair{a}{b} \cdot \varepsilon \multipair{c}{d}$ for multiplication, and 
\begin{align*}
  \varepsilon \multipair{a}{b} + \varepsilon \multipair{c}{d}
  & = a\comp{b} + c \comp{d} \\
  & = \transplus(a, b) + \comp{\transminus(a, b)} + \transplus(c, d) + \comp{\transminus(c, d)} \\
  & = \transplus(a, b) + \transplus(c, d) + \comp{\transminus(a, b) \cdot \transminus(c, d)} \\
  & = \transminus(\transminus(a, b) \cdot \transminus(c, d), \transplus(a, b) + \transplus(c, d)) \cdot \comp{\transplus(\transminus(a, b) \cdot \transminus(c, d), \transplus(a, b) + \transplus(c, d))} \\
  & = \varepsilon (\multipair{a}{b} \oplus \multipair{c}{d}).
\end{align*}
  The kernel of this map is precisely $\theta$, therefore $\theta$ is a congruence and the map $[\multipair{a}{b}]_{\theta} \mapsto a \comp{b}$ is a well-defined isomorphism of bimonoids between $\fracalgA$ and $\alg{B}$. Since $\comp{a \comp{b}} = \comp{\transplus(a, b) + \comp{\transminus(a, b)}} = \transminus(a, b) \cdot \comp{\transplus(a, b)}$, it follows by this isomorphism that the complement of $[\multipair{a}{b}]_{\theta}$ in $\fracalgA$ is $[\multipair{\transminus(a, b)}{\transplus(a, b)}]$. We also have $\varepsilon(\multipair{c}{d}) = a$ if and only if $[\multipair{c}{d}]_{\theta} = [\multipair{a}{0}]_{\theta}$. Since $\alg{B}$ is a commutative complemented bimonoid of fractions of $\alg{A}$ relative to the inclusion of $\alg{A}$ into $\alg{B}$ and $\varepsilon$ is an isomorphism between $\fracalgA$ and $\alg{B}$, this implies that $\fracalgA$ is a commutative complemented bimonoid of fractions of $\alg{A}$ relative to the map $a \mapsto [\multipair{a}{0}]_{\theta}$. Finally, $\varepsilon(\fraciota(a) \circ \comp{\fraciota(b)}) = \varepsilon(\fraciota(a)) \cdot \comp{\varepsilon(\fraciota(b))} = a \comp{b} = \varepsilon(\multipair{a}{b})$, so $\fraciota(a) \circ \comp{\fraciota(b)} = [\multipair{a}{b}]_{\theta}$.
\end{proof}

  For order-cancellative bimonoids, where we can take $\transplus(a, b) \assign a$ and $\transminus(a, b) = b$, the above construction simplifies to the usual quotient construction of the group of fractions.

  There is no reason to expect the above construction to be functorial. Suppose that $\alg{A}$ and $\alg{B}$ are commutative bimonoids such that $\fracalgA$ and $\fracalgB$ exist. Given a homomorphism of bimonoids $h\colon \alg{A} \to \alg{B}$, it~is natural to try to define a map $\frachom{h}\colon \fracalgA \to \fracalgB$ as $\frachom{h}\colon \multipair{a}{b} \mapsto \multipair{h(a)}{h(b)}$. However, we have no reason to expect this map to be well-defined in general, much less order preserving. Owing to the parameters $x$~and~$y$ in the definition of the preorder on $\prefracalg{A}$, the pairs $\multipair{a}{b}$ and $\multipair{c}{d}$ might be equivalent in $\prefracalg{A}$ without $\multipair{h(a)}{h(b)}$ and $\multipair{h(c)}{h(d)}$ being equivalent in $\prefracalg{B}$.

  The construction becomes functorial if we restrict to order-cancellative pomonoids, because the condition for $\multipair{a}{b} \leq \multipair{c}{d}$ in $\prefracalg{A}$ is then equivalent to $d \cdot a \leq c \cdot b$, i.e.\ the parameters $x$ and $y$ can be eliminated. However, even in this case the functor fails to be full: there may be homomorphisms $h\colon \fracalgA \to \fracalgB$ which do not restrict to maps from $\alg{A}$ to $\alg{B}$. For example, the non-trivial automorphism of the group of integers $\Z \cong \fracalgN$ does not arise from any endomorphism of $\N$.

  To resolve these issues, we adopt the solution that Montagna \& Tsinakis~\cite{montagna+tsinakis10} used in the context of groups of fractions of cancellative residuated pomonoids: we extend $\fracalgA$ by an interior operator which allows us to recover $\alg{A}$ as its image. This requires us to move to the setting of \emph{residuated} commutative bimonoids. Let us therefore assume from now on that the commutative bimonoid $\alg{A}$ is residuated.

\begin{proposition}[The interior operator $\fracsigma$ on bimonoids of fractions]
  Let $\alg{B}$ be a com\-mutative bimonoid of fractions of a residuated commutative bimonoid $\alg{A}$. Then the following map is a well-defined interior operator on $\alg{B}$ whose image is precisely $\alg{A}$:
\begin{align*}
  \fracsigma(a + \comp{b}) \assign b \rightarrow a \text{ for } a, b \in \alg{A}.
\end{align*}
\end{proposition}

\begin{proof}
  If $a + \comp{b} \leq c + \comp{d}$, then $b \rightarrow (a + x) \leq d \rightarrow (c + x)$ for each $x$, hence $\fracsigma (a + \comp{b}) = b \rightarrow a \leq d \rightarrow c = \fracsigma (c + \comp{d})$. The map $\fracsigma$ is therefore well-defined ($a + \comp{b} = c + \comp{d}$ implies $b \rightarrow a = d \rightarrow c$) and isotone. It is decreasing because $\fracsigma (a + \comp{b}) \leq a + \comp{b} \iff b \rightarrow a \leq a + \comp{b} \iff b \cdot (b \rightarrow a) \leq a$. The map is also idempotent: $\fracsigma(\fracsigma (a + \comp{b})) = \fracsigma (b \rightarrow a) = \fracsigma ((b \rightarrow a) + \comp{1}) = {1 \rightarrow (b \rightarrow a)} = {b \rightarrow a} =\fracsigma (a + \comp{b})$. Finally, $\fracsigma[B] = A$ since $\fracsigma(a + \comp{b}) = b \rightarrow a \in \alg{A}$ and $\fracsigma(a) = a$ for $a \in \alg{A}$.
\end{proof}

\begin{proposition}[The interior operator $\fracsigma$ on $\Delta_{1}$-extensions]
  Let $\alg{B}$ be a commutative $\Delta_{1}$-extension of a complete commutative $\ell$-bimonoid $\alg{A}$. Then the following map is a well-defined interior operator on $\alg{B}$ whose image is precisely $\alg{A}$:
\begin{align*}
  \fracsigma\left(\bigwedge_{i \in I} (a_{i} + \comp{b}_{i})\right) \assign \bigwedge_{i \in I} (b_{i} \rightarrow a_{i}) \text{ for } a, b \in \alg{A}.
\end{align*}
\end{proposition}

\begin{proof}
  Suppose that $\bigwedge_{i \in I} (a_{i} + \comp{b}_{i}) \leq \bigwedge_{j \in J} (c_{j} + \comp{d}_{j})$ in $\alg{B}$ for $a_{i}, b_{i}, c_{j}, d_{j} \in \alg{A}$. To prove that the map $\fracsigma$ is well-defined and isotone, we must show that $\bigwedge_{i \in I} (b_{i} \rightarrow a_{i}) \leq \bigwedge_{j \in J} (d_{j} \rightarrow c_{j})$. But the former inequality is equivalent to the condition that for each $x, y \in \alg{A}$ we have: $x b_{i} \leq a_{i} + y$ for each $i \in I$ implies $x d_{j} \leq c_{j} + y$ for each $j \in J$. Taking $y \assign 0$ yields the latter inequality. The other conditions are proved as in the previous proposition.
\end{proof}

  Of course, instead of requiring that $\alg{A}$ be complete, we may alternatively require that each element of $\alg{B}$ be a finite meet of elements of the form $a + \comp{b}$ for $a, b \in \alg{A}$.

  Expanding bimonoids of fractions by the unary operation $\fracsigma$ eliminates all homomorphisms $h\colon \fracalgA \to \fracalgB$ which do not restrict to homomorphisms from $\alg{A}$ to $\alg{B}$: such maps do not commute with $\fracsigma$, since the images $\fracsigma[\fracalgA]$ and $\fracsigma[\fracalgB]$ are precisely $\alg{A}$ and $\alg{B}$. To obtain a functorial construction, we must now restrict to complemented bimonoids of fractions where each element has a certain canonical or \emph{normal} representation.

\begin{definition}[Normal representations]
  Let $\alg{B}$ be a commutative complemented bimonoid of fractions of a commutative residuated bimonoid $\alg{A}$. Then an element $x \in \alg{B}$ is \emph{normal} if
\begin{align*}
  x = \fracsigma(x) \cdot \comp{\fracsigma(\comp{x})}.
\end{align*}
  A pair $\multipair{a}{b} \in \algAsquared$ is \emph{normal} if
\begin{align*}
  \multipair{a}{b} = \multipair{\fracsigma(x)}{\fracsigma(\comp{x})} \text{ for some } x \in \alg{B}.
\end{align*}
  Such a pair $\multipair{a}{b}$ will be called the \emph{normal representation} of $x$. If each $x \in \alg{B}$ is normal, we call $\alg{B}$ itself a \emph{normal} commutative bimonoid of fractions. Transformation functions for $\alg{A}$ will be called \emph{normal} transformation functions if $\alg{B}$ is normal.
\end{definition}

\begin{fact} \label{fact: normal negation}
  If $\multipair{a}{b}$ is a normal pair, then so is $\multipair{b}{a}$.
\end{fact}

\begin{proof}
  If $\multipair{a}{b} = \multipair{\fracsigma(x)}{\fracsigma(\comp{x})}$, then $\multipair{b}{a} = \multipair{\fracsigma(y)}{\fracsigma(\comp{y})}$ for $y \assign \comp{x}$.
\end{proof}

  Comparing normal representations is much easier than comparing general representations.

\begin{lemma}[Comparing normal representations] \label{lemma: comparing normal representations}
  Let $\multipair{a}{b}$ and $\multipair{c}{d}$ be normal representations of $x$ and $y$ in a commutative complemented bimonoid of fractions $\alg{B}$ of $\alg{A}$. Then $x \leq y$ in $\alg{B}$ if and only if $a \leq c$ and $d \leq b$ in $\alg{A}$.
\end{lemma}

\begin{proof}
  If $a \leq c$ and $d \leq b$, then $x = a \comp{b} \leq c \comp{d} = y$ by monotonicity of multiplication. Conversely, if $x \leq y$, then $a = \fracsigma(x) \leq \fracsigma(y) = c$ and $d = \fracsigma(\comp{y}) \leq \fracsigma(\comp{x}) = b$ by the monotonicity of $\fracsigma$.
\end{proof}

  A crucial observation is that if $\alg{B}$ is normal, then
\begin{align*}
  x =\fracsigma (x) \cdot \comp{\fracsigma(\comp{x})} = \fracsigma (x) + \comp{\fracsigma(\comp{x})},
\end{align*}
  since $\comp{x} = \fracsigma(\comp{x}) \cdot \comp{\fracsigma(x)}$ implies $x = \fracsigma (x) + \comp{\fracsigma(\comp{x})}$. That is, the pair $\pair{\fracsigma(x)}{\fracsigma(\comp{x})}$ works \emph{both} as a multiplicative and and additive representation of $x$. Both multiplying and adding two normal representations can therefore be done na\"{\i}vely: if $\pair{a}{b}$ and $\pair{c}{d}$ are normal representations of $x$ and $y$, then
\begin{align*}
  x \cdot y & = (a \cdot c) \cdot \comp{(b+d)}, & & \text{and} & x + y & = (a + c) + \comp{(b \cdot d)}.
\end{align*}
  That is, $\pair{a}{b} \cdot \pair{c}{d} = \pair{a \cdot c}{b+d}$ is a multiplicative representation of $x \cdot y$ and $\pair{a}{b} + \pair{c}{d} = \pair{a+c}{b \cdot d}$ is an additive representation of $x + y$. The only catch here is that $\pair{a \cdot c}{b+d}$ and $\pair{a+c}{b \cdot d}$ need not be normal representations. To obtain normal representations of $x \cdot y$ and $x + y$, we need to further project them onto the set of normal representations.

\begin{fact} \label{fact: normalization}
  If $x = a \cdot \comp{b}$ in $\cdmalg{A}$, then $\pair{\fracsigma(x)}{\fracsigma(\comp{x})} = \pair{\transminus(a, b) \rightarrow \transplus(a, b)}{a \rightarrow b}$.
\end{fact}

\begin{proof}
  This follows immediately from the definitions: $a \cdot \comp{b} = \transplus(a, b) + \comp{\transminus(a, b)}$ and $\fracsigma(a + \comp{b}) = b \rightarrow a$.
\end{proof}

  Let us therefore define the map $\pi\colon \alg{A}^2 \to \alg{A}^2$ as
\begin{align*}
  \pi \multipair{a}{b} & \assign \multipair{\transminus(a, b) \rightarrow \transplus(a, b)}{a \rightarrow b}.
\end{align*}
  In other words, $\pi \multipair{a}{b}$ is a normal representation of $a \cdot \comp{b}$. Moreover, $\multipair{a}{b}$ is normal if and only if $\multipair{a}{b} = \pi \multipair{a}{b}$. It follows that normal pairs are definable by equations.

\begin{fact} \label{fact: normal transformation functions}
  The transformation functions $\transplus$, $\transminus$ for $\alg{A}$ are normal if and only if for all $a, b, x \in \alg{A}$
\begin{align*}
  a \rightarrow (b + x) & = (\transminus(a, b) \rightarrow \transplus(a, b)) \rightarrow ((a \rightarrow b) + x).
\end{align*}
\end{fact}

\begin{proof}
  This is precisely what the condition that $\multipair{a}{b}$ and $\pi \multipair{a}{b}$ represent the same element of $\cdmalg{A}$, i.e. $ a\cdot \comp{b}=(\transminus(a, b) \rightarrow \transplus(a, b)) \cdot \comp{a \rightarrow b}$, comes down to according to Lemma~\ref{lemma: joins below joins}.
\end{proof}

  Unlike general transformation functions, normal transformation functions are always unique if they exist. In a variety, (normal) transformation functions are always witnessed by certain terms.

\begin{definition}[Normal transformation terms]
  Let $\class{K}$ be a class of commutative residuated ($\ell$-)bimonoids. A pair of \emph{(normal) transformation terms} for $\class{K}$ is a pair of terms $t(x, y)$ and $u(x, y)$ in the language of commutative residuated ($\ell$-)bimonoids such that their interpretation on each ${\alg{A} \in \class{K}}$ is a pair of (normal) transformation functions for $\alg{A}$.
\end{definition}

\begin{fact} \label{fact: normal transformation terms}
  Let $\class{K}$ be an ordered variety of commutative residuated bimonoids (a variety of commutative residuated $\ell$-bimonoids). Every $\alg{A} \in \class{K}$ has a pair of (normal) transformation functions if and only if $\class{K}$ has a pair of (normal) transformation terms.
\end{fact}

\begin{proof}
  The right-to-left implication is trivial. Conversely, let $\alg{F}$ be a $\class{K}$-free ($\ell$-)bimonoid over $6$ or more generators, among them $x$ and $y$. Then $\alg{F}$ has certain (normal) transformation functions $\transplus$, $\transminus$. Applying them to $x$ and $y$ yields elements $\transplus(x, y), \transminus(x, y) \in \alg{F}$. These can be obtained by applying certain terms $t$ and $u$ to $x$ and $y$: $\transplus(x, y) = t(x, y)$ and $\transminus(x, y) = u(x, y)$. Because $\transplus$ and $\transminus$ are (normal) transformation functions, these satisfy the inequalities of Fact~\ref{fact: transformation conditions} (and Fact~\ref{fact: normal transformation functions}) if we take $a$, $b$, $p$, $q$ to be some of the other generators of $\alg{F}$. Since these inequalities hold in the $\class{K}$-free algebra $\alg{F}$, they hold in every algebra of $\class{K}$, therefore $t$ and $u$ are (normal) transformation terms.
\end{proof}

  We now show how to explicitly construct a commutative complemented bimonoid of fractions of $\alg{A}$ as a bimonoid on the set of all normal pairs in $\algAsquared$ using a pair of normal transformation functions $\transplus$ and $\transminus$.

\begin{theorem}[Constructing normal complemented bimonoids of fractions] \label{thm: normal bimonoids of fractions}
  Let $\alg{A}$ be a commutative residuated bimonoid with normal transformation functions $\transplus$ and $\transminus$. We define~$\fracalgA$ to be the ordered algebra over the set of all normal pairs
\begin{align*}
  \set{\multipair{a}{b} \in A^{2}}{\multipair{a}{b} = \pi \multipair{a}{b}} = \set{\multipair{\transminus(a, b) \rightarrow \transplus(a, b)}{a \rightarrow b}}{a, b \in \alg{A}}
\end{align*}
  with the operations
\settowidth\auxlength{$\,\preplus\,$}
\begin{align*}
  \comp{\multipair{a}{b}} & \assign \multipair{b}{a}, \\
  \multipair{a}{b} \hbox to \auxlength{\hfil$\precdot$\hfil} \multipair{c}{d} & \assign \pi \multipair{a \cdot c}{b+d}, \\
  \multipair{a}{b} \hbox to \auxlength{\hfil$\preplus$\hfil} \multipair{c}{d} & \assign \comp{\pi \multipair{b \cdot d}{a+c}},
\end{align*}
  the constants $\preone \assign \pi \multipair{1}{0}$ and $\prezero \assign \pi \multipair{0}{0}$, and the partial order
\begin{align*}
  \multipair{a}{b} \preleq \multipair{c}{d} & \iff a \leq c \text{ and } d \leq b.
\end{align*}
  If $\alg{A}$ is an $\ell$-bimonoid, then we also equip $\fracalgA$ with the operations
\begin{align*}
  \multipair{a}{b} \vee \multipair{c}{d} & \assign \pi \multipair{a \vee c}{b \wedge d}, \\
  \multipair{a}{b} \wedge \multipair{c}{d} & \assign \comp{\pi \multipair{b \wedge d}{a \vee c}}.
\end{align*}
  Then $\fracalgA$ is a normal commutative complemented ($\ell$-)bimonoid of fractions of~$\alg{A}$ relative to the embedding $\fraciota\colon a \mapsto \pi \multipair{a}{0}$. Moreover, $\fraciota(a) \circ \comp{\fraciota(b)} = \pi \multipair{a}{b}$.
\end{theorem}

\begin{proof}
  We have already observed (Fact~\ref{fact: normal negation}) that if $\multipair{a}{b}$ is normal, then so is $\multipair{b}{a}$, hence the operations of $\fracalgA$ indeed yield normal pairs. Because $\alg{A}$ has a pair of normal transformation functions, the sub-bimonoid $\alg{B}$ of $\cdmalg{A}$ consisting of elements of the form $a \comp{b}$ is a normal commutative complemented bimonoid of fractions of $\alg{A}$ (Proposition~\ref{prop: existence of bimonoid of fractions}). Now let us consider the map $\varepsilon\colon \fracalgA \to \alg{B}$ such that $\varepsilon\colon \multipair{a}{b} \mapsto a \comp{b} \in \alg{B}$.

  This map is surjective, since each $x \in \alg{B}$ has a normal representation. By Lemma~\ref{lemma: comparing normal representations}, $\multipair{a}{b} \preleq \multipair{c}{d}$ if and only if $\varepsilon \multipair{a}{b} \leq \varepsilon \multipair{c}{d}$. Observe that $\varepsilon (\pi \multipair{a}{b}) = a \comp{b}$: $\varepsilon(\pi \multipair{a}{b}) = \varepsilon \multipair{\transminus(a, b) \rightarrow \transplus(a, b)}{a \rightarrow b} = (\transminus(a, b) \rightarrow \transplus(a, b)) \cdot \comp{a \rightarrow b} = a \comp{b}$ by Fact~\ref{fact: normalization}. The map $\varepsilon$ is therefore a homomorphism: $\varepsilon(\preone) = \varepsilon(\pi \multipair{1}{0}) = 1 \comp{0} = 1$ and $\varepsilon(\prezero) = \varepsilon(\pi \multipair{0}{0}) = 0 \comp{0} = 0$ for the two units, $\varepsilon (\multipair{a}{b} \circ \multipair{c}{d}) = \varepsilon (\pi \multipair{a \cdot c}{b + d}) = ac \cdot \comp{b + d} = a\comp{b} \cdot c \comp{d} = \varepsilon \multipair{a}{b} \cdot \varepsilon \multipair{c}{d}$ for multiplication, and 
\begin{align*}
  \varepsilon \multipair{a}{b} + \varepsilon \multipair{c}{d}
  & = a \comp{b} + c \comp{d} = a + \comp{b} + c + \comp{d} \\
  & = (a + c) + \comp{(b \cdot d)} \\
  & = ((b \cdot d) \rightarrow (a + c)) \cdot \comp{(\transminus(b \cdot d, a + c) \rightarrow \transplus(b \cdot d, a + c))} \\
  & = \varepsilon (\multipair{(b \cdot d) \rightarrow (a + c)}{\transminus(b \cdot d, a + c) \rightarrow \transplus(b \cdot d, a + c)}) \\
  & = \varepsilon (\multipair{a}{b} \oplus \multipair{c}{d}).
\end{align*}
  Here we used Fact~\ref{fact: normalization} and the fact that for normal pairs $\multipair{a}{b}$ we have $a \cdot \comp{b} = a + \comp{b}$ in $\cdmalg{A}$. It follows that the map $\varepsilon\colon \fracalgA \to \alg{B}$ is an isomorphism of bimonoids. The complement of $\multipair{a}{b}$ in $\fracalgA$ is therefore $\varepsilon^{-1}(\comp{\varepsilon(\multipair{a}{b})}) = \varepsilon^{-1}(b + \comp{a}) = \varepsilon^{-1}(b \comp{a}) = \multipair{b}{a}$.

  Observe that $\varepsilon^{-1}(a)$ is by definition the unique normal representation of $a \in \alg{B}$, that is, $\pi \multipair{a}{b}$. Since $\alg{B}$ is a commutative complemented bimonoid of fractions of $\alg{A}$ relative to the inclusion of $\alg{A}$ into $\alg{B}$ and $\varepsilon$ is an isomorphism between $\fracalgA$ and $\alg{B}$, this implies that $\fracalgA$ is a commutative complemented bimonoid of fractions of $\alg{A}$ relative to the map $a \mapsto \pi \multipair{a}{0}$. Finally, $\varepsilon(\fraciota(a) \circ \comp{\fraciota(b)}) = \varepsilon(\fraciota(a)) \cdot \comp{\varepsilon(\fraciota(b))} = a \comp{b} = \varepsilon(\pi \multipair{a}{b})$, so $\fraciota(a) \circ \comp{\fraciota(b)} = \pi \multipair{a}{b}$.
\end{proof}

\begin{corollarynoname}
  Let $\alg{A}$ be $\ell$-bimonoid with a normal commutative complemented bimonoid of fractions. Then $\fracalgA$ is lattice-ordered. In fact, $\fracalgA$ is a sublattice of $\cdmalg{A}$ and $\cdmalg{A}$ is the DM-completion of $\fracalgA$. In~particular, if $\alg{A}$ is moreover finite, then $\cdmalg{A}=\fracalgA$.
\end{corollarynoname}

  If $\class{K}$ is a class of commutative residuated ($\ell$-)bimonoids with normal transformation terms, let $\fracclassK$ be the class of all commutative complemented ($\ell$-)bimonoids $\fracalgA$ equipped with the unary operation $\fracsigma$ for $\alg{A} \in \class{K}$:
\begin{align*}
  \fracclassK & \assign \set{\pair{\fracalgA}{\fracsigma}}{\alg{A} \in \class{K}}.
\end{align*}
  The class of all structures isomorphic to those in $\fracclassK$ will be denoted $\Iso(\fracclassK)$. This is precisely the class of all ($\ell$-)bimonoids of fractions of ($\ell$-)bimonoids in $\class{K}$, equipped with $\fracsigma$.

\begin{definition}[Normal complemented bimonoids with an interior operator]
  A commutative complemented ($\ell$-)bimonoid equipped with an interior operator $\pair{\alg{A}}{\sigma}$, and by extension the interior operator $\sigma$ itself, is called \emph{normal} if the image of $\sigma$ is a sub-($\ell$-)bimonoid $\sigma[\alg{A}]$ of $\alg{A}$ and moreover $a = \sigma(a) \cdot \sigma(\comp{a})$ for each $ a\in \alg{A}$.
\end{definition}

\begin{fact}
  If $\class{K}$ is an ordered variety (a variety) of commutative residuated ($\ell$-)bimonoids with normal bimonoids of fractions, then $\Iso(\fracclassK)$ is an ordered subvariety (a subvariety) of the ordered variety (the variety) of commutative complemented ($\ell$-)bimonoids with a normal interior operator.
\end{fact}

\begin{proof}
  The condition of being a normal interior operator is definable by a set of inequalities. Now suppose that $\class{K}$ is axiomatized relative to the class of commutative residuated ($\ell$-)bimonoids with normal bimonoids of fractions by the bimonoidal inequalities $t_{i}(x_1, \dots, x_n) \inequals u_{i}(x_1, \dots, x_n)$ for $i \in I$. Then $\Iso (\fracclassK)$ is the ordered subvariety of the ordered variety of normal commutative complemented bimonoids with an interior operator axiomatized by $t_{i}(\sigma(x_1), \dots, \sigma(x_n)) \inequals u_{i}(\sigma(x_1), \dots, \sigma(x_n))$ for $i \in I$. The proof for $\ell$-bimonoids is identical.
\end{proof}

  We now define the functor $\div\colon \class{K} \to \Iso(\fracclassK)$ as $\div(\alg{A}) \assign \pair{\fracalgA}{\fracsigma}$ on objects and for each homomorphism $h\colon \alg{A} \to \alg{B}$ in $\class{K}$ we take $\div(h)\colon \div(\alg{A}) \to \div(\alg{B})$ to be the map $\frachom{h}\colon \pair{a}{b} \mapsto \pair{h(a)}{h(b)}$. (This is map is well-defined because being a normal pair is defined by an equational condition and it is a homomorphism because the operations of $\fracalgA$ and $\fracalgB$ correspond to pairs of ($\ell$-)bimonoidal terms.)

  Conversely, there is a functor $\Sigma$ which to each commutative residuated ($\ell$-)bimonoid $\alg{A}$ with a normal interior operator $\sigma$ assigns the ($\ell$-)bimonoid $\sigma[\alg{A}]$ and to each homomorphism $h\colon \pair{\alg{A}}{\sigma} \to \pair{\alg{B}}{\tau}$ of such structures assigns its restriction to $\sigma[\alg{A}]$. Note that the functor $\Sigma$ is defined on the class of all normal commutative complemented ($\ell$-)bimonoids of fractions of commutative residuated ($\ell$-)bimonoids, while the functor $\div$ is only defined on the class of all commutative residuated ($\ell$-)bimonoids which have commutative complemented bimonoids of fractions witnessed by a specific pair of transformation terms.

\begin{theorem}[Categorical equivalences for bimonoids of fractions] \label{thm: equivalence}
  Let $\class{K}$ be an ordered variety of commutative residuated bimonoids with normal commutative complemented bimonoids of fractions. Then the functors $\div\colon \class{K} \to \Iso(\fracclassK)$ and $\Sigma\colon \Iso(\fracclassK) \to \class{K}$ form a categorical equivalence between $\class{K}$ and $\Iso(\fracclassK)$, the~unit being the map $a \mapsto \pi \multipair{a}{0}$ and the counit being the map $\multipair{a}{b} \mapsto a \comp{b}$.
\end{theorem}

\begin{proof}
  Let $\class{K}$ be axiomatized by the bimonoidal inequalities $t_{i}(x_1, \dots, x_n) \inequals u_{i}(x_1, \dots, x_n)$ for $i \in I$. Then $\Iso (\class{K})$ is the ordered subvariety of the ordered variety of normal commutative complemented bimonoids with an interior operator axiomatized by $t_{i}(\sigma(x_1), \dots, \sigma(x_n) \inequals u_{i}(\sigma(x_1), \dots, \sigma(x_n))$ for $i \in I$. By Fact~\ref{fact: normal transformation terms} there are normal transformation terms for $\class{K}$, therefore $\div\colon \class{K} \to \Iso (\fracclassK)$ is indeed a functor. We already know that $\Sigma$ is a functor. It now suffices to provide natural isomorphisms between $\alg{A}$ and $\Sigma(\div(\alg{A}))$ for $\alg{A} \in \class{K}$ and between $\div(\Sigma(\pair{\alg{B}}{\sigma}))$ and $\pair{\alg{B}}{\sigma}$ for $\pair{\alg{B}}{\sigma} \in \Iso(\fracclassK)$.

  Let $\alg{A} \in \class{K}$. The map $a \mapsto \pi \multipair{a}{0}$ is a bimonoidal embedding of $\alg{A}$ into $\fracalgA$ by Theorem~\ref{thm: normal bimonoids of fractions}. Its image coincides with the image of $\fracsigma$, i.e.\ with $\alg{A}$ as a sub-bimonoid of $\fracalgA$, therefore this map is a (natural) isomorphism between the bimonoids $\alg{A}$ and $\fracsigma[\fracalgA]$. On the other hand, each element of $\Iso(\fracclassK)$ is a normal commutative complemented bimonoid $\alg{B}$ with an interior operator~$\sigma$ such that $\sigma[\alg{B}] \in \class{K}$. Let $\alg{A} \assign \sigma[\alg{B}]$. By Theorem~\ref{thm: normal bimonoids of fractions} the map $\multipair{a}{b} \mapsto a \comp{b}$ is a (natural) isomorphism between $\fracof{\sigma[\alg{B}]}$ and $\alg{B}$.
\end{proof}

  The same proof yields an analogous theorem for varieties of commutative residuated $\ell$-bimonoids.

\begin{theorem}[Categorical equivalences for $\ell$-bimonoids of fractions] \label{thm: l-equivalence}
  Let $\class{K}$ be a variety of commutative residuated $\ell$-bimonoids with normal commutative complemented $\ell$-bi\-monoids of fractions. Then the functors $\div\colon \class{K} \to \Iso(\fracclassK)$ and $\Sigma\colon \Iso(\fracclassK) \to \class{K}$ form a categorical equivalence between $\class{K}$ and $\Iso(\fracclassK)$, the~unit being the map $a \mapsto \pi \multipair{a}{0}$ and the counit being the map $\multipair{a}{b} \mapsto x$.
\end{theorem}

\subsection{Applications}
\label{subsec: applications}

  We now apply the theory developed in the previous section to obtain uniform proofs of some new categorical equivalences, as well as alternative proofs of some known ones. We first use Theorem~\ref{thm: normal bimonoids of fractions} to show that certain varieties $\class{K}$ of commutative residuated ($\ell$-)bimonoids have normal complemented ($\ell$-)bimonoids of fractions. Then in each case we find an intrinsic inequational description of $\Iso (\fracclassK)$. Finally, this will yield a categorical equivalence between certain ordered varieties (varieties) by Theorem~\ref{thm: equivalence} (Theorem~\ref{thm: l-equivalence}).

  We consider three cases: Brouwerian semilattices, Boolean-pointed Brouwerian algebras, and a certain ordered variety of cancellative residuated pomonoids. We already saw that all of these classes of commutative residuated bimonoids have transformation terms (Facts~\ref{fact: fractions for brsls}, \ref{fact: fractions for boolean-pointed bras}, and~\ref{fact: fractions for cancellative pomonoids}). Moreover, in all three cases we can in fact use the \emph{same} transformation terms:
\begin{align*}
  a \cdot \comp{b} & = 0 a + \comp{a \rightarrow ab}.
\end{align*}
  This is because in the first two cases $a \rightarrow ab = a \rightarrow b$, in the first and last case $0 a = a$, and in the last case $a \rightarrow ab = b$. We~now show that these transformation terms are normal. This amounts to verifying the following equality in each of these classes (Fact~\ref{fact: normal transformation functions}):
\begin{align*}
  a \rightarrow (b + x) \equals ((a \rightarrow ab) \rightarrow 0 a) \rightarrow ((a \rightarrow b) + x).
\end{align*}

\begin{fact}
  Each Brouwerian semilattice has a normal commutative complemented bimonoid of fractions with normal transformation functions $\transplus(a, b) = a$ and $\transminus(a, b) = a \rightarrow b$.
\end{fact}

\begin{proof}
  It suffices to verify that $a \rightarrow bx = ((a \rightarrow b) \rightarrow a) \rightarrow (a \rightarrow b) x$. This is routine.
\end{proof}

\begin{fact}\label{f:BpBr}
  Each Boolean-pointed Brouwerian algebra has a normal commutative complemented bimonoid of fractions with normal transformation functions $\transplus(a, b) = 0 a$ and $\transminus(a, b) = a \rightarrow b$.
\end{fact}

\begin{proof}
  We need to show that $a \rightarrow (b + x) = ((a \rightarrow b) \rightarrow 0a) \rightarrow ((a \rightarrow b) + x)$. Let $l$ and $r$ be the left- and the right-hand side of this equation. By Lemma~\ref{lemma: validity in bpbras} it suffices to show that $a \rightarrow bx = ((a \rightarrow b) \rightarrow a) \rightarrow (a \rightarrow b) x$ holds in all Brouwerian algebras and $l \vee 0 = r \vee 0$ holds in all Boolean-pointed Brouwerian algebras. The former inequality is routine to prove. The inequality $l \leq r \vee 0$ is equivalent to the conjunction of the inequalities
\begin{align*}
  (0 a \rightarrow bx) (a \rightarrow (b \vee x)) ((a \rightarrow b) \rightarrow 0a) & \leq (0 \rightarrow (a \rightarrow b) x) \vee 0, \\
  (0 a \rightarrow bx) (a \rightarrow (b \vee x)) ((a \rightarrow b) \rightarrow 0a) & \leq (a \rightarrow b) \vee x \vee 0.
\end{align*}
  The first is routine and the second is equivalent to $(a \rightarrow (b + x)) ((a \rightarrow b) \rightarrow 0a) ((a \rightarrow b) \rightarrow 0) (x \rightarrow 0) \leq 0$, since $y \vee 0 = (y \rightarrow 0) \rightarrow 0$ in Boolean-pointed Brouwerian algebras. We have $(a \rightarrow b) \rightarrow 0 a \leq b \rightarrow 0$ and $(b \rightarrow 0) (a \rightarrow (b \vee x)) \leq a \rightarrow (0 \vee x)$ and $(x \rightarrow 0) (a \rightarrow (0 \vee x)) \leq a \rightarrow 0$. But $(a \rightarrow 0) (0a \rightarrow bx) \leq a \rightarrow b$ and $(a \rightarrow b) ((a \rightarrow b) \rightarrow 0) \leq 0$.

  It remains to prove that $r \leq l \vee 0$. This is equivalent to $(b \rightarrow 0) \rightarrow (b + x) \leq (b + x) \vee 0$, i.e.\ to the conjunction of $((b \rightarrow 0) \rightarrow (0 \rightarrow bx) (b \vee x)) \leq (0 \rightarrow bx) \vee 0$ and $((b \rightarrow 0) \rightarrow (0 \rightarrow bx) (b \vee x)) \leq b \vee x \vee 0$. The first inequality is routine, and the second is equivalent to $((b \rightarrow 0) \rightarrow (0 \rightarrow bx) (b \vee x)) (b \rightarrow 0) (x \rightarrow 0) \leq 0$. But $((b \rightarrow 0) \rightarrow (0 \rightarrow bx) (b \vee x)) (b \rightarrow 0) (x \rightarrow 0) \leq (b \vee x) (b \rightarrow 0) (x \rightarrow 0) \leq 0$.
\end{proof}

  Finally, while every commutative order-cancellative pomonoid (viewed as a bimonoid with $x + y \assign x \cdot y$ and $0 \assign 1$) has a commutative complemented bimonoid of fractions, namely its partially ordered group of fractions, to obtain a normal bimonoid of fractions we need to restrict to a certain subclass of order-cancellative pomonoids. A residuated pomonoid is called \emph{divisible} if it satisfies the implication
\begin{align*}
  y \inequals x & \implies x \cdot (x \bs y) \equals y \equals (y / x) \cdot x.
\end{align*}
  An integral residuated pomonoid is divisible if and only if $x \cdot (x \bs y) = x \wedge y = (y / x) \cdot x$ for all $x$ and $y$. This~condition can be expressed by a set of inequalities, therefore divisible integral residuated pomonoids form an ordered subvariety of the ordered variety of integral residuated pomonoids. In this context, $x \wedge y$ may thus be treated as an abbreviation for $x \cdot (x \bs y) = (y / x) \cdot x$. Each divisible integral residuated pomonoid satisfies the equations $x (y \wedge z) \equals xy \wedge xz$ and $(x \wedge y) z \equals xz \wedge yz$ (see~\cite{galatos+tsinakis05}). It follows that a divisible integral residuated pomonoid is order-cancellative if and only if it is cancellative.

\begin{fact} \label{fact: fractions of lgminus}
  Each cancellative divisible integral commutative residuated pomonoid has a normal commutative complemented bimonoid of fractions with normal transformation functions $\transplus(a, b) = a$ and $\transminus(a, b) = b$.
\end{fact}

\begin{proof}
  We need to verify that $a \rightarrow bx = (b \rightarrow a) \rightarrow ((a \rightarrow b) x)$:
\begin{align*}
  a (b \rightarrow a) (a \rightarrow bx) & = (b \rightarrow a) (a \wedge bx) \\
  & = (b \rightarrow a) a \wedge (b \rightarrow a) b x \\
  & = (b \rightarrow a) a \wedge (a \wedge b) x \\
  & = (b \rightarrow a) a \wedge a (a \rightarrow b) x \\
  & = a ((b \rightarrow a) \wedge (a \rightarrow b) x) \\
  & = a (b \rightarrow a) ((b \rightarrow a) \rightarrow (a \rightarrow b) x).
\end{align*}
  The required equation now follows by cancellativity.
\end{proof}

  Up to term equivalence these pomonoids in fact form a variety of residuated lattices known as integral cancellative commutative GMV-algebras (see~\cite{galatos+tsinakis05}).

  Let us now provide an intrinsic inequational description of the complemented bimonoids of fractions obtained from the above three classes of ($\ell$-)bimonoids.

\begin{fact}
  The commutative complemented bimonoids of fractions of Brouwerian semilattices are precisely the idempotent involutive commutative residuated pomonoids equipped with the map $\sigma(x) \assign 1 \wedge x$ which satisfy $x \equals \sigma(x) \cdot \comp{\sigma(\comp{x})}$.
\end{fact}

\begin{proof}
  Let $\alg{A}$ be a Brouwerian semilattice. Then $\fracalgA$ is idempotent: $a \comp{b} a \comp{b} = a a \cdot \comp{b+b} = a \comp{b}$. To show that $\fracsigma(x) = 1 \wedge x$, it suffices to prove that the bimonoidal embedding $a \mapsto \pi \multipair{a}{0} = \multipair{a}{1}$ maps $\alg{A}$ onto the negative cone of $\fracalgA$, which is the principal filter generated by $\multipair{1}{1}$. But $\multipair{a}{b} \leq \multipair{1}{1}$ if and only if $a \leq 1$ and $1 \leq b$, or equivalently if and only if $b = 1$. Such pairs by definition form the image of $a \mapsto \multipair{a}{1}$.

  Conversely, let $\alg{B}$ be an idempotent involutive commutative residuated pomonoid equipped with the map $\sigma(x) \assign 1 \wedge x$ which satisfies $x \equals \sigma(x) \cdot \comp{\sigma(\comp{x})}$. This equation states that $\alg{B}$ is a commutative bimonoid of fractions of its negative cone, which is an integral idempotent residuated pomonoid with respect to both multiplication and addition, since $0 = 1$.
\end{proof}

\begin{fact}
  The commutative complemented bimonoids of fractions of Boolean-pointed Brouwerian algebras are precisely the idempotent involutive commutative residuated lattices which satisfy $x \equals (1 \wedge x) (0 \vee x)$, expanded by $\sigma(x) \assign 1 \wedge x$.
\end{fact}

\begin{proof}
  Let $\alg{A}$ be a Boolean-pointed Brouwerian algebra. Then $\fracalgA$ is again idempotent. To show that $\fracsigma(x) = 1 \wedge x$, it suffices to prove that the bimonoidal embedding $a \mapsto \pi \multipair{a}{0}$ maps $\alg{A}$ onto the negative cone of $\fracalgA$, which is the principal filter generated by $\multipair{1}{0}$. The into part of this claim follows from the integrality of $\alg{A}$. To prove the onto part, suppose that $\multipair{a}{b} \leq \multipair{1}{0}$, i.e.\ $a \leq 1$ and $0 \leq b$. Since the interval $[0, 1]$ is Boolean, $\comp{b} \in \cdmalg{A}$ is in fact an element of $\alg{A}$ (which is a sub-bimonoid of $\cdmalg{A}$), say $\comp{b} = c \in \alg{A}$. But then $a \comp{b} = a c \in \alg{A}$, so $\multipair{a}{b}$ represents the element $ac \in \alg{A}$. That is, $\multipair{a}{b} = \pi \multipair{ac}{0}$. The equation $x \equals (1 \wedge x) (0 \vee x)$ now follows from the fact that $x = \fracsigma(x) \cdot \comp{\fracsigma(\comp{x})}$.

  Conversely, let $\alg{B}$ be an idempotent involutive commutative residuated lattice satisfying $x \equals (1 \wedge x) (0 \vee x)$. This equation states precisely that $\alg{B}$ is a normal bimonoid of fractions of its negative cone. Moreover, $1 = (1 \wedge 1) (1 \vee 0) = (1 \vee 0)$, so $0$ lies in the negative cone. The negative cone of $\alg{B}$ is closed under multiplication because $1 + 1 = 1$, it is therefore an integral idempotent residuated $\ell$-bimonoid, i.e.\ a pointed Brouwerian algebra. The interval $[0, 1]$ of $\alg{B}$ is then a complemented bi-integral idempotent bimonoid, i.e.\ a Boolean lattice.
\end{proof}

  Our construction of complemented bimonoids of fractions of Boolean-pointed Brouwerian algebras extends a construction of Fussner \& Galatos~\cite{fussner+galatos19} for semilinear Boolean-pointed Brouwerian algebras. Recall that a (Boolean-pointed) Brouwerian algebra is \emph{semilinear} if it is a subdirect product of (Boolean-pointed) Brouwerian chains, or equivalently if it satisfies the equation $(x \rightarrow y) \vee (y \rightarrow x) \equals 1$. Such algebras are called \emph{relative Stone algebras} in~\cite{fussner+galatos19}. In our terminology, Fussner \& Galatos prove the following fact.

\begin{fact}
  \newlength{\origdimen}
  \setlength{\origdimen}{\fontdimen2\font}
  \setlength{\fontdimen2\font}{0.98\origdimen}
  The commutative complemented bimonoids of fractions of semilinear Boolean-pointed Brouwerian algebras are precisely Sugihara monoids (distributive commutative idempotent involutive residuated lattices).
  \setlength{\fontdimen2\font}{\origdimen}
\end{fact}

  For cancellative commutative residuated pomonoids, we obtain precisely (the restriction to the com\-mutative case of) the result of Montagna \& Tsinakis~\cite{montagna+tsinakis10}. Recall that a conucleus on a pomonoid $\alg{A}$ is an interior operator $\sigma$ such that $\sigma(x) \cdot \sigma(y) \leq \sigma(x \cdot y)$ and $\sigma(1) = 1$. The image $\sigma[\alg{A}]$ of a conucleus is therefore a submonoid of $\alg{A}$. Moreover, if $\alg{A}$ is a residuated pomonoid, then so is $\sigma[\alg{A}]$, the residuals in $\sigma[\alg{A}]$ being the $\sigma$-images of residuals in $\alg{A}$: $a \bs_{\sigma[\alg{A}]} b = \sigma(a \bs_{\alg{A}} b)$ and $a /_{\!\sigma[\alg{A}]} b = \sigma(a /_{\!\alg{A}} b)$.

\begin{fact}
  The commutative complemented bimonoids of fractions of cancellative com\-mutative residuated pomonoids (residuated lattices) are precisely Abelian pogroups ($\ell$-groups) with a conucleus $\sigma$ which are groups of fractions of the image of~$\sigma$.
\end{fact}

\begin{proof}
  We have already verified that the commutative complemented bimonoid of fractions is a partially ordered Abelian group in this case (Fact~\ref{fact: fractions for cancellative pomonoids}). We only need to verify that $\fracsigma(x) \cdot \fracsigma(y) \leq \fracsigma(x \cdot y)$:
\begin{align*}
  \fracsigma(a \comp{b}) \cdot \fracsigma(c \comp{d}) & = (b \rightarrow a) (d \rightarrow c) \leq bd \rightarrow ac = \fracsigma(ac \cdot \comp{bd}) = \fracsigma(ac \cdot \comp{b+d}) = \fracsigma(a\comp{b} \cdot c \comp{d}). \qedhere
\end{align*}
\end{proof}

  The reader will immediately observe that this result has a less satisfactory form than the previous ones for Brouwerian semilattices and Boolean-pointed algebra. This is because complemented bimonoids of fractions of cancellative com\-mutative residuated pomonoids need not be normal. One need not consider any exotic examples to see this: the complemented bimonoid of fractions of an Abelian $\ell$-group $\alg{G}$ is of course $\alg{G}$ itself with the identity map as $\sigma$. But this is never a normal bimonoid of fractions, unless $\alg{G}$ is trivial: the~normality equation $x \equals \sigma(x) \cdot \comp{\sigma(\comp{x})}$ becomes $x \equals x \cdot x$. To obtain normal groups of fractions, we must restrict to the following case.

\begin{fact} \label{fact: normal sigma}
  Let $\alg{G}$ be an Abelian $\ell$-group equipped with a conucleus $\sigma$. Then $\alg{G}$ is a normal commutative complemented bimonoid of fractions of $\sigma[\alg{G}]$ with respect to the inclusion in $\alg{G}$ if and only if $\sigma(x) = 1 \wedge x$.
\end{fact}

\begin{proof}
  Every $\ell$-group satisfies the equation $x \equals (1 \wedge x) (1 \wedge x^{-1})^{-1}$. Conversely, the normality condition states that $a b^{-1} = (b \rightarrow a) (a \rightarrow b)^{-1}$, i.e.\ $a (a \rightarrow b) = b (b \rightarrow a)$, for $a, b \in \sigma[\alg{G}]$, where the residuals are taken in $\sigma[\alg{G}]$. But then $b (b \rightarrow a) \leq b$, so $b \rightarrow a \leq b \rightarrow b$ and taking $b \assign 1$ yields that $a \leq 1$ for each $a \in \sigma[\alg{G}]$. Conversely, consider some $x = a b^{-1} \leq 1$. Then $x = \sigma(x) \sigma(x^{-1})^{-1} = (b \rightarrow a) (a \rightarrow b)^{-1}$. But $a b^{-1} \leq 1$ implies $a \leq b$, so $a \rightarrow b = 1$ in $\sigma[\alg{G}]$. Thus $x = b \rightarrow a \in \sigma[\alg{G}]$. This proves that $\sigma[\alg{G}]$ is precisely the negative cone of $\alg{G}$. The claim that $\sigma(x) = 1 \wedge x$ now follows.
\end{proof}

  The negative cones of $\ell$-groups, images of $\ell$-groups with respect to the conucleus $\sigma(x) = 1 \wedge x$, were already described by Bahls et al.~\cite{bahls+et+al03} as integral cancellative divisible residuated lattices. Recall that an integral residuated lattice is \emph{divisible} if it satisfies the equations $x \cdot (x \bs y) = x \wedge y = (y / x) \cdot x$.

\begin{fact}
  The commutative complemented bimonoids of fractions of cancellative integral divisible com\-mutative residuated lattices are precisely Abelian $\ell$-groups equipped with the map $\sigma(x) \assign 1 \wedge x$.
\end{fact}

\begin{proof}
  By Fact~\ref{fact: fractions of lgminus} such bimonoids of fractions are normal complemented bimonoids of fractions and moreover they are $\ell$-groups. By Fact~\ref{fact: normal sigma} normality implies that $\sigma(x) = 1 \wedge x$. Conversely, one merely verifies the negative cone of each Abelian $\ell$-group is indeed a cancellative and divisible residuated lattice.
\end{proof}

  Since in each case the interior operator $\fracsigma$ was simply the projection onto the negative cone, we immediately obtain the following corollaries. The last one was already proved in~\cite{bahls+et+al03}.

\begin{corollary}[Brouwerian semilattices as negative cones]
  Brouwerian semilattices are precisely the negative cones of idempotent involutive commutative residuated pomonoids satisfying $0 \equals 1$ where $1 \wedge x$ exists for each $x$.
\end{corollary}

\begin{corollary}[Boolean-pointed Brouwerian algebras as negative cones]
  Boolean-pointed Brouwerian algebras (Brouwerian algebras) are precisely the negative cones of idempotent involutive commutative residuated lattices (satisfying $0 \equals 1$).
\end{corollary}

\begin{proof}
  It only remains to observe that the negative cone of an idempotent involutive commutative residuated lattice (satisfying $0 \equals 1$), not necessarily satisfying $x \equals (1 \wedge x) \cdot (0 \vee x)$, is a Boolean-pointed Brouwerian algebra (a Brouwerian algebra).
\end{proof}

\begin{corollary}[Negative cones of $\ell$-groups]
  The negative cones of $\ell$-groups are precisely the cancellative divisible integral commutative residuated lattices.
\end{corollary}

  Finally, let us state the categorical equivalences that Theorems~\ref{thm: equivalence} and \ref{thm: l-equivalence} now yield.

\begin{theorem}[Categorical equivalence: Brouwerian semilattices]
  The variety of Brouwerian semilattices is categorically equivalent to the variety of commutative idempotent involutive residuated pomonoids which satisfy $0 \equals 1$ equipped with the map $\sigma(x) \assign 1 \wedge x$.
\end{theorem}

\begin{theorem}[Categorical equivalence: Boolean-pointed Brouwerian algebras] \label{thm: equivalence for brbpas}
  The variety of Boolean-pointed Brouwerian algebras (Brouwerian algebras) is categorically equivalent to the variety of commutative idempotent involutive residuated lattices which satisfy ($0 \equals 1$ and) $x \equals (1 \wedge x) \cdot (0 \vee x)$.
\end{theorem}

\begin{theorem}[Categorical equivalence: Abelian $\ell$-groups]
  The variety of Abelian $\ell$-groups is categorically equivalent to the variety of cancellative divisible integral commutative residuated lattices.
\end{theorem}

  In the case of Brouwerian semilattices and Boolean-pointed Brouwerian algebras, one can easily formulate bounded variants of the above results, adding on each side the constant $\bot$ and $\top$ interpreted as the bottom and top element. It is entirely routine to verify that all of our proofs then still go through. Thus we obtain a similar equivalence between Boolean-pointed \emph{Heyting} algebras and \emph{bounded} idempotent involutive residuated lattices which satisfy $x \equals (1 \wedge x) \cdot (0 \vee x)$.

  The equivalence for Abelian $\ell$-groups is a restriction to the commutative case of the equivalence between cancellative divisible integral residuated lattices and arbitrary $\ell$-groups due to Bahls et al.~\cite{bahls+et+al03}. On the other hand, the equivalence for Boolean-pointed Brouwerian algebras is an extension of the equivalence between semilinear Boolean-pointed Brouwerian algebras and Sugihara monoids due to Fussner \& Galatos~\cite{fussner+galatos19}, which in turn extends the equivalence between semilinear Brouwerian algebras and odd Sugihara monoids due to Galatos \& Raftery~\cite{galatos+raftery12}. (A Sugihara monoid is odd if it satisfies $0 \equals 1$.)

  The equivalences for Boolean-pointed Brouwerian algebras and Abelian $\ell$-groups are in fact restrictions of a \emph{single} equivalence between a variety of commutative residuated $\ell$-bimonoids and a variety of commutative involutive residuated lattices. The variety of commutative residuated $\ell$-bimonoids is axiomatized by the equations stating that $\transplus(a,b) \assign 0 \cdot a$ and $\transminus(a,b) \assign a \rightarrow ab$ are normal transformation functions. (These equations are described in Facts~\ref{fact: normal transformation functions} and~\ref{fact: transformation conditions}.) The variety of commutative involutive residuated lattices is axiomatized by the equation $x \equals (1 \wedge x) \cdot \comp{(1 \wedge \comp{x})}$, i.e.\ $x \equals (1 \wedge x) \cdot (0 \vee x)$. Similarly, the categorical equivalences for Brouwerian semilattices and Abelian $\ell$-groups are restrictions of a single equivalence between an ordered variety of commutative residuated bimonoids and an ordered variety of commutative involutive residuated pomonoids equipped with a normal interior operator $\sigma$. The former ordered variety is again axiomatized by the inequalities expressing that $\transplus(a,b) \assign 0 \cdot a$ and $\transminus(a,b) \assign a \rightarrow ab$ are normal transformation functions. The latter is axiomatized by $x \equals \sigma(x) \cdot \comp{\sigma(x)}$ and by inequalities stating that $\sigma(x)$ is the meet of $1$ and $x$. 

  This~is somewhat remarkable, since Brouwerian algebras and Abelian $\ell$-groups in a sense represent opposite ends of the residuated lattice spectrum: $\ell$-groups do not contain any non-trivial idempotents, while every element is idempotent in a Brouwerian algebra.

  Note that Abelian $\ell$-groups play a dual role here: we can either see them as involutive residuated lattices or as involutive residuated pomonoids equipped with a map $\sigma$ which projects onto the negative cone. These are term equivalent ways of looking at Abelian $\ell$-groups, since $x \wedge y = x \cdot \sigma (x^{-1} y)$.

  The above categorical equivalences allow us to transfer categorical properties from the well-studied category of Brouwerian (Heyting) algebras to the category (bounded) commutative idempotent involutive residuated lattices which satisfy $0 \equals 1$ and $x \equals (1 \wedge x) \cdot (0 \vee x)$. For example, Maksimova~\cite{maksimova77} proved that there are exactly three (seven) non-trivial varieties of Brouwerian (Heyting) algebras with the amalgamation property. There are thus also exactly three (seven) non-trivial subvarieties of the above variety of (bounded) involutive residuated lattices which enjoy the amalgamation property, namely those whose negative cone lies in the varieties described by Maksimova.

\subsection{Examples of complemented bimonoids of fractions}
\label{sec: examples of bimonoids of fractions}

  We now describe some examples of complemented bimonoids of fractions of Boolean-pointed Brouwerian algebras. Consider the smallest non-trivial Brouwerian algebra: the two-element chain $\elembot < \elemone$. Taking $0 \assign \elembot$ yields a bimonoid with $x \cdot y = x \wedge y$ and $x + y = x \vee y$, which is already complemented. Taking $0 \assign \elemone$ yields a bimonoid with operations $x \cdot y = x \wedge y = x + y$. This bimonoids has precisely three normal pairs, ordered as follows: $\multipair{\elembot}{\elemone} \preleq \multipair{\elemone}{\elemone} \preleq \multipair{\elemone}{\elembot}$. The algebras of fractions is thus a linearly ordered commutative idempotent involutive residuated lattice, in other words a Sugihara chain. It is isomorphic to the three-element Sugihara chain with the universe $-1 < 0 < 1$ and the operations $\comp{x} = - x$ and
\begin{align*}
  x \cdot y & \assign \begin{cases}x \text{ if } | x | > | y | \text{ or } (| x | = | y | \text{ and } x \leq y), \\ y \text{ if } | x | < | y | \text{ or } (| x | = | y | \text { and } x \geq y), \end{cases} &
  x + y & \assign \begin{cases}x \text{ if } | x | > | y | \text{ or } (| x | = | y | \text{ and } x \geq y), \\ y \text{ if } | x | < | y | \text{ or } (| x | = | y | \text { and } x \leq y).\end{cases}
\end{align*}
  More generally, the algebra of fractions of the $n$-element Brouwerian chain where $0$ is the top element is the Sugihara chain $- n < \dots < -1 < 0 < 1 < \dots < n$ with the above operations. The Sugihara chains $- n < \dots < -1 < 1 < \dots < n$ with the same operations are obtained as bimonoids of fractions of $n$-element Brouwerian chains where $0$ is the coatom (corresponding to the element $-1$) rather than the top element. These examples, however, are already covered by the existing constructions of Galatos \& Raftery~\cite{galatos+raftery12,galatos+raftery14} and Fussner~\&~Galatos~\cite{fussner+galatos19} for semilinear (Boolean-pointed) Brouwerian algebras.

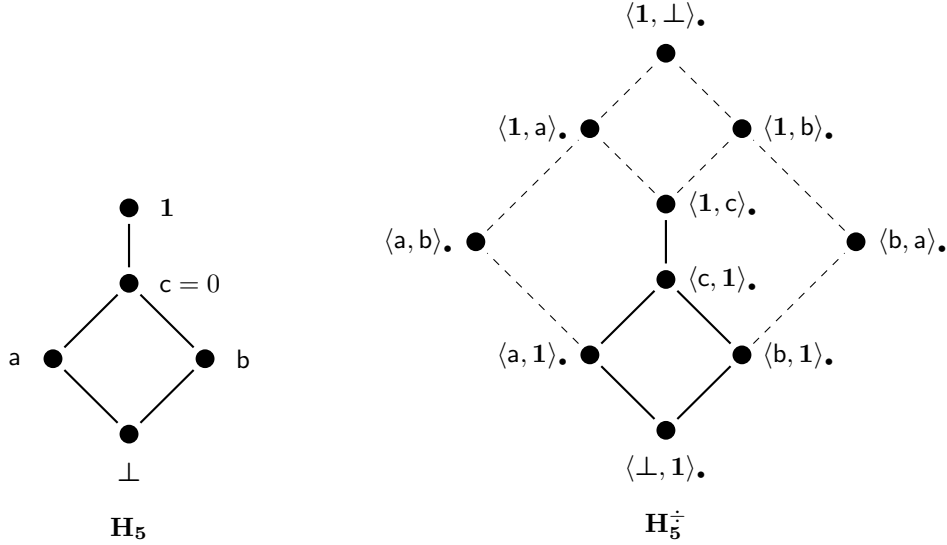
\begin{figure}
\caption{The Boolean-pointed Brouwerian algebra $\Hfive$ and its complemented bimonoid of fractions}
\label{fig: hfive 0=c}
\begin{center}
\begin{tikzpicture}[scale=1,dot/.style={circle,fill,inner sep=2.5pt,outer sep=2.5pt}]
  \node (bot) at (0, 0) [dot] {};
  \node (a) at (-1, 1) [dot] {};
  \node (b) at (1, 1) [dot] {};
  \node (c) at (0, 2) [dot] {};
  \node (1) at (0, 3) [dot] {};
  \draw[-,thick] (bot) -- (a) -- (c) -- (1);
  \draw[-,thick] (bot) -- (b) -- (c);
  \node (nbot) at (0, -0.5) {$\elembot$};
  \node (na) at (-1.5, 1) {$\elem{a}$};
  \node (nb) at (1.5, 1) {$\elem{b}$};
  \node (nc) at (0.8, 2) {$\elem{c} = 0$};
  \node (n1) at (0.5, 3) {$\elemone$};
  \node at (0, -1.25) {$\Hfive$};
\end{tikzpicture}
\qquad \qquad
\begin{tikzpicture}[scale=1,dot/.style={circle,fill,inner sep=2.5pt,outer sep=2.5pt}]
  \node (bot) at (0, 0) [dot] {};
  \node (a) at (-1, 1) [dot] {};
  \node (b) at (1, 1) [dot] {};
  \node (c) at (0, 2) [dot] {};
  \node (1) at (0, 3) [dot] {};
  \draw[-,thick] (bot) -- (a) -- (c) -- (1);
  \draw[-,thick] (bot) -- (b) -- (c);
  \node (nega) at (-1, 4) [dot] {};
  \node (negb) at (1, 4) [dot] {};
  \node (top) at (0, 5) [dot] {};
  \node (leftside) at (-2.5, 2.5) [dot] {};
  \node (rightside) at (2.5, 2.5) [dot] {};
  \draw[dashed] (1) -- (nega) -- (top);
  \draw[dashed] (1) -- (negb) -- (top);
  \draw[dashed] (a) -- (leftside) -- (nega);
  \draw[dashed] (b) -- (rightside) -- (negb);
  \node (nbot) at (0, -0.5) {$\multipair{\elembot}{\elemone}$};
  \node (na) at (-1.75, 1) {$\multipair{\elem{a}}{\elemone}$};
  \node (nb) at (1.75, 1) {$\multipair{\elem{b}}{\elemone}$};
  \node (nc) at (0.75, 2) {$\multipair{\elem{c}}{\elemone}$};
  \node (n1) at (0.75, 3) {$\multipair{\elemone}{\elem{c}}$};
  \node (nleftside) at (-3.25, 2.5) {$\multipair{\elem{a}}{\elem{b}}$};
  \node (nrightside) at (3.25, 2.5) {$\multipair{\elem{b}}{\elem{a}}$};
  \node (nnega) at (-1.75, 4) {$\multipair{\elemone}{\elem{a}}$};
  \node (nnegb) at (1.75, 4) {$\multipair{\elemone}{\elem{b}}$};
  \node (ntop) at (0, 5.5) {$\multipair{\elemone}{\elembot}$};
  \node at (0, -1.25) {$\fracHfive$};
\end{tikzpicture}
\end{center}
\end{figure}

  The smallest Brouwerian lattice not covered by these existing constructions, i.e.\ the smallest non-semilinear Brouwerian lattice, is shown in Figure~\ref{fig: hfive 0=c}. It can be expanded into a Boolean-pointed Brouwerian lattice in two different ways: either $0 \assign \elem{c}$ or $0 \assign \elemone$. Let us consider the first of these expansions. We~shall call the resulting commutative residuated $\ell$-bimonoid $\Hfive$. \nopagebreak Recall that addition in $\Hfive$ is defined as $x + y \assign (0 \rightarrow xy) (x \vee y)$. In particular, $x + y = x \wedge y$ for $x, y \leq 0$ and $x + y = x \vee y$ for $x, y \geq 0$. The only values not covered by these two clauses are the sums $1 + x$ or $x + 1$ for $x < 0$: in that  case $x + 1 = 1 + x = x$.

  The elements of $\fracHfive$ are precisely the normal pairs $\multipair{a}{b} \in \Hfivesquared$, i.e.\ pairs such that $\multipair{a}{b} = \pi \multipair{a}{b} \assign \multipair{\transminus(a, b) \rightarrow \transplus(a, b)}{a \rightarrow b}$, where $\transplus(a, b) = 0a$ and $\transminus(a, b) = a \rightarrow b$. The first step in describing $\fracHfive$ is therefore to find all pairs $\multipair{a}{b}$ such that
\begin{align*}
  a & = (a \rightarrow b) \rightarrow 0a, \\
  b & = a \rightarrow b.
\end{align*}
  These are precisely the pairs shown in Figure~\ref{fig: hfive 0=c}, with $\multipair{a}{b} \preleq \multipair{c}{d}$ if and only if $a \leq c$ and $d \leq b$. The~question is now how to succinctly describe the operations of $\fracHfive$ on these pairs.

  The lattice operations, of course, are determined by the order, and complementation is simply the map $\multipair{a}{b} \mapsto \multipair{b}{a}$. The monoidal operations are defined in terms of the monoidal operations of $\Hfive$ and the projection map $\pi$. The information in Table~\ref{tab: hfive 0 = c} therefore suffices to compute any product or sum in $\fracHfive$. (The~vertical axis represents $x$ and the horizontal axis represents $y$.) For~example, $\multipair{\elem{a}}{\elem{b}} \circ \multipair{\elem{b}}{\elem{a}} = \pi \multipair{\elem{a} \cdot \elem{b}}{\elem{b} + \elem{a}} = \pi \multipair{\elem{a} \wedge \elem{b}}{\elem{a} \wedge \elem{b}} = \pi \multipair{\elembot}{\elembot} = \multipair{\elembot}{\elemone}$, while $\multipair{\elem{a}}{\elem{b}} \oplus \multipair{\elemone}{\elem{c}} = \comp{\pi \multipair{\elem{b} \cdot \elem{c}}{\elem{a} + \elemone}} = \comp{\pi \multipair{\elem{b}}{\elem{a}}} = \comp{\multipair{\elem{b}}{\elem{a}}} = \multipair{\elem{a}}{\elem{b}}$ and $\multipair{\elem{a}}{\elem{b}} \circ \multipair{\elem{a}}{\elemone} = \pi \multipair{\elem{a} \cdot \elem{c}}{\elem{b} + \elemone} = \pi \multipair{\elem{a}}{\elem{b}} = \multipair{\elem{a}}{\elem{b}}$.

\begin{table}
\caption{Projection onto normal pairs in $\Hfive$}
\label{tab: hfive 0 = c}
\begin{center}
\begin{tabular}{c | c c c c c}
  $\pi \multipair{x}{y}$ & $\elembot$ & $\elem{a}$ & $\elem{b}$ & $\elem{c}$ & $\elemone$ \\
\hline
  $\elembot$ & $\multipair{\elembot}{\elemone}$ & $\multipair{\elembot}{\elemone}$ & $\multipair{\elembot}{\elemone}$ & $\multipair{\elembot}{\elemone}$ & $\multipair{\elembot}{\elemone}$ \\
  $\elem{a}$ & $\multipair{\elembot}{\elemone}$ & $\multipair{\elem{a}}{\elemone}$ & $\multipair{\elem{a}}{\elem{b}}$ & $\multipair{\elem{a}}{\elemone}$ & $\multipair{\elem{a}}{\elemone}$ \\
  $\elem{b}$ & $\multipair{\elem{b}}{\elem{a}}$ & $\multipair{\elem{b}}{\elem{a}}$ & $\multipair{\elem{b}}{\elemone}$ & $\multipair{\elem{b}}{\elemone}$ & $\multipair{\elem{b}}{\elemone}$ \\
  $\elem{c}$ & $\multipair{\elemone}{\elembot}$ & $\multipair{\elemone}{\elem{a}}$ & $\multipair{\elemone}{\elem{b}}$ & $\multipair{\elem{c}}{\elemone}$ & $\multipair{\elem{c}}{\elemone}$ \\
  $\elemone$ & $\multipair{\elemone}{\elembot}$ & $\multipair{\elemone}{\elem{a}}$ & $\multipair{\elemone}{\elem{b}}$ & $\multipair{\elemone}{\elem{c}}$ & $\multipair{\elem{c}}{\elemone}$
\end{tabular}
\end{center}
\end{table}

  Recall that the complemented bimonoid of fractions of a semilinear Boolean-pointed Brouwerian algebra is distributive, by the results of Fussner \& Galatos~\cite{fussner+galatos19}. On the other hand, even the simplest non-semilinear Boolean-pointed Brouwerian algebra $\Hfive$ has a non-modular complemented bimonoid of fractions~$\fracHfive$. It~thus seems natural to ask whether this holds for all Boolean-pointed Brouwerian algebras. That is, is the complemented bimonoid of fractions of a Boolean-pointed Brouwerian algebras modular if and only if it is distributive, perhaps by virtue of complemented bimonoids of fractions of Boolean-pointed Brouwerian algebras being semidistributive? We do not know.

  We can, however, use the algebra $\Hfive$ for another purpose, namely to construct an idempotent involutive residuated lattice which fails to satisfy the equation $x \equals (1 \wedge x) \cdot (0 \vee x)$. This equation is thus not valid in all commutative idempotent involutive residuated lattices, although it is true in all distributive ones.\footnote{We are grateful to Jos\'{e} Gil-F\'{e}rez for pointing out this example to us.}

  The universe of this idempotent involutive residuated lattice is $\{ \elembot, \elem{a}, \elem{b}, \elem{c}, \elemone \}$. Multiplication coincides with meets in $\Hfive$. Thus e.g.\ $\elem{a} \cdot \elem{b} = \elembot$. In particular, $\elemone$ is indeed the multiplicative unit. The lattice structure, however, is the one shown in Figure~\ref{fig: diamond}. Finally, taking $0 \assign \elemone$ yields the complementation $\comp{\elem{a}} = \elem{b}$, $\comp{\elem{b}} = \elem{a}$, $\comp{\elemone} = \elemone$, $\comp{\elembot} = \elem{c}$, $\comp{\elem{c}} = \elembot$. It is now routine to verify that this indeed yields a commutative idempotent involutive residuated lattice where $(\elemone \wedge \elem{a}) \cdot (\elemone \vee \elem{a}) = \elembot \neq \elem{a}$. Observe that the negative cone of this residuated lattice is precisely the two-element bimonoid $\elembot < \elemone$ with $x \cdot y = x \wedge y = x + y$. We therefore have two non-isomorphic commutative idempotent involutive residuated lattices with the same negative cone.

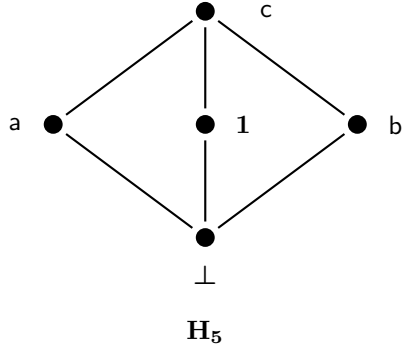
\begin{figure}
\caption{A commutative idempotent involutive residuated lattice which fails $x \equals (1 \wedge x) (0 \vee x)$}
\label{fig: diamond}
\begin{center}
\begin{tikzpicture}[scale=1,dot/.style={circle,fill,inner sep=2.5pt,outer sep=2.5pt}]
  \node (bot) at (0, 0) [dot] {};
  \node (a) at (-2, 1.5) [dot] {};
  \node (1) at (0, 1.5) [dot] {};
  \node (b) at (2, 1.5) [dot] {};
  \node (c) at (0, 3) [dot] {};
  \draw[-,thick] (bot) -- (a) -- (c);
  \draw[-,thick] (bot) -- (1) -- (c);
  \draw[-,thick] (bot) -- (b) -- (c);
  \node (nbot) at (0, -0.5) {$\elembot$};
  \node (na) at (-2.5, 1.5) {$\elem{a}$};
  \node (nb) at (2.5, 1.5) {$\elem{b}$};
  \node (nc) at (0.8, 3) {$\elem{c}$};
  \node (n1) at (0.5, 1.5) {$\elemone$};
  \node at (0, -1.25) {$\Hfive$};
\end{tikzpicture}
\end{center}
\end{figure}

  Finally, let us show that while being Brouwerian is a sufficient condition for a unital meet semilattice to have a complemented bimonoid of fractions, it is not necessary. Consider the semilattices $\twoalg \assign \{ 0, 1 \}$ and $\omegaplusonealg \assign \{ 0, 1, 2, \dots, \infty \}$ with the usual order. Let $\alg{T} \assign \twoalg \times (\omegaplusonealg)$ and let $\alg{S}$ be the subalgebra obtained by removing the element $\pair{0}{\infty}$. Then $\alg{T}$ is a Brouwerian semilattice but $\alg{S}$ is not: there is no largest $x \in \alg{S}$ such that $x \wedge \pair{1}{0} \leq \pair{0}{0}$. We show that $\alg{S}$ has a commutative complemented bimonoid of fractions. Note~that while $\alg{S}$ is not a Brouwerian semilattice, it is still distributive.

  By Proposition~\ref{prop: existence of bimonoid of fractions} it suffices to show that for each $a, b \in \alg{S}$ there are $x, y \in \alg{S}$ such that $a \comp{b} = x + \comp{y}$ in~$\cdmalg{S}$. By Fact~\ref{fact: semilattice transformation conditions} there are $x, y \in \alg{S}$ such that $a \cdot \comp{b} = x + \comp{y}$ if and only if there is $y \in \alg{S}$ such that $a y = a b$ and
\begin{align*}
  (\forall p q \in \alg{S}) \left( p y \leq a b ~ \& ~ a q \leq a b \implies p q \leq a b \right).
\end{align*}
  Moreover, $a \cdot \comp{b} = a \cdot \comp{a b}$, therefore we may assume without loss of generality that $b \leq a$. By Fact~\ref{fact: fractions for brsls} we know that in $\alg{T}$ we may take $y \assign a \rightarrow b$. It follows that for each $a, b \in \alg{S}$ such that $a \rightarrow b \in \alg{S}$ (the element $a \rightarrow b$ being computed in $\alg{T}$), there is $y \in \alg{S}$ such that $a \cdot \comp{b} = a + \comp{y}$. The only cases we have to consider are therefore $a, b \in \alg{S}$ with $b \leq a$ such that $a \rightarrow b = \pair{0}{\infty}$. The only such elements have the form $a = \pair{1}{i}$ and $b = \pair{0}{i}$ for some $i \in \omega$. But then $y \assign \pair{0}{i+1}$ does the job: $ay = ab$, and if $p \pair{0}{i+1} \leq \pair{0}{i}$ and $\pair{1}{i} q \leq \pair{0}{i}$, then $p \leq \pair{1}{i}$ and $q \leq \pair{0}{\infty}$, so $p q \leq \pair{0}{i} = a b$.

\section{Bimonoidal subreducts of positive universal classes of involutive residuated lattices}
\label{sec: subreducts}

  In this section, we study the bimonoidal and $\ell$-bimonoidal subreducts of positive universal classes of commutative residuated pomonoids and lattices. Throughout the section we use the abbreviations CRL and CRP for commutative residuated lattices and pomonoids. The existence of a commutative complemented DM completion (Theorem~\ref{thm: dm completions exist}) settles the question of what the bimonoidal and $\ell$-bimonoidal subreducts of involutive CRPs and CRLs are: they are precisely commutative bimonoids and $\ell$-bimonoids.\footnote{Recall that an algebra $\alg{A}$ is a \emph{reduct} of an algebra $\alg{B}$ if it is obtained from $\alg{B}$ by forgetting part of the signature of $\alg{B}$. It is a \emph{subreduct} of $\alg{B}$ if it embeds into a reduct of $\alg{B}$.} The same problem, however, arises more generally for other classes of involutive CRLs and CRPs.

  In this section, we provide an algorithm for axiomatizing the $\ell$-bimonoidal subreducts of each class of involutive CRLs axiomatized by s$\ell$-monoidal positive universal clauses, i.e.\ universally quantified finite disjunctions of equations in the signature $\{ \vee, \cdot, 1 \}$. The first step of this algorithm is to transform a given set of positive universal clauses into positive universal clauses of special form.

\begin{definition}[Linear inequalities]
  An s$\ell$-monoidal inequality is \emph{linear} if it has the form $t \inequals u$ where $t$ is a \emph{linear term}, i.e.\ a product of distinct variables, and $u$ is a join of products of variables. An s$\ell$-monoidal positive universal clause is \emph{linear} if it is a universally quantified finite disjunction of linear s$\ell$-monoidal inequalities.
\end{definition}

  The following lemma was essentially proved in~\cite[p.\ 1232]{galatos+jipsen13residuated-frames}.

\begin{lemma}[Linearization] \label{lemma: linearization}
  Each s$\ell$-monoidal equation (positive universal clause) is equivalent over s$\ell$-monoids to a set of linear inequalities (linear positive universal clauses).
\end{lemma}

  For example, the s$\ell$-monoidal inequality $x^2 \inequals x$ is equivalent to the linear inequality $x \cdot y \inequals x \vee y$. Recall now that admissible joins are joins over which multiplication distributes (Definition~\ref{def: admissible join}).

\begin{definition}[Preservation of inequalities]
  Let $\alg{A}$ be an s$\ell$-monoid. An s$\ell$-monoidal inequality \emph{holds in $X \subseteq \alg{A}$} if it holds in all valuations where variables take values in $X$. It is \emph{preserved under products} in~$\alg{A}$ if it holds in the monoid generated by $X$ whenever it holds in~$X$. It is \emph{preserved under admissible joins} in $\alg{A}$ if it holds in the set of all admissible joins of elements of $X$ whenever it holds in $X$.
\end{definition}

\begin{lemma}[Preservation of linear inequalities] \label{lemma: linear inequalities}
  Linear s$\ell$-monoidal inequalities are preserved under admissible joins. The inequality $x \inequals x^{n}$ is preserved under products in commutative pomonoids.
\end{lemma}

\begin{proof}
  The latter claim is immediate, since $xy \leq x^{n} y^{n} = (xy)^{n}$ if $x \leq x^{n}$ and $y \leq y^{n}$. As for the former, consider a linear s$\ell$-monoidal inequality $x_1 \cdot \ldots \cdot x_m \inequals u(x_1, \dots , x_m, y_1, \dots, y_n)$. Suppose that $x_{i} \assign \bigvee_{k \in I_{i}} x_{i,k}$ for $1 \leq i \leq m$ and $x_{i, k} \in X$, and likewise $y_{j} \assign \bigvee_{l \in J_{j}} y_{j,l}$ for $1 \leq j \leq n$ and $y_{j, l} \in X$, where all of these joins are admissible. Then $x_1 \cdot \ldots \cdot x_m = \bigvee \set{x_{1, k_{1}} \cdot \ldots \cdot x_{m, k_{m}}}{k_{1} \in I_{1}, \dots, k_{m} \in I_{m}}$. If the inequality in question holds in $X$, then $x_{1, k_{1}} \cdot \ldots \cdot x_{n, k_{n}} \leq u(x_{1, k_{1}}, \ldots, x_{m, k_{m}}, y_{1, l_{1}}, \dots, y_{n, l_{n}})$ for all $k_{i} \in I_{i}$ and $l_{j} \in J_{j}$. But $x_{i, k_{i}} \leq x_{i}$ and $y_{j, l_{j}} \leq y_{j}$, therefore $x_{1, k_{1}} \cdot \ldots \cdot x_{m, k_{m}} \leq u(x_1, \dots, x_m, y_1, \dots, y_n)$ for all $k_{i} \in I_{i}$, and $x_1 \cdot \ldots x_m \leq u(x_1, \dots, x_m, y_1, \dots, y_n)$.
\end{proof}

  We now show how to obtain an axiomatization of the $\ell$-bimonoidal subreducts of involutive CRLs which satisfy e.g.\ the inequality $x \cdot y \inequals x \vee y$ obtained by linearizing ${x^{2} \inequals x}$. By the previous lemma, this inequality holds in $\cdmalg{A}$ whenever it holds in a join dense subset of $\cdmalg{A}$, in particular whenever it holds for all $x$ and $y$ of the form $a \comp{b}$ for $a, b \in \alg{A}$. But by the meet density of elements of the form $e + \comp{f}$ for $e, f \in \alg{A}$, the inequality $a \comp{b} \cdot c \comp{d} \leq a \comp{b} \vee c \comp{d}$ is equivalent to the implication
\begin{align*}
  a \comp{b} \vee c \comp{d} \leq e + \comp{f} \implies a \comp{b} \cdot c \comp{d} \leq e + \comp{f},
\end{align*}
  or equivalently
\begin{align*}
  a \comp{b} \leq e + \comp{f} ~ \& ~ c \comp{d} \leq e + \comp{f} \implies ac \comp{d + b} \leq e + \comp{f},
\end{align*}
  where $a, b, c, d, e, f$ range over $\alg{A}$. Finally, this inequality may be expressed using the residuation laws as
\begin{align*}
  a \cdot f \leq b + e ~ \& ~ c \cdot f \leq d + e \implies a \cdot c \cdot f \leq b + d + e.
\end{align*}
  Thus $\cdmalg{A}$ satisfies $x^{2} \inequals x$ if and only if $\alg{A}$ satisfies the quasiequation above. In particular, each commutative $\ell$-bimonoid $\alg{A}$ which satisfies this quasi\-equation is a subreduct of an involutive CRL which satisfies $x^{2} \inequals x$.

  Conversely, if an involutive CRL satisfies the inequality $x \cdot y \inequals x \vee y$, then it satisfies the implication
\begin{align*}
  x \leq z ~ \& ~ y \leq z \implies x \cdot y \leq z.
\end{align*}
  In particular, it satisfies
\begin{align*}
  a \comp{b} \vee c \comp{d} \leq e + \comp{f} \implies a \comp{b} \cdot c \comp{d} \leq e + \comp{f},
\end{align*}
  and as before, this implication is equivalent to
\begin{align*}
  a \cdot f \leq b + e ~ \& ~ c \cdot f \leq d + e \implies a \cdot c \cdot f \leq b + d + e.
\end{align*}
  This completes the proof of the following fact.

\begin{fact}
  The following are equivalent for each commutative $\ell$-bimonoid $\alg{A}$:
\begin{enumerate}[(i)]
\item $\alg{A}$ is subreduct of an involutive CRL with $x^{2} \inequals x$,
\item $\cdmalg{A}$ satisfies $x^{2} \inequals x$,
\item $\alg{A}$ satisfies the bimonoidal quasiequation $a \cdot f \leq b + e ~ \& ~ c \cdot f \leq d + e \implies a \cdot c \cdot f \leq b + d + e$.
\end{enumerate}
\end{fact}

  The above reasoning works equally well for any set of s$\ell$-monoidal equations, or indeed any set of s$\ell$-monoidal positive universal clauses. However, although the procedure itself is straightforward, it would be somewhat cumbersome to try to describe the resulting $\ell$-bimonoidal quasiequation or universal clause in full generality. Instead of an explicit proof, let us therefore consider one more example. We axiomatize the bimonoidal subreducts of linear (totally ordered) involutive CRPs.

  Linearity is expressed by the positive universal clause
\begin{align*}
  x \leq y \text{ or } y \leq x.
\end{align*}
  This positive clause is linear (in a different sense of the word), therefore it holds in $\cdmalg{A}$ whenever it holds in a join dense subset of $\cdmalg{A}$, in particular whenever it holds for all $x$ and $y$ of the form $a \comp{b}$. By the meet density of elements of the form $e + \comp{f}$, the linearity of $\cdmalg{A}$ is equivalent to the following disjunction of universally quantified implications in $\alg{A}$:
\begin{align*}
  (\forall e, f) (c \comp{d} \leq e + \comp{f} \implies a \comp{b} \leq e + \comp{f}) \text{ or } (\forall e, f) (a \comp{b} \leq e + \comp{f} \implies c \comp{d} \leq e + \comp{f}).
\end{align*}
  Renaming the variables in order to transform the above condition into a universal clause yields the universally quantified sentence
\begin{align*}
  (c \comp{d} \leq e + \comp{f} \implies a \comp{b} \leq e + \comp{f}) \text{ or } (a \comp{b} \leq g + \comp{h} \implies c \comp{d} \leq g + \comp{h}),
\end{align*}
  or equivalently the universal clause
\begin{align*}
  a \comp{b} \leq g + \comp{h} ~ \& ~ c \comp{d} \leq e + \comp{f} \implies a \comp{b} \leq e + \comp{f} \text{ or } c \comp{d} \leq g + \comp{h}.
\end{align*}
  Applying the residuation laws now yields the bimonoidal universal clause
\begin{align*}
  a \cdot h \leq b + g ~ \& ~ c \cdot f \leq d + e \implies a \cdot f \leq b + e \text{ or } c \cdot h \leq d + g.
\end{align*}
  The algebra $\cdmalg{A}$ is therefore linear if and only if $\alg{A}$ satisfies the universal sentence above. Conversely, if an involutive CRP is linear, then it satisfies the implication
\begin{align*}
  (\forall z) (y \leq z \implies x \leq z) \text{ or } (\forall z) (x \leq z \implies y \leq z).
\end{align*}
  Taking $x = a \comp{b}$, $y = c \comp{d}$, $z = e + \comp{f}$, we may now repeat the above reasoning to show that our linear involutive CRP satisfies the desired bimonoidal universal clause. This completes the proof of the following fact.

\begin{fact}
  The following are equivalent for each commutative bimonoid $\alg{A}$:
\begin{enumerate}[(i)]
\item $\alg{A}$ is subreduct of a linear involutive CRP,
\item $\cdmalg{A}$ is linear,
\item $\alg{A}$ satisfies the universal clause $a \cdot h \leq b + g ~ \& ~ c \cdot f \leq d + e \implies a \cdot f \leq b + e \text{ or } c \cdot h \leq d + g$.
\end{enumerate}
\end{fact}

  It should be clear enough that the procedure outlined above applies in full generality to any set of pomonoidal or s$\ell$-monoidal \emph{linear} positive universal clauses. We therefore obtain the following theorems.

\begin{theorem}[Subreducts of involutive CRLs]
  Let $\alg{A}$ be a commutative $\ell$-bimonoid and $\Pi$ be a set of positive universal clauses in the signature $\{ \vee, \cdot, 1 \}$. Then the following are equivalent:
\begin{enumerate}[(i)]
\item $\cdmalg{A}$ satisfies $\Pi$.
\item $\alg{A}$ is a subreduct of an involutive CRL satisfying $\Pi$,
\item for each $\pi \in \Pi$, $\alg{A}$ is a subreduct of an involutive CRL satisfying $\pi$,
\end{enumerate}
\end{theorem}

\begin{theorem}[Subreducts of involutive CRPs]
  Let $\alg{A}$ be a commutative bimonoid and $\Pi$ be a set of linear positive universal clauses in the signature $\{ \cdot, 1 \}$. Then the following are equivalent:
\begin{enumerate}[(i)]
\item $\cdmalg{A}$ satisfies $\Pi$.
\item $\alg{A}$ is a subreduct of an involutive CRP satisfying $\Pi$,
\item for each $\pi \in \Pi$, $\alg{A}$ is a subreduct of an involutive CRP satisfying $\pi$,
\end{enumerate}
\end{theorem}

  The only difference between the two cases is that for pomonoidal clauses we must assume linearity, while if joins are available each positive universal clause may be linearized by Lemma~\ref{lemma: linearization}. For example, although $x^2 \inequals x$ is a perfectly good pomonoidal inequality, the above algorithm does not tell us how to axiomatize the bimonoidal subreducts of involutive residuated pomonoids which satisfy this inequality. 

  The reader can verify that the universal clause axiomatizing bimonoidal subreducts of linear involutive CRPs fails in the three-element {\L}ukasiewicz chain $\elemone > \elem{a} > \elem{b}$ considered in Subsection~\ref{subsec: macneille existence} (with $x + y = x \cdot y$). This is witnessed by the following valuation:
\begin{align*}
  a & \assign \elemone & c & \assign \elem{a} & e & \assign \elem{b} & g & \assign \elemone \\
  b & \assign \elemone & d & \assign \elem{b} & f & \assign \elem{a} & h & \assign \elemone
\end{align*}
  More generally, replacing $\elem{a}$ and $\elem{b}$ by elements $x$ and $y$ such that $x \nleq y$ but $x^{2} \leq y^{2}$, the same valuation shows that no pomonoid (with $x + y \assign x \cdot y$ and $0 \assign 1$) with such elements $x$ and $y$ satisfies the above universal (inequational) clause. In particular, this holds for each pomonoid with a bottom element $\bot$ and a non-trivial nilpotent element, i.e.\ an $a$ such that $a^{n} = \bot$ for some $n$.

  In certain cases, the quasiequations obtained by the above procedure can be simplified to inequalities. This is the case with linear inequalities which are preserved under products, in particular with $x \inequals x^{n}$.

\begin{theorem}[Subreducts of involutive CRLs with $x \inequals x^{n}$]
  The $\ell$-bimonoidal subreducts of involutive CRLs satisfying a set of in\-equalities of the form $x \inequals x^{n}$ form a variety of $\ell$-bimonoids axiomatized by the corresponding set of inequalities $x \inequals x^{n}$ and $nx \inequals x$.
\end{theorem}

\begin{proof}
  The inequality $nx \inequals x$ holds in each involutive CRL where $x \inequals x^{n}$ holds. Conversely, recall that the inequality $x \inequals x^{n}$ is preserved under both joins and products. It therefore holds in $\cdmalg{A}$ if and only if it holds for each $x$ of the form $a$ or $\comp{a}$ for $a \in \alg{A}$. But $\comp{a} \leq (\comp{a})^{n}$ if and only if $na \leq a$.
\end{proof}

  The same proof yields an analogous result for involutive CRPs.

\begin{theorem}[Subreducts of involutive CRPs with $x \inequals x^{n}$]
  The bimonoidal subreducts of involutive CRPs satisfying a set of in\-equalities of the forms $x \inequals x^{n}$ form a class of bimonoids axiomatized by the corresponding set of inequalities $x \inequals x^{n}$ and $nx \inequals x$.
\end{theorem}

\section{Open problems}
\label{sec: open problems}

  Let us end the paper with a list of unresolved questions which arose in the course of the paper. The~main task left open is to construct complemented DM~completions for non-commutative bimonoids.

\begin{problem}
  Can we embed an arbitrary (not necessarily commutative) bimonoid into a complemented one? In particular, can we find a non-commutative generalization of complemented DM completions?
\end{problem}

  We can also consider the same embedding problem for the categorical version of commutative bimonoids, so-called symmetric weakly distributive categories. The role of complemented commutative bimonoids is then played by so-called $\ast$-autonomous categories (see~\cite{barr79,cockett+seely97weakly-distributive-categories}).

\begin{problem}
  Does each (small) symmetric weakly distributive category embed into a $\ast$-autonomous category?
\end{problem}

  There is also space for other kinds of complemented envelopes intermediate between bimonoids of fractions and complemented DM completions. Bounded distributive lattices are an example: their most natural complemented envelopes are their free Boolean extensions, where each element is a \emph{finite} join of elements of the form $a \cdot \comp{b}$. One can also consider cases where each element has either the form $a \cdot \comp{b}$ or the form $a + \comp{b}$, as in the case of the algebra $\Lukthreeext$ in Subsection~\ref{subsec: macneille existence}.

\begin{problem}
  Investigate other kinds of complemented $\Delta_{1}$-extensions of com\-mutative bimonoids which are more general than bimonoids of fractions but more restrictive than complemented DM completions.
\end{problem}

  Within the variety of bounded distributive lattices, complemented DM completions can be characterized in purely categorical terms: they are precisely the injective hulls in this category. We saw in Subsection~\ref{subsec: macneille definition} that this does not hold in the category of commutative $\ell$-bimonoids. Nonetheless, it may be the case that this categorical characterization of complemented DM completions at least holds in some wider variety of commutative $\ell$-bimonoids than the variety of bounded distributive lattices.

\begin{problem}
  In the variety of bounded distributive lattices, complemented DM completions coincide with injective hulls. Does this extend to some larger variety of commutative $\ell$-bimonoids?
\end{problem}

  We saw in Section~\ref{sec: subreducts} that the class of $\ell$-bimonoidal subreducts of involutive commutative residuated lattices axiomatized by inequalities of the form ${x \inequals x^{n}}$ or $1 \inequals x^{n}$ is the variety of commutative $\ell$-bimonoids axiomatized by the inequalities $x \inequals x^{n}$ and $nx \inequals x$ or $1 \inequals x^{n}$ and $nx \inequals 1$. Are there other knotted varieties of involutive commutative residuated lattices, i.e.\ classes axiomatized by inequalities of the form $x^{m} \leq x^{n}$ for $m, n \geq 1$, whose $\ell$-bimonoidal subreducts form a variety? The same question can of course be asked for knotted partially ordered varieties of involutive commutative residuated pomonoids.

\begin{problem}
  Is it the case that the only knotted varieties of involutive commutative residuated lattices whose $\ell$-bimonoidal subreducts form a variety are axiomatized by inequalities of the form $x \inequals x^{n}$ or $1 \inequals x^{n}$?
\end{problem}

\section*{Acknowledgments}

  We are grateful to the anonymous referee for their careful reading of the paper and for spotting a number of typos.

\end{document}